\documentclass[11pt, oneside, reqno]{amsart}

\usepackage{amssymb,amsfonts}
\usepackage[letterpaper, margin=1in]{geometry}
\usepackage{parskip}
\usepackage[all,arc]{xy}
\usepackage{enumerate}
\usepackage{mathrsfs}
\usepackage{hyperref}
\usepackage{graphicx}
\usepackage{tikz}
\usepackage{tikz-cd}
\usetikzlibrary{calc}
\usepackage{amsrefs}
\graphicspath{ {./Images/} }
\makeatletter
\newcommand*{\rom}[1]{\expandafter\@slowromancap\romannumeral #1@}
\makeatother

\newtheorem{theorem}{Theorem}[section]
\newtheorem{corollary}[theorem]{Corollary}
\newtheorem{proposition}[theorem]{Proposition}
\newtheorem{lemma}[theorem]{Lemma}

\theoremstyle{definition}
\newtheorem{definition}[theorem]{Definition}

\theoremstyle{remark}
\newtheorem{remark}[theorem]{Remark}

\numberwithin{equation}{section}

\usepackage{tikz-cd}
\newcommand{\R}{\mathbb R}
\newcommand{\C}{\mathbb C}

\newcommand{\T}{\mathbb T}

\newcommand{\cF}{\mathcal F}
\newcommand{\cG}{\mathcal G}
\newcommand{\cK}{\mathcal K}

\newcommand{\cN}{\mathcal N}

\newcommand{\cS}{\mathcal S}

\newcommand{\re}{\mathrm {Re}}
\newcommand{\im}{\mathrm {Im} }
\newcommand{\tr}{\mathrm {Tr}}
\newcommand{\av}{\mathrm {Av}}
\newcommand{\pv}{\mathrm {p.v.\ }}

\newcommand{\lip}{\mathrm {Lip}}

\newcommand{\Div}{\mathrm {div\ }}

\newcommand{\bs}{\backslash }
\newcommand{\wstar}{\stackrel{\ast}{\rightharpoonup}}

\begin{document}
\author{Zhengjun Liang}
\address{Department of Mathematics, University of Michigan, 530 Church St, Ann Arbor, MI 48104}
\email{jspliang@umich.edu}
\title{Global Well-Posedness for the Hele-Shaw Problem with Point Injection}
\date{May 10, 2026}

\begin{abstract}
We study the two-dimensional Hele-Shaw problem with point injection for star-shaped domains. We reduce the system to a nonlocal parabolic equation on $\T$, and for arbitrary Lipschitz initial interface away from the source, we prove global well-posedness of the interface equation in a \(L^\infty_{\mathrm{loc}, t}L_\alpha^2\) sense. We also introduce a viscosity solution framework for the interface equation and relate it to the classical viscosity theory for the Hele-Shaw problem.
\end{abstract}

\maketitle

\section{Introduction}
In this paper, we revisit the classical one-phase Hele-Shaw problem with a
point injection. Physically, Hele-Shaw flow describes the motion of a viscous fluid confined between two closely spaced parallel plates, or equivalently, fluid flow through a homogeneous porous medium. In the injection-driven setting considered here, fluid is continuously supplied through a small source, idealized as a point source at the origin, and the occupied fluid region expands in time.

Mathematically, if \(\Omega(t)\) denotes the fluid region and
\(\Gamma(t)=\partial\Omega(t)\) denotes its free boundary, the problem is
formulated as
\begin{equation}\label{domain_eqn_u_intro}
\left\{
\begin{array}{ll}
u=-\nabla p, & \text{in } \Omega(t),\\[2pt]
\operatorname{div} u=2\pi\delta, & \text{in } \Omega(t),\\[2pt]
p=0, & \text{on } \Gamma(t),\\[2pt]
V(\Gamma(t))=u\cdot n, & \text{on } \Gamma(t).
\end{array}
\right.
\end{equation}
Here \(u\) is the fluid velocity, \(p\) is the pressure, \(n\) is the outward unit normal to the free boundary, and \(V(\Gamma(t))\) denotes the normal velocity of the moving interface. The physical scenario is illustrated in Figure \ref{fig:hele_shaw_injection}.

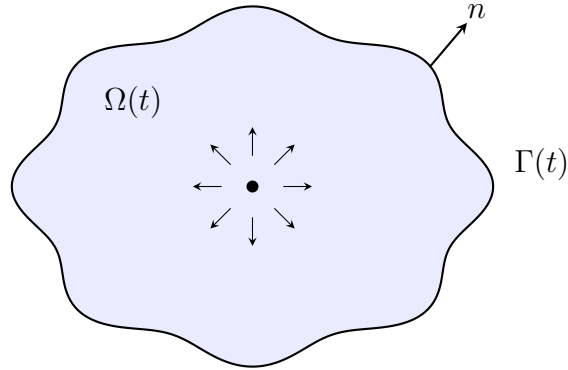
\begin{figure}[!htbp]
\centering
\begin{tikzpicture}[>=stealth, line cap=round, line join=round, scale=0.75]

    \path[fill=blue!8, draw=black, thick]
      plot[domain=0:360, samples=240, smooth, variable=\t]
        ({(1 + 0.06*cos(8*\t))*4.0*cos(\t)},
         {(1 + 0.06*cos(8*\t))*3.0*sin(\t)})
      -- cycle;

    \fill (0,0) circle (3pt);

    \foreach \a in {0,45,...,315}{
        \draw[->, thin]
            ({0.55*cos(\a)},{0.55*sin(\a)}) --
            ({1.05*cos(\a)},{1.05*sin(\a)});
    }

    \node[font=\large] at (-2.1, 1.5) {$\Omega(t)$};

    \node[font=\large] at (5.1, 0.4) {$\Gamma(t)$};

    \draw[->, thick] (3.13, 2.12) -- (3.78, 2.89);
    \node[font=\large] at (3.95, 3.10) {$n$};

\end{tikzpicture}
\caption{Hele-Shaw flow with point injection}
\label{fig:hele_shaw_injection}
\end{figure}

\subsection{Main results}

We study the global-in-time behavior of the free boundary \(\Gamma(t)\) for the Hele-Shaw problem with injection. We work directly with an equation on the moving interface: when
\(\Gamma(t)\) is Lipschitz and star-shaped with respect to the origin, we find that the system \eqref{domain_eqn_u_intro} admits a natural nonlocal parabolic formulation on the boundary. More precisely, writing
\[
\Gamma(t):=\{h(\alpha,t)e^{i\alpha}:\alpha\in\mathbb T\},
\qquad
h(\alpha,t)=e^{\eta(\alpha,t)},
\]
with \(\eta(\cdot,t)\) Lipschitz, \eqref{domain_eqn_u_intro} can be equivalently formulated as
\begin{equation}\label{interface_eqn_intro}
\partial_t\eta+e^{-2\eta}\bigl(G(h)\eta-1\bigr)=0,
\end{equation}
where \(G(h)\eta\) is the Dirichlet-to-Neumann operator on star-shaped domains
as in Definition~\ref{def:DN_star_shaped}.

Our main result constructs global in-time solutions to \eqref{interface_eqn_intro}
for arbitrary Lipschitz initial data. To the best of the author's knowledge, this is the first large-data global result for the injection-driven Hele-Shaw problem that gives global quantitative control of the Lipschitz geometry of the free boundary without any smallness assumption on the initial interface.

\begin{theorem}[Main theorem]\label{thm:main_1}
For every \(\eta_0\in W^{1,\infty}(\mathbb T)\), there exists
\[
\eta\in C(\mathbb T\times [0,\infty))
\cap L_{\mathrm{loc}}^\infty([0,\infty);W^{1,\infty}(\mathbb T)),
\qquad
\partial_t\eta\in L_{\mathrm{loc}}^\infty([0,\infty);L^2(\mathbb T)),
\]
such that \(\eta(\cdot,0)=\eta_0\) and the interface
equation
\begin{equation}
\partial_t\eta+e^{-2\eta}\bigl(G(e^\eta)\eta-1\bigr)=0
\end{equation}
holds in the \(L_{\mathrm{loc}}^\infty([0,\infty);L^2(\mathbb T))\) sense. In addition, \(\eta\) satisfies the maximum principle
\begin{equation}\label{uniform_angle_intro}
\|\partial_\alpha\eta(\cdot,t)\|_{L^\infty}
\le
\|\partial_\alpha\eta_0\|_{L^\infty},
\end{equation}
which implies that the free-boundary normal remains in the same radial
cone determined by the initial data; see Figure~\ref{fig:cone_condition}. Finally, there exist constants
\(0<C_1<C_2\) such that the boundary height \(e^{\eta}\) satisfies
\begin{equation}
\sqrt{C_1+2t}\le e^\eta(\cdot,t)\le \sqrt{C_2+2t}.
\end{equation}
\end{theorem}

We next address uniqueness and finer interface dynamics. We introduce in Definition~\ref{def:visc_sol_bdry} a notion of viscosity solution for the interface equation \eqref{interface_eqn_intro}, following ideas from Dong--Gancedo--Nguyen~\cite{Dong-Gancedo-Nguyen-23} and
Chang-Lara--Guillen--Schwab~\cite{ChangLaraGuillenSchwab2019}. We show that the solution constructed in Theorem~\ref{thm:main_1} is the unique viscosity solution in this sense. We then relate this boundary formulation to the classical viscosity theory for the Hele-Shaw problem \cites{Kim2003, Kim2006Reg, Kim2006Long, JerisonKim2005, ChoiKim2006,ChoiJerisonKim2007, ChoiJerisonKim2009} and obtain the following proposition.

\begin{proposition}\label{prop:visc_sols_relations}
Suppose $\eta$ is a solution as in
Theorem~\ref{thm:main_1}. Let $\phi$ be the
unique harmonic extension of $\eta$ in $\Omega_\eta(t):=\{re^{i\alpha}: 0\leq r < e^{\eta(\alpha,t)},\
\alpha\in\T\}$. Define $p(x,t):=\max\{\phi(x,t)-\log|x|,0\}$, then $p$ is a viscosity solution as in Definition~\ref{def:visc_sol_domain}.
\end{proposition}

As applications of Proposition \ref{prop:visc_sols_relations}, we revisit classical results on free boundary regularity \cites{ChoiKim2006, JerisonKim2005, ChoiJerisonKim2007}. Further discussions can be found in section \ref{sec:visc_sol_domain}.

\subsection{Relationship with previous work}

Regarding injection-driven Hele-Shaw problems, Choi--Jerison--Kim~\cite{ChoiJerisonKim2007} was the first to construct global viscosity solutions starting from star-shaped Lipschitz initial domains. They proved that the positive phase remains star-shaped and Lipschitz in space for all time, and they also obtained local regularity estimates for the free boundary. Up to minor differences in the setup, we can show that their solution coincides with ours, although our construction is very different and yields the global uniform angle estimate \eqref{uniform_angle_intro}. We defer a more detailed discussion to Section~\ref{sec:visc_sol_domain}.

Our approach stems from the belief that the injection-driven Hele-Shaw problem is closely related to the one-phase Muskat problem. For the one-phase Muskat problem in graph domains, Alazard--Meunier--Smets~\cite{AlazardMeunierSmets2020} and Nguyen--Pausader~\cite{nguyen-pausader} independently introduced the interface formulation
\begin{equation}\label{eqn:one-phase_Muskat}
\partial_t f + G(f)f = 0.
\end{equation}
A major advance, which is also the main inspiration of our present work, was made by Dong--Gancedo--Nguyen~\cite{Dong-Gancedo-Nguyen-23}. They established global well-posedness for \eqref{eqn:one-phase_Muskat} with arbitrary Lipschitz initial data. One of the main ideas in \cite{Dong-Gancedo-Nguyen-23} is to use the translation invariance of \eqref{eqn:one-phase_Muskat} and pointwise properties of $G(f)f$ to prove a comparison principle and hence obtain uniform control of the Lipschitz norm \(\|f(\cdot,t)\|_{W^{1,\infty}}\).

\subsection{New challenges and main ideas} For the injection-driven problem, our key structural finding is to use the logarithmic variable $\eta$ to obtain the interface equation \eqref{interface_eqn_intro}. However, the nonlinear expression
\[
    e^{-2\eta}\bigl(G(h)\eta-1\bigr)
\]
is no longer translation invariant. To close the comparison argument, we use the Taylor sign condition
\[
    G(h)\eta-1\leq 0
\]
to recover a weaker symmetry sufficient for our purposes. More precisely, we show that if \(\eta\) is a solution, then for every \(C\geq 0\), \(\eta-C\) and \(\eta+C\) are respectively sub- and supersolutions. Subsequently, we establish a comparison principle and then the maximum principle for \(\|\partial_\alpha\eta(\cdot,t)\|_{L^\infty}\). Throughout the paper, we keep careful track of signs, especially in the analysis of the interface viscosity solution introduced in Section~\ref{sec:visc_sol}.

We also establish in Appendix~A an equivalence between the Dirichlet-to-Neumann operator on star-shaped domains and the Dirichlet-to-Neumann operator on graph domains under an exponential change of variables. This allows us to transfer graph-domain estimates for Dirichlet-to-Neumann operators to the star-shaped setting used here. The result is of independent interest and may offer a useful perspective on the study of Dirichlet-to-Neumann operators. 

Beyond the specific results proved here, the paper provides a bridge between the classical viscosity-solution approach to Hele-Shaw problems and the modern interface-equation approach developed for Muskat-type problems. This gives a common language for comparing the two frameworks and suggests that interface-equation methods can be useful beyond the graph setting.

\subsection{Literature review}\label{subsec:literature_review}
The mathematical literature on Hele-Shaw and Muskat problems is vast, and we do not attempt a comprehensive survey here. Instead, we review only the directions most related to the present work. We also fix some terminology used throughout the paper. By the \emph{Hele-Shaw problem}, or \emph{horizontal Hele-Shaw problem}, we mean a free-boundary problem posed in a bounded domain, driven by injection, suction, surface tension, or a combination thereof. In contrast, the \emph{Muskat problem}, either one-phase or two-phase, will refer to the gravity-driven porous-medium problem, with ot without surface tension.

The Hele-Shaw model originates from Hele-Shaw's thin-gap experiment
\cite{HeleShaw1898}; in the depth-averaged limit, the velocity is governed by a Darcy law, whose porous-media origin goes back to Darcy \cite{Darcy1856}.  The closely related Muskat problem was introduced by Muskat in the study of fluid motion in porous media \cite{Muskat1937}.  In Muskat's originally proposed two-phase form, the problem describes the evolution of the interface separating two immiscible fluids in a porous medium or Hele-Shaw cell, with the motion driven by pressure and by jumps in density or viscosity across the interface. The unstable displacement of a more viscous fluid by a less viscous one leads to the Saffman--Taylor fingering instability \cite{SaffmanTaylor1958}, and the planar zero-surface-tension Hele-Shaw problem is also connected to conformal mapping, quadrature domains, and Laplacian growth; see, for example, the monograph of
Gustafsson and Vasiliev \cite{GustafssonVasiliev2006}.

There is a long tradition of studying source-driven Hele-Shaw flow by
complex-analytic, variational, and free-boundary methods.  Richardson's work
\cite{Richardson1972} is a classical reference for injection-driven planar
flows and moment identities.  Explicit and complex-variable methods have also
played an important role in understanding singular free-boundary geometries.
For example, Howison~\cite{Howison1986} studied cusp development in
zero-surface-tension Hele-Shaw flow, and Howison--King~\cite{HowisonKing1989}
constructed explicit solutions for several free-boundary problems in fluid
flow and diffusion.  Particularly relevant to the angle dynamics discussed
below is the work of King--Lacey--V\'azquez~\cite{KingLaceyVazquez1995},
who used self-similar special solutions to describe the evolution of corners:
in the injection case, obtuse corners move and smooth out instantaneously,
whereas acute corners persist for a finite waiting time. Local and global
well-posedness near special geometries has also been studied in several
settings.  Constantin and Pugh \cite{ConstantinPugh1993} proved global
existence for small perturbations in a conformal-mapping formulation, while
Duchon and Robert \cite{DuchonRobert1984} and Escher and Simonett
\cites{EscherSimonett1997SurfaceTension,EscherSimonett1997Multidimensional}
developed classical solution theories for Hele-Shaw models, especially in the
presence of surface tension.  Chen \cite{Chen1993} related the Hele-Shaw
problem to area-preserving curve-shortening motions.  More recently, Cheng,
Coutand, and Shkoller \cite{ChengCoutandShkoller2014} studied global existence
and decay for Hele-Shaw flows with injection and surface tension near
expanding spheres, distinguishing between different injection regimes. For the
capillarity-driven Hele-Shaw problem, Matioc and Walker
\cite{MatiocWalker2025Capillary} recently used potential theory to obtain
local well-posedness and parabolic smoothing in nearly optimal function
spaces, as well as stability of stationary solutions.

Another important branch of the Hele-Shaw theory uses weak formulations,
variational inequalities, and obstacle-problem methods.  The key device is the Baiocchi--Duvaut transform, which replaces the moving pressure by its time integral. After extending the pressure by zero outside the fluid region, one defines
\begin{equation}\label{eq:Baiocchi_transform_intro}
U(t,x):=\int_0^t p(s,x)\,ds.
\end{equation}
Then the time variable becomes a parameter in an elliptic problem.  For example, in the free-falling coordinates used for the graph Hele-Shaw problem in \cite{AlazardKoch2025HeleShaw}, the positivity sets are increasing and \(\Omega(t)=\{U(t)>0\}\).  Formally, if
\(\rho(t)=\chi_{\Omega(t)}\), then the kinematic condition gives
\begin{equation}\label{eq:Baiocchi_conservation_intro}
\partial_t\rho-\Delta p=0,
\end{equation}
and hence
\begin{equation}\label{eq:Baiocchi_laplacian_intro}
\Delta U(t)=\rho(t)-\rho(0)
=\chi_{\Omega(t)}-\chi_{\Omega(0)}.
\end{equation}
Equivalently, in that normalization,
\begin{equation}\label{eq:Baiocchi_obstacle_intro}
\Delta U(t)=0 \quad\text{in }\Omega(0)\cup(\mathbb R^N\setminus\Omega(t)),
\qquad
\Delta U(t)=1 \quad\text{in }\Omega(t)\setminus\Omega(0),
\end{equation}
together with the free-boundary condition
\begin{equation}\label{eq:Baiocchi_gradient_intro}
\nabla U(t)=0\quad\text{on }\partial\Omega(t).
\end{equation}
This is the mechanism by which the Hele-Shaw evolution is converted into a
family of elliptic obstacle-type problems. This perspective goes back to
Baiocchi \cite{Baiocchi1972} and underlies the variational approach of Elliott and Janovsk\'y \cite{ElliottJanovsky1981}; see also Friedman's monograph \cite{Friedman1982}. It connects Hele-Shaw flow to the extensive regularity theory of obstacle-type free-boundary problems, beginning with Caffarelli's work on free-boundary regularity \cites{Caffarelli1976,Caffarelli1977} and the later geometric theory summarized in \cite{CaffarelliSalsa2005}.  

Viscosity solutions provide another robust framework for Hele-Shaw flows
beyond classical regularity. An extensive review of
classic literature can be found in section \ref{sec:visc_sol_domain}. More
recent work \cite{KimZhang2024SourceDrift} of Kim and Zhang treats Hele-Shaw
flows with source and drift terms, and Zhang's recent paper
\cite{Zhang2026C1} proves a local \(C^1\) regularity theorem for Hele-Shaw
flow with source and drift under a small-Lipschitz free-boundary assumption.
Moreover, Chang-Lara, Guillen, and Schwab
\cite{ChangLaraGuillenSchwab2019} considered a class of free-boundary problems
by recasting them as nonlocal parabolic equations on the moving boundary.
Their discussion of viscosity solutions directly with these non-local
parabolic equations is an important inspiration to our section
\ref{sec:visc_sol_domain}.

There is also a plethora of literature on the Muskat problem. The work of Siegel, Caflisch, and Howison \cite{SiegelCaflischHowison2004} is a landmark in the study of global existence, singular solutions, and ill-posedness of the Muskat problem. Modern interface formulations of the one-phase and two-phase Muskat problem were developed in a particularly influential form by C\'ordoba, C\'ordoba, and Gancedo \cite{CordobaCordobaGancedo2011}, who studied contour dynamics equations for interfaces governed by Darcy's law. In the two-phase problem, the interface separates two fluids with different physical parameters, and the sign of the Rayleigh--Taylor condition determines whether the evolution is parabolic or unstable. Castro, C\'ordoba, Fefferman, Gancedo, and
L\'opez-Fern\'andez \cite{CastroCordobaFeffermanGancedo2012} proved
Rayleigh--Taylor breakdown for the Muskat problem, and Castro, C\'ordoba,
Fefferman, and Gancedo \cite{CastroCordobaFeffermanGancedo2013} constructed
finite-time breakdown of smoothness. In the stable regime, local well-posedness was studied by, for example, Alazard and Lazar \cite{AlazardLazar2020}, and global results were obtained by Constantin, C\'ordoba, Gancedo, and Strain
\cite{ConstantinCordobaGancedoStrain2013} and by Constantin, C\'ordoba,
Gancedo, Rodr\'iguez-Piazza, and Strain
\cite{ConstantinCordobaGancedoRodriguezPiazzaStrain2016}. Further global
regularity results under finite-slope assumptions were proved by Constantin,
Gancedo, Shvydkoy, and Vicol \cite{ConstantinGancedoShvydkoyVicol2017} and
Cameron \cite{Cameron2019}. There is also a substantial low-regularity and critical-space theory: see, for
instance, the critical Sobolev global results of C\'ordoba--Lazar
\cite{CordobaLazar2021MuskatH32} and Gancedo--Lazar
\cite{GancedoLazar2022Muskat3D}, the critical Cauchy theories of
Alazard--Nguyen
\cites{AlazardNguyen2021Critical,AlazardNguyen2022Muskat3D,AlazardNguyen2023Endpoint}, the scaling-critical Besov result of Nguyen \cite{Nguyen2022BesovMuskat}, the \(C^1\) theory of Chen--Nguyen--Xu \cite{ChenNguyenXu2022C1Muskat}, and
the global large-Lipschitz theory with surface tension of Lazar
\cite{Lazar2024SurfaceTensionMuskat}. For the viscosity-jump problem, Gancedo, Garc\'ia-Ju\'arez, Patel, and Strain
\cite{GancedoGarciaJuarezPatelStrain2019} obtained global-in-time results in
critical spaces for stable data. For formulation issues and equivalence between different versions of the two-dimensional Muskat problem, see Matioc
\cite{Matioc2019}.

In the one-phase graph setting, the Muskat problem is mathematically equivalent to a vertical Hele-Shaw problem driven by gravity.  The free boundary is often encoded by a graph function, and the evolution takes the form of a nonlocal parabolic equation involving a graph Dirichlet-to-Neumann operator.  Local well-posedness in rough Sobolev spaces has been studied by Nguyen and Pausader \cite{nguyen-pausader} and others; these works develop many paradifferential and Dirichlet-to-Neumann methods that are now standard in the analysis of Muskat-type equations. For a recent survey of paradifferential methods for water waves, Hele-Shaw, and Muskat problems, see Alazard \cite{Alazard2024Paralinearization} and the references therein. On global well-posedness, Dong, Gancedo, and
Nguyen \cite{Dong-Gancedo-Nguyen-23} proved a large-data global well-posedness result for the one-phase Muskat problem with arbitrary periodic Lipschitz graph initial data, using comparison principles, layer-potential estimates, and a viscosity-solution uniqueness argument. This work is also a major inspiration of our present work. Their three-dimensional analogue was later developed in \cite{Dong-Gancedo-Nguyen-23-2}. In addition, Schwab, Tu, and Turanova \cite{SchwabTuTuranova2024} recently obtained well-posedness for viscosity solutions of the one-phase Muskat problem in all dimensions with bounded uniformly continuous graph data. In addition, Alazard and Koch \cite{AlazardKoch2025HeleShaw} developed a modern semigroup theory for the one-phase Muskat problem, emphasizing comparison, Lyapunov structure, and an obstacle-problem formulation. Surface-tension variants of the one-phase Muskat problem have also been studied extensively; see, for example, Nguyen \cite{nguyen20}, Flinn--Nguyen \cite{flynn-nguyen}, and the more recent result by Dong--Kwon \cite{DongKwon2026SurfaceTensionMuskat}.

\subsection{Outline of the paper}\label{subsec:outline_paper}

The paper is organized as follows.
\begin{enumerate}
    \item In section \ref{sec:notations}, we introduce basic notation and conventions; additional notations are introduced throughout the paper as needed.

    \item In section \ref{sec:eqn_eulerian}, we present the injection-driven Hele-Shaw equations in their original Eulerian free-boundary form. We also discuss the physical background of the system in more detail.

    \item In sections \ref{sec:formulation_DN} and \ref{sec:interface_eqn}, we reformulate the injection-driven Hele-Shaw equations as a nonlocal parabolic equation on the interface. In the process, we introduce the Dirichlet-to-Neumann operator and its layer-potential formulation. We also prove several important properties of the Dirichlet-to-Neumann operator on star-shaped domains, as well as basic symmetries of the interface equation.

    \item In section \ref{sec:visc_reg}, we perform a viscosity regularization of the interface equation and prove global well-posedness for the regularized equations. The main ingredients include comparison principles established directly at the level of the interface equation and Calder\'on--Zygmund estimates for the Dirichlet-to-Neumann operator in its layer-potential formulation.

    \item In section \ref{sec:visc_sol}, we discuss the notion of viscosity solution for the interface equation. We establish a comparison principle for such solutions at low regularity and prove that viscosity solutions become classical solutions of the interface equation whenever sufficient regularity is available.

    \item In section \ref{sec:global_wp}, we use a vanishing-viscosity approach to establish global well-posedness of the interface equation. In particular, we prove uniqueness by showing that the solution obtained through the limiting procedure is a viscosity solution in the sense of Section \ref{sec:visc_sol}, which is unique by the comparison principle.

    \item In section \ref{sec:visc_sol_domain}, we connect the viscosity solutions introduced in Section \ref{sec:visc_sol} with the viscosity solutions used in the classical Hele-Shaw literature. As an application, we obtain results on the angle dynamics of the initial interface.
\end{enumerate}

\subsection*{Acknowledgements}
The author is deeply grateful to their doctoral advisor Sijue Wu for many valuable discussions, guidance, and generous support. The author also thanks Hongjie Dong, Francisco Gancedo, and Inwon Kim for their warm enouragements, helpful suggestions, and corrections.

\section{Notations and Conventions}\label{sec:notations}
Throughout the paper, we identify $\mathbb T:=\mathbb R/2\pi\mathbb Z$ with the interval \([0,2\pi)\), with endpoints identified. Unless otherwise specified, all functions on \(\mathbb T\) are understood to be \(2\pi\)-periodic. We also define the periodic distance
\[|\alpha - \beta|_\T:=\min_{k\in\mathbb Z}|\alpha-\beta+2\pi k|\in[0,\pi].\]
We often identify $\R^2$ with $\C$, but $z\cdot w$ always denotes dot product of $z$ and $w$ in $\R^2$.

For a domain \(\Omega\subset\mathbb R^d\), we use the standard Sobolev spaces
\(W^{k,p}(\Omega)\) using norm
\[
\|f\|_{W^{k,p}(\Omega)}
:=
\sum_{|\beta|\le k}\|D^\beta f\|_{L^p(\Omega)}
\]
with the usual modification when \(p=\infty\). We write $H^k(\Omega):=W^{k,2}(\Omega)$. The Sobolev space \(H^s(\mathbb T)\) can also be defined by Fourier series. We write
\[
\widehat f(k):=\frac{1}{2\pi}\int_{\mathbb T} f(\alpha)e^{-2\pi i k}\,d\alpha
\]
and define
\[
\|f\|_{H^s(\mathbb T)}^2
:=
\sum_{k\in\mathbb Z}\langle k\rangle^{2s}|\widehat f(k)|^2,
\qquad
\langle k\rangle:=(1+|k|^2)^{1/2}.
\]
We also use the homogeneous seminorm
\[
\|f\|_{\dot H^s(\mathbb T)}^2
:=
\sum_{k\in\mathbb Z\setminus\{0\}} |k|^{2s}|\widehat f(k)|^2.
\]
For \(0<s<1\), the \(H^s(\mathbb T)\) norm is equivalently characterized by finite differences:
\[
\|f\|_{H^s(\mathbb T)}^2
\simeq
\|f\|_{L^2(\mathbb T)}^2
+
\int_{\mathbb T}\int_{\mathbb T}
\frac{|f(\alpha)-f(\beta)|^2}
{|2\sin\frac{\alpha-\beta}{2}|^{1+2s}}
\,d\alpha\,d\beta .
\]
More generally, if \(s=m+\sigma\), where $m$ is a non-negative integer and \(0<\sigma<1\), then
\[
\|f\|_{H^s(\mathbb T)}^2
\simeq
\sum_{j=0}^m\|\partial_\alpha^j f\|_{L^2(\mathbb T)}^2
+
\int_{\mathbb T}\int_{\mathbb T}
\frac{|\partial_\alpha^m f(\alpha)-\partial_\alpha^m f(\beta)|^2}
{|2\sin\frac{\alpha-\beta}{2}|^{1+2\sigma}}
\,d\alpha\,d\beta .
\]
Here and below, \(A\simeq B\) means that \(A\lesssim B\) and \(B\lesssim A\),
with constants depending only on fixed parameters such as \(s\).

We write \(\operatorname{Lip}(\Omega)\) for the space of bounded Lipschitz
functions on \(\Omega\). Its seminorm and norm are
\[
[f]_{\operatorname{Lip}(\Omega)}
:=
\sup_{\substack{x,y\in\Omega\\ x\neq y}}
\frac{|f(x)-f(y)|}{|x-y|},
\qquad
\|f\|_{\operatorname{Lip}(\Omega)}
:=
\|f\|_{L^\infty(\Omega)}+[f]_{\operatorname{Lip}(\Omega)}.
\]
When \(\Omega=\mathbb T\), the distance in the definition of the Lipschitz
seminorm is understood as the periodic distance \(|\cdot|_{\mathbb T}\). We also commonly have the identification $\operatorname{Lip}(\mathbb T)=W^{1,\infty}(\mathbb T)$. For \(k\) a non-negative integer and \(0<\alpha\le1\), we write
\(C^{k,\alpha}(\Omega)\) for the usual Hölder space. Its norm is
\[
\|f\|_{C^{k,\alpha}(\Omega)}
:=
\sum_{|\beta|\le k}\|D^\beta f\|_{L^\infty(\Omega)}
+
\sum_{|\beta|=k}
[D^\beta f]_{C^\alpha(\Omega)},
\]
where the H\"older semi-norm is
\[
[g]_{C^\alpha(\Omega)}
:=
\sup_{\substack{x,y\in\Omega\\x\neq y}}
\frac{|g(x)-g(y)|}{|x-y|^\alpha}.
\]
On \(\mathbb T\), the Hölder seminorm is again computed using the periodic
distance. The notation \(C^{k,1}\) refers to functions whose \(k\)-th
derivatives are Lipschitz. We will also use pointwise \(C^{1,1}\) regularity. A function \(f\) is said to be \(C^{1,1}\) at \(x_0\) if there exist \(v\in\mathbb R^d\), \(M>0\), and
\(r>0\) such that
\[
|f(x)-f(x_0)-v\cdot(x-x_0)|
\le
M|x-x_0|^2
\qquad
\text{for } |x-x_0|<r.
\]
In one space dimension, this means that \(f\) admits a tangent line at \(x_0\)
with quadratic error.

For a Banach space \(B\) and \(1\leq p\leq \infty\), we use
\(L^p([0,T];B)\) to denote the usual Bochner space of strongly measurable
functions \(f:[0,T]\to B\) such that
\[
    \|f\|_{L^p([0,T];B)}
    :=
    \left(\int_0^T \|f(t)\|_B^p\,dt\right)^{1/p}
    <\infty,
    \qquad 1\leq p<\infty,
\]
with the standard modification for $p = \infty$.

\section{Equations in Eulerian Coordinates}\label{sec:eqn_eulerian}
Throughout the paper, we work on fluids in $\R^2$ unless explicitly mentioned otherwise. The fundamental equation underlying the Hele-Shaw problem is the Darcy's law:
\begin{equation}\label{eqn:Darcy}
\frac{\mu}{\kappa}u = -\nabla p.
\end{equation}
Here $\mu$ is viscosity, $\kappa$ is permeability, $u$ is fluid velocity, and $p$ represents scalar pressure. Darcy's law is used to describe either fluid flow through porous media or fluid confined between two parallel plates with a very thin gap in between. 

The Hele-Shaw equations consider the Darcy's law \eqref{eqn:Darcy} on a moving fluid domain $\Omega(t)$, where $t\geq 0$ is the time parameter. Denoting our free interface by $\Gamma(t)$, we have the dynamic boundary condition or Young-Laplace equation
\begin{equation}\label{eqn:dynamic}
p(x,t) = \sigma H(\Gamma(t)) \quad \text{on }\Gamma(t)
\end{equation}
which describes the jump of pressure across the interface. Here $\sigma$ is the surface tension coefficient $\geq 0$, and $H(\Gamma(t))$ is the mean curvature of the interface. In this article we consider the \emph{zero surface tension} case $\sigma = 0$. We also have the kinematic boundary condition
\begin{equation}\label{eqn:kinematic}
V(\Gamma(t)) = u\cdot n\quad \text{on } \Gamma(t)
\end{equation}
where $V(\Gamma(t))$ is the pointwise speed of the boundary $\Gamma(t)$ and $n$ is the unit outward normal on $\Gamma(t)$. Physically, the kinematic boundary condition \eqref{eqn:kinematic} says that particles on the boundary always stay on the boundary, so there is no invasion of air. Since we consider the zero-surface-tension case, if we further assume that the fluid is incompressible, then one can show that $\Omega(t)$ is actually \emph{static}. Therefore, some external source is needed. In this paper, we consider the case where the incompressible fluid has a constant point injection at the origin, namely
\begin{equation}
\Div u = 2\pi I\delta\quad \text{on } \Omega(t).
\end{equation}
Here the parameter $I$ is a constant rate of injection. Since we are not studying any limiting behavior of the parameters, we normalize $\mu =\kappa = I\equiv 1$. In summary, we obtain the system
\begin{equation}\label{domain_eqn_u}
\left\{
\begin{array}{ll}
u(\cdot,t) = -\nabla p(\cdot,t), & \text{in }\Omega(t)\\
\Div u(\cdot,t) = 2\pi \delta, & \text{in }\Omega(t)\\
p(\cdot, t) = 0, & \text{on }\Gamma(t)\\
V(\Gamma(t)) = u\cdot n, & \text{on }\Gamma(t)
\end{array}
\right.
\end{equation}
Equivalently, the equations can be formulated in terms of pressure. We have
\begin{equation}\label{domain_eqn_p}
\left\{
\begin{array}{ll}
      -\Delta p = 2\pi \delta & \mathrm{in}\ \Omega(t)\\
      p = 0  & \mathrm{on}\ \Gamma(t) \\
      V(\Gamma(t)) = \frac{\partial p}{\partial n} & \text{on }\Gamma(t)
\end{array}
\right.
\end{equation}

\section{The Dirichlet-to-Neumann Operator and Layer Potentials}\label{sec:formulation_DN}
In order to formulate the domain equations \eqref{domain_eqn_u} or \eqref{domain_eqn_p} as an equation on the interface, we introduce some preliminaries on the Dirichlet-to-Neumann operator and layer potentials. An extensive discussion in the context of periodic graph domain can be found in \cite{Dong-Gancedo-Nguyen-23}. 

\subsection{Dirichlet-to-Neumann Operator} We begin by recalling the classical definition of the Dirichlet-to-Neumann operator in the smooth setting.
\begin{definition}\label{def:DN_general_formal}
Suppose that \(\Omega\) and the boundary data \(f\) are sufficiently
regular. Let $\phi$ solve the Dirichlet boundary value problem
\begin{equation}\label{dirichlet_bv}
\left\{
\begin{array}{ll}
 \Delta\phi = 0 & \mathrm{in}\ \Omega  \\
 \phi = f & \mathrm{on}\ \Sigma
\end{array}
\right.
\end{equation}
and let $n$ be the outward unit normal vector field on $\Sigma$. We define the \emph{(normalized) Dirichlet-to-Neumann operator}
\begin{equation}
\cG(\Sigma)(f)(x,y) := n(x,y)\cdot\nabla\phi(x,y)\quad \mathrm{for}\ (x,y)\ \mathrm{on}\ \Sigma.
\end{equation}
\end{definition}

We have the following standard result which gives the weak harmonic
extension on bounded Lipschitz domains and hence the weak formulation of
the Dirichlet-to-Neumann operator. See, for instance, \cite[Chapters 3--4]{McLean2000} or
\cite{Grisvard1985}. Analogous results have also been proven in unbounded graph domains of the form $\{(x', x_d): x_d < f(x)\}$. See, for example, Proposition 3.6 of \cite{nguyen-pausader} and Proposition 2.6 of \cite{Dong-Gancedo-Nguyen-23}.

\begin{proposition}\label{prop:DN_well_defined}
Let $d\geq 2$, and $\Omega\subset\mathbb R^d$ be a bounded Lipschitz domain. Define
\[
\dot H^1(\Omega)
:=
\{u\in L^1_{\mathrm{loc}}(\Omega):\nabla u\in L^2(\Omega)\}/\mathbb R,
\qquad
\|u\|_{\dot H^1(\Omega)}:=\|\nabla u\|_{L^2(\Omega)}.
\]
Let $\dot H^{\frac{1}{2}}(\partial\Omega)$ denote the trace space of
$\dot H^1(\Omega)$, equipped with the quotient norm
\[
\|g\|_{\dot H^{\frac{1}{2}}(\partial\Omega)}
:=
\inf\bigl\{\|\nabla U\|_{L^2(\Omega)}:
U\in \dot H^1(\Omega),\ 
\operatorname{Tr}U=g\bigr\}.
\]
Then for every $g\in \dot H^{\frac{1}{2}}(\partial\Omega)$ there exists a unique $\phi\in \dot H^1(\Omega)$ such that $\operatorname{Tr}\phi=g$ and
\begin{equation}\label{harmonic_extension_weak}
\int_\Omega \nabla \phi\cdot\nabla \varphi\,dx=0
\quad
\text{for all}\ \varphi\in \dot H^1_0(\Omega).
\end{equation}
Moreover, we have the trace inequality
\begin{equation}
\|\phi\|_{\dot H^1(\Omega)}
\le
\|g\|_{\dot H^{\frac{1}{2}}(\partial\Omega)}.
\end{equation}
\end{proposition}

\begin{definition}[Dirichlet-to-Neumann operator]\label{def:DN_general}
Let \(d\geq 2\), and let \(\Omega\subset \mathbb R^d\) be a bounded Lipschitz domain with boundary \(\Sigma:=\partial\Omega\). For \(f\in \dot H^{\frac{1}{2}}(\Sigma)\), let \(\phi\in \dot H^1(\Omega)\) be the unique weak harmonic extension of \(f\), i.e.
\begin{equation}
\operatorname{Tr}\phi=f,
\qquad
\int_\Omega \nabla \phi\cdot \nabla \varphi\,dx=0
\quad\text{for all }\varphi\in \dot H^1_0(\Omega).
\end{equation}
The Dirichlet-to-Neumann operator $\mathcal G(\Sigma):H^{\frac{1}{2}}(\Sigma)\to H^{-\frac{1}{2}}(\Sigma)$ is defined by
\begin{equation}
\big\langle \mathcal G(\Sigma)f,g\big\rangle_{H^{-\frac{1}{2}},H^{\frac{1}{2}}}
:=
\int_\Omega \nabla \phi_f\cdot \nabla \psi \,dx,
\end{equation}
where \(\psi\in \dot H^1(\Omega)\) is any function satisfying
\(\operatorname{Tr}\psi =g\).
\end{definition}

\begin{remark}\label{remark:low_regularity_DN}
The variational definition realizes \(\mathcal G(\Sigma)f\) as an element of
\(\dot H^{-\frac{1}{2}}(\Sigma)\). Under additional regularity assumptions, this
distribution may be represented by an actual function. For instance, in the
Lipschitz setting considered below, layer potential theory gives
\(\mathcal G(\Sigma)f\in L^2(\Sigma)\) for suitable boundary data, such as
\(f\in H^1(\Sigma)\). In this case the weak pairing above agrees with the
\(L^2\)-pairing
\[
\big\langle \mathcal G(\Sigma)f,g\big\rangle
=
\int_\Sigma \mathcal G(\Sigma)f\, g\,d\sigma .
\]
If, in addition, the harmonic extension admits a classical normal derivative
on \(\Sigma\), then this \(L^2\)-representative agrees with the classical
expression
\[
\mathcal G(\Sigma)f=n\cdot\nabla\phi .
\]
Eventually, by interior elliptic regularity, \(\phi\) is smooth and
classically harmonic in the interior of \(\Omega\).
\end{remark}

\subsection{Layer potential theory} We begin by recalling the definitions of the single- and double-layer potentials. We mostly restrict attention to \(\mathbb R^2\), the setting of the present paper, since the explicit kernel formulas take different forms in higher dimensions. Many of the corresponding layer-potential estimates,
however, have higher-dimensional analogues. For a treatment of the layer potential theory on doubly periodic graph domains in \(\mathbb R^3\), see
Dong--Gancedo--Nguyen \cite{Dong-Gancedo-Nguyen-23-2}.

\begin{definition}[Single- and double-layer potententials]
Let 
\[\cN(z,z') := -\frac{1}{2\pi}\log|z - z'|,\quad z, z' \in \R^2\]
and by abusing notation we also write
\[\cN(z - z') = \cN(z, z')\]
be the Newtonian kernel. Let $\Omega\subset\R^2$ be a Lipschitz  $f\in L^1(\Sigma)$. We define the \emph{single-layer potential} operator
\begin{equation}
\cS[\Sigma] f(z) := \int_\Sigma \cN(z - z')f(z')\ d\sigma(z'),\quad z\in \R^2 \bs \Sigma 
\end{equation}
the \emph{boundary single-layer potential} operator
\begin{equation}
S[\Sigma] f(z) := \pv \int_\Sigma \cN(z - z')f(z')\ d\sigma(z'),\quad z\in \Sigma 
\end{equation}
the \emph{double-layer potential} operator
\begin{equation}
\cK[\Sigma] f(z) := \int_\Sigma n(z')\cdot \nabla \cN(z - z') f(z')\ d\sigma(z'),\quad z\in\R^2\bs\Sigma
\end{equation}
and the \emph{boundary double-layer potential} operator
\begin{equation}
K[\Sigma] f(z) := \pv \int_\Sigma n(z')\cdot \nabla \cN(z - z') f(z')\ d\sigma(z'),\quad z\in\Sigma
\end{equation}
\end{definition}

We often drop the dependence on $\Sigma$ when the context is clear. We have the following classical results on the layer potential operators. See, for example, \cite[Chapter~6]{McLean2000} and
\cites{FabesJodeitRiviere1978,Verchota1984,Costabel1988,Kenig1986}.

\begin{theorem}\label{thm:jump_relations}
Let $\Omega_-$ and $\Omega_+$ be the interior and exterior of the boundary $\Sigma$. For convenience of notation, we use $\lim_{w^\pm\to z}$ to denote a non-tangential limit of elements $w^\pm\in \Omega_\pm$ to $z\in \Sigma$. Then, for $f\in L^p(\Sigma)$ with $1\le p\le \infty$, $\cK f$ and and $\cS f$ are both harmonic in $\R^2\bs \Sigma$ with the following standard jump relations;

\begin{enumerate}[(1)]
\item The double-layer potential has a jump across the boundary, given by
\[
\lim_{w^\pm\to z} \cK f(w^\pm)
=
\left(\mp \frac{1}{2}I + K\right)f(z).
\]
\item The single-layer potential is continuous across the boundary, namely
\[
\lim_{w^\pm\to z} \cS f(w^\pm)
=
Sf(z).
\]
\item The normal derivative of the single-layer potential has a jump across the boundary, namely
\[
\lim_{w^\pm\to z} n(z)\cdot \nabla \cS f(w^\pm)
=
\left(\pm \frac{1}{2}I + K^*\right)f(z).
\]
\item The tangential derivative of the single-layer potential is continuous across the boundary, namely
\[
\lim_{w^\pm\to z} \tau(z)\cdot \nabla \cS f(w^\pm)
=
\tau(z)\cdot \nabla \cS f(z)
\]
for $z\in \Sigma$.
\end{enumerate}
\end{theorem}

We also have the following theorem on the invertibility of layer potential operators. For more detailed discussions, see, for example, \cites{Kenig1986, Verchota1984}.
\begin{theorem}[Invertibility of layer potential operators]\label{invertibility}
Suppose $\Omega$ is a Lipschitz domain with boundary $\Sigma:= \partial\Omega$. Then the mappings
\begin{equation}
\begin{aligned}
\frac{1}{2}I + K: L^p(\Sigma) \to L^p(\Sigma),\quad p\in (2, \infty],\\
\frac{1}{2}I + K: W^{1,p}(\Sigma) \to W^{1,p}(\Sigma),\quad p\in (1, 2],\\
\frac{1}{2}I - K^*: L_0^p(\Sigma) \to L_0^p(\Sigma),\quad p\in (1, 2]\\
\end{aligned}
\end{equation}
are invertible. Here $L_0^p$ denotes mean-zero functions in $L^p$. 
\end{theorem}

Beyond invertibility, we want to obtain quantitative bounds of the operators $(\frac{1}{2}I\pm K)^{-1}$. A key ingredient we need is the following classical identity, which allows us to see that $L^2(\Sigma)$-norms of tangential and normal derivatives of a harmonic function mutually control each other. 
\begin{lemma}[Rellich identity] 
Let $\Omega$ be a Lipschitz domain. Suppose $H$ is harmonic and $\nabla H\in L^2(\Omega)$. For any $V\in C^1(\overline\Omega)$, we have
\begin{equation}
\int_{\partial\Omega} V\cdot n|\nabla H|^2\ d\sigma = \int_{\partial\Omega}2\partial_n H(V\cdot\nabla H)\ d\sigma + \int_\Omega (\nabla\cdot V)|\nabla H|^2\ dx - \int_\Omega 2(\nabla V\nabla H)\cdot\nabla H\ dx
\end{equation}
\end{lemma}

The proof is a straightforward application of the divergence theorem, so we omit it here. See, for example, \cites{Verchota1984, Kenig1986, AgrawalAlazard2023Rellich}, for more details.

We are now ready to prove our desired quantitative bounds. Such bounds are classical; see, for example, \cites{Kenig1986, Verchota1984}. We primarily run through the proof to make geometric constants more explicit. For convenience of notation, we use $|\cdot|$ to denote the surface measure on $\Sigma$ induced by Lebesgue measure on $\R^2$.

\begin{theorem}[Boundedness of layer potentials]\label{thm:quantitative_bound_layer_potential}
Suppose $\Omega$ is a bounded Lipschitz domain with $\Sigma := \partial\Omega$ and $|\Sigma| < \infty$, then 
\begin{equation}
\bigg\|\bigg(\frac{1}{2}I \pm K^*\bigg)^{-1} f\bigg\|_{L^2(\Sigma)}\leq M\|f\|_{L^2(\Sigma)}
\end{equation}
for all $f\in L^2(\Sigma)$ with mean zero, where $M$ depends on the Lipschitz geometric constants of $\Omega$. Moreover, suppose one can find a smooth vector field $V$ compactly supported on a neighborhood of $\Sigma$ such that $V\cdot n\geq c_0$ on $\Sigma$ for some $c_0 > 0$, $\|V\|_{L^\infty}\leq C_0$ for some $C_0 > 0$ and $\|\nabla V\|_{L^\infty}\leq r_0^{-1}$ for some $r_0 > 0$, then 
\begin{equation}
M\leq C\bigg(\frac{C_0}{c_0}\bigg)\bigg[1 + \bigg(\frac{|\Sigma|}{r_0}\bigg)^2\bigg]^{\frac{1}{2}}
\end{equation}
where $C$ is an absolute constant. 
\end{theorem}

Before we dive into the proof, we show that the vector field $V$ in Theorem~\ref{thm:quantitative_bound_layer_potential} can always be found, and the quantitaties $c_0, C_0$, and $r_0$ depends on \emph{Lipschitz character} of $\Omega$, which we define below.

\begin{definition}[Lipschitz character]
Given a Lipschitz domain $\Omega$, there exist $\rho,L>0$, an integer $N$, and boundary patches $U_1,\dots,U_N$ covering $\partial\Omega$ with overlap bounded by some $M\in\mathbb N$, such that for each $j$ a rigid motion $R_j$ satisfies
\[R_j(\Omega\cap U_j)=\{(x, y)\in B'_\rho\times(-\rho,\rho): y< \varphi_j(x)\},\quad \mathrm{with}\ \|\partial_x\varphi_j\|_{L^\infty}\le L.\]
Any equivalent choice of such parameters is a Lipschitz character of $\Omega$.    
\end{definition}

Now we construct the desired vector field $V$. Suppose we have boundary patches $U_1, ..., U_N$ as given in the above definition, we choose smaller patches $U_j'\Subset U_j$ still covering $\partial\Omega$, and let $\chi_j\in C_c^\infty(U_j)$ satisfy
\[0\le\sum_j\chi_j\le 1\ \mathrm{on}\ \mathbb R^2,\quad \sum_j\chi_j=1\ \mathrm{near}\ \partial\Omega,\quad \mathrm{and}\ \|\partial_x\chi_j\|_{L^\infty}\le \frac{C}{\rho}.\]
We set
\[\nu_j:=R_j^{-1}(e_2),\quad V(x):=\sum_{j=1}^N\chi_j(x)\nu_j,\]
and then $V\in C_c^\infty$ as a finite sum of smooth compactly supported fields. On the graph $y=\varphi_j(x)$ the outward unit normal is 
\[n=\frac{(-\partial_x\varphi_j,1)}{\sqrt{1+|\partial_x\varphi_j|^2}}\]
so for a.e.\ $x\in\partial\Omega\cap U_j$, $\nu_j\cdot n(x)\ge 1/\sqrt{1+L^2}$. Hence for a.e.\ $x\in\partial\Omega$,
\begin{equation}
V(x)\cdot n(x)=\sum_{j=1}^N\chi_j(x)\,\nu_j\cdot n(x)\ge\frac{1}{\sqrt{1+L^2}}\sum_{j=1}^N\chi_j(x)=\frac{1}{\sqrt{1+L^2}},
\end{equation}
and we may take $c_0=(1+L^2)^{-1/2}$. Also, direct computation gives
\begin{equation}
\begin{aligned}
|V(x)|&\le\sum_j\chi_j(x)\le 1,\\
|\nabla V(x)|&\le\sum_j|\nabla\chi_j(x)|\,|\nu_j|\le \frac{CM}{\rho}
\end{aligned}
\end{equation}
so one may choose $C_0=1$ and $r_0=\rho/(CM)$, so $c_0$, $C_0$, $r_0$ depend on the Lipschitz character of $\Omega$.

\begin{proof}[Proof of Theorem \ref{thm:quantitative_bound_layer_potential}]
First of all, we note that if there is some $M > 0$ such that
\begin{equation}\label{double_layer_pm}
\frac{1}{M}\leq \frac{\|(\frac{1}{2}I - K^*)f\|_{L^2(\Sigma)}}{\|(\frac{1}{2}I + K^*)f\|_{L^2(\Sigma)}}\leq M,
\end{equation}
then 
\begin{equation}
\begin{aligned}
\|f\|_{L^2(\Sigma)} &\leq \bigg\|\bigg(\frac{1}{2}I - K^*\bigg)f\bigg\|_{L^2(\Sigma)} + \bigg\|\bigg(\frac{1}{2}I + K^*\bigg)f\bigg\|_{L^2(\Sigma)}\\
&\leq (M + 1)\bigg\|\bigg(\frac{1}{2}I - K^*\bigg)f\bigg\|_{L^2(\Sigma)}
\end{aligned}
\end{equation}
and thus $(\frac{1}{2}I - K^*)^{-1}$ is bounded on $L^2(\Sigma)$ with operator norm $M + 1$. Therefore, our goal is to prove (\ref{double_layer_pm}). Now, we consider $H = \cS f$ and recall from \eqref{thm:jump_relations} that
\begin{equation}\label{single_layer_normal_jump_2}
\begin{aligned}
\partial_n H|_{\partial\Omega} = \bigg(-\frac{1}{2}I + K^*\bigg)f,\qquad \partial_n^c H|_{\partial\Omega} = \bigg(\frac{1}{2}I + K^*\bigg)f.
\end{aligned}
\end{equation}

If we do the decompositions $V = (V\cdot n)n + (V\cdot \tau)\tau$, $\nabla H = (\partial_n H)n + (\partial_\tau H)\tau$, and use Rellich identity, we get
\begin{equation}
\begin{aligned}
\int_{\partial\Omega} (V\cdot n)|\partial_\tau H|^2\ d\sigma  &= \int_{\partial\Omega} (V\cdot n)|\partial_n H|^2\ d\sigma + \int_{\partial\Omega} 2(V\cdot \tau)(\partial_n H)(\partial_\tau H)\ d\sigma\\
&+ \int_\Omega (\Div V)|\nabla H|^2\ dx - 2\int_\Omega (\nabla V\nabla H)\cdot \nabla H\ dx.
\end{aligned}
\end{equation}
Suppose $V$ is a vector field as in the assumption of the problem. Then
\begin{equation}
\begin{aligned}
\int_{\partial\Omega} |\partial_n H|^2 d\sigma &\leq \frac{C_0}{c_0}\int_{\partial\Omega} |\partial_\tau H|^2\ d\sigma + \frac{2C}{c_0 r_0}\int_{\partial\Omega}|\partial_n H|\, |\partial_\tau H|\ d\sigma +  \frac{C}{c_0 r_0}\int_\Omega |\nabla H|^2\ dx,
\end{aligned}
\end{equation}
and an AM-GM inequality argument gives
\begin{equation}
\begin{split}
\int_{\partial\Omega} |\partial_n H|^2 d\sigma &\leq \frac{CC_0^2}{c_0^2}\int_{\partial\Omega}|\partial_\tau H|^2\ d\sigma +  \frac{C}{c_0 r_0}\int_\Omega |\nabla H|^2\ dx.
\end{split}
\end{equation}
To bound the last term on the right-hand side, we note that $H$ is harmonic in $\Omega$, so integration by parts gives
\begin{equation}\label{H_int_by_parts}
\int_\Omega |\nabla H|^2 \ dx = \int_{\Omega} (H-\av(H))\partial_n H\ d\sigma,
\end{equation}
where 
\begin{equation}
\av (f) := \frac{1}{|\partial\Omega|}\int_{\partial\Omega} f\ d\sigma.  
\end{equation} 
Now, since $H - \av(H)$ has zero mean on $\partial\Omega$, we note that
\begin{equation}
\begin{aligned}
(H - \av(H))(z) &= (H - \av(H))(z) - \av(H - \av(H)) \\
                &= \frac{1}{|\partial\Omega|}\int_{\partial\Omega}[(H-\av(H))(z) - (H - \av(H))(w)] d\sigma(w),
\end{aligned}
\end{equation}
A straightforward calculation using any parametrization of $\partial\Omega$ yields
\begin{equation}
|(H-\av(H))(z) - (H - \av(H))(w)|\leq \int_{\partial\Omega} |\partial_\tau H|\ d\sigma
\end{equation}
and thus
\begin{equation}
\begin{aligned}
|(H - \av(H))(z)|&\leq \int_{\partial\Omega} |\partial_\tau H|\ d\sigma \leq |\partial\Omega|^{\frac{1}{2}}\bigg(\int_{\partial\Omega} |\partial_\tau H|^2\ d\sigma\bigg)^{\frac{1}{2}}.
\end{aligned}
\end{equation}
It follows that
\begin{equation}
\| H - \av(H)\|_{L^2(\partial\Omega)}\leq |\partial\Omega|\, \|\partial_\tau H\|_{L^2(\partial\Omega)}.
\end{equation}
Using (\ref{H_int_by_parts}) and AM-GM, we get
\begin{equation}
\|\partial_n H\|_{L^2(\partial\Omega)}^2 \leq C\bigg(\frac{C_0}{c_0}\bigg)\bigg[1 + \bigg(\frac{|\partial\Omega|}{r_0}\bigg)^2\bigg] \|\partial_\tau H\|_{L^2(\partial\Omega)}^2.
\end{equation} 
An analogous computation using $\Omega^c$ in place of $\Omega$ yields
\begin{equation}\label{tangential_normal}
\|\partial_\tau H\|_{L^2(\partial\Omega)}^2 \leq C\bigg(\frac{C_0}{c_0}\bigg)\bigg[1 + \bigg(\frac{|\partial\Omega|}{r_0}\bigg)^2\bigg] \|\partial_n^c H\|_{L^2(\partial\Omega)}^2
\end{equation}
and using (\ref{single_layer_normal_jump_2}) we get
\begin{equation}
\frac{\|(\frac{1}{2}I - K^*)f\|_{L^2(\Sigma)}}{\|(\frac{1}{2}I + K^*)f\|_{L^2(\Sigma)}}\leq M.
\end{equation}
The other side of (\ref{double_layer_pm}) follows in an analogous fashion. 
\end{proof}

\subsection{Layer potential representation of Dirichlet-to-Neumann operator} We now derive an integral formulation of the Dirichlet-to-Neumann operator on a bounded Lipschitz domain \(\Omega\), with \(\Sigma:=\partial\Omega\), for Lipschitz boundary data. The formulation is intrinsic to the boundary and hence independent of parametrization. Nevertheless, for computations we often fix a non-degenerate counterclockwise parametrization \(z:\T\to\Sigma\).

By Theorem \ref{thm:jump_relations} and \ref{invertibility}, the solution to \eqref{dirichlet_bv} can be written as
\begin{equation}
\begin{split}
\phi &= \cK\bigg(\frac{1}{2}I + K\bigg)^{-1}f .
\end{split}
\end{equation}
With the parametrization $z$, we write
\begin{equation}
\begin{split}
\cK f(z) &= -\frac{1}{2\pi}\int_\T \frac{z_1'(\beta)(z_2 - z_2(\beta)) - z_2'(\beta)(z_1 - z_1(\beta))}{(z_1 - z_1(\beta))^2 + (z_2 - z_2(\beta))^2} f(z(\beta))\ d\beta.
\end{split}
\end{equation}
We denote 
\begin{equation}\label{theta_form_1}
\theta(\beta) = \frac{\partial}{\partial_\beta}\bigg[\bigg(\frac{1}{2}I + K\bigg)^{-1}f(z(\beta))\bigg] ,\qquad \Theta = \partial_\tau\bigg(\frac{1}{2}I + K\bigg)^{-1}f.
\end{equation}
Direct computation using integration by parts gives
\begin{equation}
\begin{split}
\partial_x\phi(x,y) &= \frac{1}{2\pi}\int_\T \frac{-(y - z_2(\beta))}{(x-z_1(\beta))^2 + (y - z_2(\beta))^2}\theta(\beta)\ d\beta\\
&=\frac{1}{2\pi}\int_\T \frac{-(y - y')}{(x-x')^2 + (y - y')^2}\Theta(x', y')\ d\sigma;
\end{split}
\end{equation}
\begin{equation}
\begin{split}
\partial_y\phi(x,y) &= \frac{1}{2\pi}\int_\T \frac{x - z_1(\beta)}{(x-z_1(\beta))^2 + (y - z_2(\beta))^2}\theta(\beta)\ d\beta\\
&= \frac{1}{2\pi}\int_\T \frac{x - x'}{(x-x')^2 + (y - y')^2}\Theta(x', y')\ d\sigma.
\end{split}
\end{equation}
Then it follows that
\begin{equation}
\begin{aligned}
\cG(\Sigma)f(x, y) &= \lim_{(\tilde x, \tilde y)\to z(\alpha)} n(z(\alpha))\cdot \nabla\phi(\tilde{x}, \tilde{y})\\
&= -\lim_{(\tilde x, \tilde y)\to (x, y)} \frac{1}{2\pi}\frac{1}{|z'(\alpha)|}\int_\T \frac{z_1'(\alpha)(\tilde{x} - z_1(\beta)) + z_2'(\alpha)(\tilde{y} - z_2(\beta))}{(\tilde{x} - z_1(\beta))^2 + (\tilde{y} - z_2(\beta))^2}\theta(\beta)\ d\beta\\
&= -\lim_{(\tilde x, \tilde y)\to (x, y)}\frac{1}{2\pi}\int \frac{\tau(x,y)\cdot (\tilde x - x', \tilde y - y')}{(\tilde x-x')^2 + (\tilde y - y')^2} \Theta(x', y')\ d\sigma(x', y')\\
&= -\lim_{(\tilde x, \tilde y)\to (x, y)}\tau(x, y)\cdot \nabla \cS\Theta(\tilde x, \tilde y).
\end{aligned}
\end{equation}
We remember from Theorem \ref{thm:jump_relations} that the tangential derivative of the single-layer potential is continuous up to the boundary, so 
\begin{equation}\label{integral_dn}
\begin{aligned}
\cG(\Sigma)f(x, y) &= -\tau(x, y)\cdot \nabla \cS\Theta (x, y)\\
&= -\pv\frac{1}{2\pi}\int \frac{\tau(x,y)\cdot (x - x',  y - y')}{(x-x')^2 + (y - y')^2} \Theta(x', y')\ d\sigma(x', y').
\end{aligned}
\end{equation}
The integral expression is almost explicit except for the inverse $(\frac{1}{2}I + K)^{-1}$ in $\Theta$. However, Theorem \ref{invertibility} ensures that the implicitness of $(\frac{1}{2}I \pm K)^{-1}$ does not cause us any trouble in doing estimates on $L^2(\Sigma)$. 

With abuse of notation, we often identify $f$ with its trace $f\circ z$. In the same spirit, we also regard $K$ and $S$ as operators defined on functions $f:\T\to \C$ when the parametrization is clear from the context. Under these conventions, we have the following commutativity lemma. 

\begin{lemma}\label{theta_lemma}
Let $\theta$ and $\Theta$ be defined as in \eqref{theta_form_1}. Then
\begin{equation}
\theta = \bigg(\frac{1}{2}I - K^*\bigg)^{-1}\partial_\alpha f,\qquad \Theta = \bigg(\frac{1}{2}I - K^*\bigg)^{-1}\partial_\tau f
\end{equation}
\end{lemma}
\begin{proof}
We only show the statement for $\theta$; rewriting the proof without parametrization yields the result for $\Theta$. By duality and a limiting argument, it suffices to show
\begin{equation}\label{lem_duality_form}
\bigg\langle \bigg(\frac{1}{2}I - K^*\bigg)\theta, \varphi\bigg\rangle_{L^2(\T)} = -\langle f, \partial_\alpha\varphi\rangle_{L^2(\T)}
\end{equation}
for every $\varphi\in C^\infty(\T)$. For convenience of notation, we drop the subscript $L^2(\T)$ in the remainder of the proof. Note that
\begin{equation}\label{lem_duality_form_2}
\begin{aligned}
\bigg\langle \bigg(\frac{1}{2}I - K^*\bigg)\theta, \varphi\bigg\rangle &= -\bigg\langle\theta, \partial_\alpha\bigg(\frac{1}{2}I - K\bigg)\varphi\bigg\rangle\\
&= -\bigg\langle \bigg(\frac{1}{2}I + K\bigg)^{-1}f,\   \partial_\alpha\bigg(\frac{1}{2}I - K\bigg)\varphi\bigg\rangle.\\
\end{aligned}
\end{equation}
A straightforward integration by parts gives
\begin{equation}\label{K_K*_adjoint_prop}
\partial_\alpha K\varphi = -K^*\partial_\alpha\varphi
\end{equation}
and plugging \eqref{K_K*_adjoint_prop} into \eqref{lem_duality_form_2} we obtain
\begin{equation}
\begin{aligned}
-\bigg\langle \bigg(\frac{1}{2}I + K\bigg)^{-1}f,\   \partial_\alpha\bigg(\frac{1}{2}I - K\bigg)\varphi\bigg\rangle &= -\bigg\langle \bigg(\frac{1}{2}I + K\bigg)^{-1}f, \bigg(\frac{1}{2}I + K^*\bigg)\partial_\alpha\varphi\bigg\rangle\\
&= -\langle f, \partial_\alpha\varphi\rangle
\end{aligned}
\end{equation}
as desired. 
\end{proof}

\section{Star-Shaped Domain and the Interface Equation}\label{sec:interface_eqn}
In this section, we show that the domain formulation \eqref{domain_eqn_u}
reduces to a clean nonlocal parabolic equation on the interface when the domain is \emph{star-shaped}. The derivation is only formal at the moment, and the reader may assume as much regularity as necessary for the computations below. In subsequent sections, we will state relevant regularity hypotheses precisely and explain how the results obtained for the interface equation translate back to the original domain formulations \eqref{domain_eqn_u} and \eqref{domain_eqn_p}.

\subsection{Removal of injection singularity} We define
\begin{equation}
u_*(z) := \frac{z}{|z|^2} = \frac{1}{\overline{z}},
\end{equation}
and a direct calculation shows that 
\begin{equation}
\Div(u - u_*) = 0\quad \text{in } \Omega(t).
\end{equation} 
Therefore, if we define
\begin{equation}
\phi(x, y):= p(x,y) + \frac{1}{2}\log(x^2 + y^2),
\end{equation}
then $\phi$ solves the Dirichlet boundary value problem
\begin{equation}
\left\{
\begin{array}{ll}
 \Delta\phi = 0 & \mathrm{in}\ \Omega(t)  \\
 \phi = \frac{1}{2}\log(x^2+y^2) & \mathrm{on}\ \Sigma(t)
\end{array}
\right.
\end{equation}

\subsection{Kinematic boundary condition} Suppose \(z(\alpha,t)\) is a counterclockwise Lagrangian parametrization of
the boundary \(\Sigma(t)\). Darcy's law gives
\begin{equation}
z_t(\alpha,t)
=
u(z(\alpha,t),t)
=
-\nabla p(z(\alpha,t),t).
\end{equation}
Writing \(p=\phi-\log|x|\), we have
\[
-\nabla p=-\nabla\phi+\frac{x}{|x|^2}.
\]
Hence, for any outward normal vector \(N\) on the boundary,
\begin{equation}\label{kinematic_bc}
\begin{split}
z_t(\alpha,t)\cdot N(z(\alpha,t)) &=
-N(z(\alpha,t))\cdot\nabla\phi(z(\alpha,t),t)
+
N(z(\alpha,t))\cdot\frac{z(\alpha,t)}{|z(\alpha,t)|^2}.
\end{split}
\end{equation}

Now assume that \(\Omega(t)\) is star-shaped with respect to the origin. Then
we may write the boundary as
\begin{equation}
z(\alpha,t)=e^{ic(\alpha,t)}h(c(\alpha,t),t),
\end{equation}
where \(c(\alpha,t)\) accounts for a possible reparametrization of the angular
variable. We define the unnormalized outward normal
\begin{equation}
N(z(\alpha,t))
:=
-i\frac{\partial}{\partial c}\bigl(e^{ic}h(c,t)\bigr)
=
e^{ic}h(c,t)-ie^{ic}\partial_ch(c,t).
\end{equation}
Direct calculation gives
\begin{equation}
z_t(\alpha,t)\cdot N(z(\alpha,t))
=
h(c,t)\,\partial_t h(c,t);
\end{equation}
\begin{equation}
N(z(\alpha,t))\cdot \frac{z(\alpha,t)}{|z(\alpha,t)|^2}=1.
\end{equation}
Therefore, after relabeling the angular variable, \eqref{kinematic_bc} yields
\begin{equation}
\bigl(h\,\partial_t h\bigr)(\alpha,t)
=
-N(e^{i\alpha}h(\alpha,t))\cdot
\nabla\phi(e^{i\alpha}h(\alpha,t),t)
+1.
\end{equation}

\subsection{Interface equation with Dirichlet-to-Neumann operator} We now define a version of the Dirichlet-to-Neumann operator specific to star-shaped domain. Under low regularity, the definition should be interpreted as in Definition \ref{def:DN_general} and Remark \ref{remark:low_regularity_DN}. 

\begin{definition}[Dirichlet-to-Neumann operator on star-shaped domains]\label{def:DN_star_shaped}
Let \(h(\cdot,t)\) be positive everywhere and define
\[
\Omega_h(t):=\{re^{i\alpha}:0\le r<h(\alpha,t),\ \alpha\in\mathbb T\},
\qquad
\Sigma_h(t):=\partial\Omega_h(t).
\]
For each fixed \(t\), let \(\phi(\cdot,t)\) be the harmonic extension of
\(f(\cdot,t)\) to \(\Omega_h(t)\), namely
\begin{equation}\label{eqn:Dirichlet_problem_star_shaped}
\begin{cases}
\Delta \phi(\cdot,t)=0 & \text{in } \Omega_h(t),\\
\phi(h(\alpha,t)e^{i\alpha},t)=f(\alpha,t) & \text{on } \Sigma_h(t),\\
\nabla\phi(\cdot,t)\in L^2(\Omega_h(t)) & \\
\end{cases}
\end{equation}
We define the
\emph{Dirichlet-to-Neumann operator} by
\begin{equation}
G(h)f(\alpha,t)
:=
N_h(\alpha,t)\cdot
\nabla\phi(h(\alpha,t)e^{i\alpha},t),
\end{equation}
where
\[
N_h(\alpha,t):=h(\alpha,t)e_r(\alpha)-h_\alpha(\alpha,t)e_\theta(\alpha)
\]
is the unnormalized outward normal to \(\Sigma_h(t)\).
\end{definition}

A key observation is that the natural boundary unknown to work on is not the radius \(h\) itself, but its logarithm $\eta:=\log h$. The interface equation then becomes
\begin{equation}
\partial_t\eta+e^{-2\eta}\bigl(G(h)\eta-1\bigr)=0,
\qquad h=e^\eta.
\end{equation}

\subsection{Integral expression} We adapt the layer-potential
computations from section \ref{sec:formulation_DN} to obtain an integral formulation of \(G(h)\eta\). With \(z(\alpha)=h(\alpha)e^{i\alpha}\),
the boundary double-layer operator and its adjoint become
\begin{equation}\label{K_expression_1}
\begin{split}
K[h]g(\alpha)
&=-\pv\frac{1}{2\pi}\int_\T
\partial_\beta
\arctan\!\bigg(
\frac{h(\alpha)\sin\alpha-h(\beta)\sin\beta}
     {h(\alpha)\cos\alpha-h(\beta)\cos\beta}
\bigg)g(\beta)\,d\beta \\
&=-\pv\frac{1}{2\pi}\int_\T
\frac{
h(\alpha)h'(\beta)\sin(\alpha-\beta)
-h(\alpha)h(\beta)\cos(\alpha-\beta)
+h^2(\beta)}
{h^2(\alpha)+h^2(\beta)-2h(\alpha)h(\beta)\cos(\alpha-\beta)}
g(\beta)\,d\beta ,
\end{split}
\end{equation}
and
\begin{equation}\label{K^*_expression_1}
\begin{split}
K^*[h]g(\alpha)
&=-\pv\frac{1}{2\pi}\partial_\alpha\int_\T
\arctan\!\bigg(
\frac{h(\alpha)\sin\alpha-h(\beta)\sin\beta}
     {h(\alpha)\cos\alpha-h(\beta)\cos\beta}
\bigg)g(\beta)\,d\beta \\
&=-\pv\frac{1}{2\pi}\int_\T
\frac{
-h(\beta)h'(\alpha)\sin(\alpha-\beta)
-h(\alpha)h(\beta)\cos(\alpha-\beta)
+h^2(\alpha)}
{h^2(\alpha)+h^2(\beta)-2h(\alpha)h(\beta)\cos(\alpha-\beta)}
g(\beta)\,d\beta .
\end{split}
\end{equation}
Moreover, by \eqref{integral_dn},
\begin{equation}\label{G_integral_expression}
G(h)g(\alpha)
=
\pv\int_\T \Lambda[h](\alpha,\beta)\theta(\beta)\,d\beta,
\end{equation}
where the kernel admits the expression
\begin{equation}\label{kernel_expression_1}
\begin{split}
\Lambda[h](\alpha,\beta)
&=\frac{1}{4\pi}\partial_\alpha
\log\!\big[
(h(\alpha)\cos\alpha-h(\beta)\cos\beta)^2
+(h(\alpha)\sin\alpha-h(\beta)\sin\beta)^2
\big] \\
&=
\frac{1}{2\pi}
\frac{
h(\alpha)h'(\alpha)
-h'(\alpha)h(\beta)\cos(\alpha-\beta)
+h(\alpha)h(\beta)\sin(\alpha-\beta)}
{h^2(\alpha)+h^2(\beta)-2h(\alpha)h(\beta)\cos(\alpha-\beta)}.
\end{split}
\end{equation}
Here, inheriting the notation in section \ref{sec:formulation_DN} and using Lemma \ref{theta_lemma}, we write
\begin{equation}\label{theta_definition}
\theta
:=
\partial_\alpha\bigg(\frac12 I+K[h]\bigg)^{-1}g
=
\bigg(\frac12 I-K^*[h]\bigg)^{-1}\partial_\alpha g.
\end{equation}

It is often useful to express \(K[h]\), \(K^*[h]\), and \(\Lambda[h]\) in
complex notation. If \(F:\mathbb R\to\mathbb C\) is non-vanishing, then we have the elementary identities
\begin{equation}
\frac{d}{dx}\arg F(x)
=
\operatorname{Im}\frac{d}{dx}\log F(x),
\qquad
\frac{d}{dx}\log|F(x)|
=
\operatorname{Re}\frac{d}{dx}\log F(x).
\end{equation}
Therefore, we write
\begin{align}
K[h]g(\alpha)
&=
-\pv\frac{1}{2\pi}\operatorname{Im}
\int_\T
\partial_\beta
\log\!\big(h(\alpha)e^{i\alpha}-h(\beta)e^{i\beta}\big)
g(\beta)\,d\beta, \\
\label{K^*_expression_2}
K^*[h]g(\alpha)
&=
-\pv\frac{1}{2\pi}\operatorname{Im}\,
\partial_\alpha
\int_\T
\log\!\big(h(\alpha)e^{i\alpha}-h(\beta)e^{i\beta}\big)
g(\beta)\,d\beta, \\
\label{kernel_expression_2}
\Lambda[h](\alpha,\beta)
&=
\frac{1}{2\pi}\operatorname{Re}\,
\partial_\alpha
\log\!\big(h(\alpha)e^{i\alpha}-h(\beta)e^{i\beta}\big).
\end{align}
Again, we often drop the dependence of $K, K^*$, and $\Lambda$ on $h$ for convenience of notations.

\subsection{Symmetries of the interface equation} We summarize the preceding derivation of the interface equation as follows.

\begin{proposition}[Interface equation]\label{prop:interface_eqn_derivation}
Consider the injection-driven Hele-Shaw problem \eqref{domain_eqn_u}. Assume that the fluid domain is star-shaped with respect to the origin and is given by
\[
\Omega_h(t):=\{re^{i\alpha}:0\le r<h(\alpha,t),\ \alpha\in\mathbb T\},
\]
where \(h>0\) is \(2\pi\)-periodic in \(\alpha\). Let
\[
\eta(\alpha,t):=\log h(\alpha,t).
\]
Then \(\eta\) satisfies
\begin{equation}\label{interface_eqn}
\partial_t\eta+e^{-2\eta}\bigl(G(h)\eta-1\bigr)=0.
\end{equation}

Here \(G(f)g\) denotes the Dirichlet-to-Neumann operator from Definition
\ref{def:DN_star_shaped}. Moreover, for sufficiently regular \(f,g:\mathbb T\to
\mathbb R\), \(G(f)g\) admits the integral representation
\begin{equation}\label{G_integral_summary}
G(f)g(\alpha)
=
\pv\int_{\mathbb T}\Lambda[f](\alpha,\beta)\,\theta(\beta)\,d\beta,
\end{equation}
where \(\Lambda[f]\) is given by either \eqref{kernel_expression_1} or
\eqref{kernel_expression_2}, and \(\theta\) is determined by
\begin{equation}\label{theta_eqn}
\bigg(\frac12 I-K^*[f]\bigg)\theta
=
\partial_\alpha g.
\end{equation}
The adjoint double-layer operator \(K^*[f]\) is given by either
\eqref{K^*_expression_1} or \eqref{K^*_expression_2}.
\end{proposition}
Since fluid is injected into the domain at a non-zero constant rate, the equations do not admit steady-state solutions. In our search for special solutions, we begin with interfaces with no angular dependence. These correspond to disks centered at the origin that expand in time.

\begin{proposition}[Radial solutions]
The only solutions of \eqref{interface_eqn} that are constant in the spatial
variable are the expanding circles centered at the origin. More precisely, if
$h(\alpha,0)=R_0>0$ is constant, then
\[
h(\alpha,t)=\sqrt{R_0^2+2t}.
\]
Equivalently, in the variable $\eta=\log h$, these solutions are given by
\[
\eta(\alpha,t)=\log \sqrt{R_0^2+2t}.
\]
\end{proposition}

We next look at the scaling laws of the interface equation.

\begin{proposition}[Scaling law]
Suppose $h$ is a solution to equation \eqref{interface_eqn}, then $\lambda h(\alpha, t/\lambda^2)$ is also a solution if $\lambda\neq 0$. Equivalently, if $\eta$ is a solution, then $\eta_\lambda(\alpha, t):=\eta(\alpha, t/\lambda^2) + \ln\lambda$ is also a solution if $\lambda\neq 0$. 
\end{proposition}

We note that 
\begin{equation}\label{eqn:critical_scaling}
\|\partial_\alpha\eta_\lambda\|_{L^\infty} = \|\partial_\alpha\eta\|_{L^\infty}\quad \text{for all}\ \lambda\neq 0,
\end{equation}
so $W^{1,\infty}(\T)$ is critical for the scaling of the interface equation \eqref{interface_eqn}.

\subsection{Properties of the Dirichlet-to-Neumann operator} We establish several properties of the Dirichlet-to-Neumann operator on star-shaped domains. Many of these could also be derived from the corresponding results for graph domains together with the equivalence between the graph and star-shaped Dirichlet-to-Neumann operators established in Appendix \ref{appendix:DN}. Nonetheless, proving them directly in the star-shaped setting provides a complementary perspective and better clarifies the underlying geometry.

First of all, we have the following basic symmetries of the Dirichlet-to-Neumann operator on star-shaped domains.
\begin{proposition}[Symmetries of Dirichlet-to-Neumann operator]\label{scaling_symm}
Let $\eta\in W^{1,\infty}(\T)$.
\begin{enumerate}
    \item (Scaling Symmetry) For $M_1, M_2\in\R$, we have
        \[G(e^{M_1}h)(\eta + M_2) = G(h)\eta \]
    \item (Rotation Symmetry) Let $\beta \in\T$. Suppose $\tilde h(\alpha) := h(\alpha - \beta)$ and $\tilde\eta := \log \tilde h$, then $G(\tilde h)\tilde\eta(\alpha) = G(h)\eta(\alpha - \beta)$.
\end{enumerate}
\end{proposition}

\begin{proof}
(2) is straightforward, so we only show (1). For simplicity of illustration we assume everything has enough regularity, but the proof can be readily adapted to a weak setting. Suppose $\phi(x, y)$ is the solution to the Dirichlet problem \eqref{eqn:Dirichlet_problem_star_shaped}. We note that $\tilde\phi(x, y) := \phi(e^{-M_1}x, e^{-M_1}y) + M_2$ solves the analogous Dirichlet boundary value problem with boundary parametrization $e^{M_1} h$ and boundary value $\eta + M_2$. At $e^{i\alpha}e^{M_1}h(\alpha)$, the unnormalized outward normal is
\begin{equation}
\begin{split}
\tilde N(e^{i\alpha}e^{M_1} h(\alpha)) &:= e^{i\alpha}e^{M_1} h(\alpha) - ie^{i\alpha} e^Mh'(\alpha) \\
                              &= e^{M_1} N(e^{i\alpha}h(\alpha)).
\end{split}
\end{equation}
We also have $\nabla\tilde\phi(e^{i\alpha}e^{M_1} h(\alpha)) = e^{-M_1}\nabla\phi(e^{i\alpha}h(\alpha))$, so it follows that
\begin{equation}
\begin{split}
G(e^{M_1} h)(\eta + M_2)(\alpha) &= \tilde N(e^{i\alpha}e^{M_1} h(\alpha))\cdot \nabla\tilde\phi(e^{i\alpha}e^{M_1} h(\alpha)) \\
                           &= N(e^{i\alpha}h(\alpha))\cdot \nabla\phi(e^{i\alpha}h(\alpha)) \\
                           &= G(h)\eta(\alpha)
\end{split}
\end{equation}
as desired.
\end{proof}

Next we present a few results that relies on maximum principle. Since we might work in lower boundary regularity than in classical regimes, we need a variational version of the maximum principle.

\begin{lemma}[Weak maximum principle]\label{lem:weak_max_principle}
Let $\Omega\subset\mathbb R^2$ be a bounded Lipschitz domain. Suppose
$u,v\in H^1(\Omega)$ satisfy
\begin{equation}\label{eq:weak-harmonic-u-v}
\int_\Omega \nabla u\cdot\nabla\varphi\,dx
=
\int_\Omega \nabla v\cdot\nabla\varphi\,dx=0
\quad
\forall\,\varphi\in H^1_0(\Omega)
\end{equation}
and $\tr\ u\le
\tr\ v$ a.e. on $\partial\Omega$, then $u\leq v$ a.e. in $\Omega$.
\end{lemma}

\begin{proof}
Set $w:=u-v$. Then
\begin{equation}\label{eq:weak-harmonic-w}
\int_\Omega \nabla w\cdot\nabla\varphi\,dx=0
\quad
\forall\,\varphi\in H^1_0(\Omega).
\end{equation}
and $\tr\ w \leq 0$. Hence $w^+:=\max\{w,0\}$ has zero trace and by truncation property of Sobolev functions $w^+\in H^1_0(\Omega)$. Taking
$\varphi=w^+$ in \eqref{eq:weak-harmonic-w} gives
\begin{equation}\label{eq:w-plus-energy-zero}
0=\int_\Omega \nabla w\cdot\nabla w^+\,dx
=\int_{\{w>0\}}|\nabla w|^2\,dx
=\int_\Omega |\nabla w^+|^2\,dx.
\end{equation}
and thus $w^+$ is constant on each connected component of $\Omega$. We then get $w^+\equiv0$ since it has zero trace, and the lemma immediately follows.
\end{proof}

We have the following pointwise comparison principle. The pointwise $C^{1,1}$ regularity is crucial in the analysis of viscosity solutions in section \ref{sec:visc_sol}. 
\begin{proposition}[Pointwise comparison principle] \label{pointwise_cp}
Let $h_1, h_2\in W^{1,\infty}(\T)$ be bounded away from the origin and we define $\eta_i = \log h_i$ for $i = 1, 2$. Suppose $\eta_1\leq \eta_2$, $\eta_1(\alpha_0) = \eta_2(\alpha_0)$, and $\eta_i$ is $C^{1,1}$ at $\alpha_0$ for $i = 1, 2$, then
\begin{equation}
G(h_1)\eta_1(\alpha_0)\geq G(h_2)\eta_2(\alpha_0).
\end{equation}
where both sides are classically well-defined.
\end{proposition}

\begin{proof}
Let $\phi_i$ be the harmonic extension of $\eta_i$ to $\Omega_{h_i}$ for $i = 1, 2$. We define
\[p_i(x, y) := \phi_i(x, y) - \frac{1}{2} \log(x^2 + y^2). \]
Then $p_i = 0$ on $\Sigma_{h_i}$ and $-\Delta p_i = 2\pi\delta$ on $\Omega_{h_i}$ for $i = 1, 2$, at least in a distributional sense. For convenience of notation, we write $N_i(\alpha) := N(e^{i\alpha}h_i(\alpha))$. 

Since $p_i$ is harmonic near $h_i(\alpha_0)e^{i\alpha_0}$ and vanishes on $\Sigma_{h_i}$, and $\Sigma_{h_i}$ is $C^{1,1}$ at $h_i(\alpha_0)e^{i\alpha_0}$, 
\begin{equation}
\begin{split}
G(h_i)\eta_i(\alpha_0) &= N_i(\alpha_0)\cdot \nabla\phi_i(e^{i\alpha_0}h_i(\alpha_0)) \\
                     &= N_i(\alpha_0)\cdot\nabla p_i(e^{i\alpha_0}h_i(\alpha_0)) + 1
\end{split}
\end{equation}
is classically well-defined (see, for example, Lemma 11.17 of \cite{CaffarelliSalsa2005} or Theorem 2.12 of \cite{Dong-Gancedo-Nguyen-23}). 

We first note that $p_i\geq 0$ in $\Omega_{h_i}$. Indeed, applying the maximum principle to $p_i$ in the punctured domain $\Omega_{h_i}\setminus B_\varepsilon(0)$ and then sending $\varepsilon\to 0$, using the fact that $p_i\to +\infty$ as $z\to 0$, gives the claim.

We have
\begin{equation}
G(h_1)\eta_1(\alpha_0) - G(h_2)\eta_2(\alpha_0) = N_1(\alpha_0)\cdot \nabla p_1(e^{i\alpha_0}h_1(\alpha_0)) - N_2(\alpha_0)\cdot \nabla p_2(e^{i\alpha_0} h_2(\alpha_0)).
\end{equation}
Note that $\alpha_0\in \T$ is a local maximum of $h_1 - h_2$, so $h_1'(\alpha_0) = h_2'(\alpha_0)$. Also by assumption $h_1(\alpha_0) = h_2(\alpha_0)$. Remember that $N_i(\alpha_0) = h_i(\alpha_0)e^{i\alpha_0} - ih_i'(\alpha_0)e^{i\alpha_0}$, so we have $N_1(\alpha_0) = N_2(\alpha_0)$. We denote both by $N(\alpha_0)$ for convenience. Then
\begin{equation}
G(h_1)\eta_1(\alpha_0) - G(h_2)\eta_2(\alpha_0) = N(\alpha_0)\cdot \nabla (p_1 - p_2)(e^{i\alpha_0}h_1(\alpha_0)).
\end{equation}
We note that $p_1 - p_2 = \phi_1 - \phi_2$ is harmonic within $\Omega_{h_1}$, and 
\begin{equation}
p_1 - p_2 = -p_2 \leq 0
\end{equation}
on $\Sigma_{h_1}$, so $p_1 - p_2\leq 0$ in $\Omega_{h_1}$ by maximum principle Lemma \ref{lem:weak_max_principle} applied to $p_1 - p_2$. Thus 
\begin{equation}
0\leq N(\alpha_0)\cdot \nabla (p_1 - p_2)(e^{i\alpha_0}h_1(\alpha_0)) = G(h_1)\eta_1(\alpha_0) - G(h_2)\eta_2(\alpha_0).
\end{equation}
as desired. Here the outward normal is well-defined since $\eta_i$ is $C^{1,1}$ at $\alpha_0$ for $i = 1, 2$.
\end{proof}

We next show the Taylor sign condition, which plays a crucial role for us to establish the comparison principle of the interface equation \eqref{interface_eqn}. Such Taylor sign condition was first observed in the study of water wave equations by Wu \cites{Wu1997WaterWaves2D, Wu1999WaterWaves3D}.

\begin{proposition}[Taylor sign condition]\label{taylor_sign}
Suppose $\eta\in \lip(\T)$ and $\eta$ is $C^{1,1}$ at $\alpha_0$. Let $h = e^\eta$. Then 
\begin{equation}
G(h)\eta - 1 \leq 0\quad \mathrm{at}\ \alpha_0
\end{equation}
\end{proposition}

\begin{proof}
Let $\phi$ be a variational solution to the Dirichlet boundary value problem 
\begin{equation}
\left\{
\begin{array}{ll}
 \Delta\phi = 0 & \mathrm{in}\ \Omega_{h}  \\
 \phi = \eta & \mathrm{on}\ \Sigma_h
\end{array}
\right.
\end{equation}
which we know is well-defined since the boundary and $\eta$ are both Lipschitz. We consider
\[p(x, y) := \phi(x, y) - \frac{1}{2}\log(x^2 + y^2)\]
Then $p$ solves the Dirichlet boundary value problem
\begin{equation}
\left\{
\begin{array}{ll}
 -\Delta p = 2\pi\delta & \mathrm{in}\ \Omega_{h}  \\
 p = 0 & \mathrm{on}\ \Sigma_h
\end{array}
\right.
\end{equation}
at least in a distributional sense. First of all, since $\eta$ is $C^{1, 1}$ at $\alpha_0$, 
\begin{equation}
\begin{aligned}
\frac{\partial p}{\partial N}(e^{i\alpha_0}h(\alpha_0)) &= \frac{\partial\phi}{\partial N}(e^{i\alpha_0}h(\alpha_0)) - N(\alpha_0)\cdot \frac{e^{i\alpha_0}}{h(\alpha_0)}\\
&= G(h)\eta(\alpha_0) - 1
\end{aligned}
\end{equation}
is classically well-defined at $e^{i\alpha_0}h(\alpha_0)$. Based on the definition of $p$, we can find some small enough $\varepsilon > 0$ such that $\overline{B(0, \varepsilon)}\subset \Omega_h$, and $p \geq 1$ on $\partial B(0, \varepsilon)$. Now, $p$ is (distributionally) harmonic on $\Omega_h\bs \overline{B(0,\varepsilon)}$, $p \geq 1$ on $\partial B(0, \varepsilon)$, and $p = 0$ on $\Sigma_h$. It follows from maximum principle Lemma \ref{lem:weak_max_principle} that $\alpha_0$ must be a local minimum, and we have
\begin{equation}
\frac{\partial p}{\partial N}(e^{i\alpha_0}h(\alpha_0))\leq 0
\end{equation}
as desired.
\end{proof}

Layer-potential theory gives rise to $L^2$-boundedness of the Dirichlet-to-Neumann operator. We have the following bound of $(\frac{1}{2}I \pm K)^{-1}$ on $L^2(\T)$ as a corollary of Theorem \ref{thm:quantitative_bound_layer_potential}. Note that the constants here are by no means optimized.

\begin{proposition}\label{inverse_layer_potential_star_shape}
Suppose $\partial\Omega = \{e^{\eta(\alpha)}e^{i\alpha}:\alpha\in \T\}$, where $\eta\in \lip(\T)$, and $K$ be the corresponding boundary double-layer potential. Then
\begin{equation}
\bigg\|\bigg(\frac{1}{2}I \pm K^*\bigg)^{-1} f\bigg\|_{L^2}\leq Ce^{4(1+\|\eta\|_{L^\infty})}(1 + \|\partial_\alpha\eta\|_{L^\infty}^2)^{\frac{3}{2}}\|f\|_{L^2}.
\end{equation}
\end{proposition}

\begin{proof}
Write $g := (\frac{1}{2}I \pm K^*)^{-1}f$. We also define $\tilde g$ and $\tilde f$ such that $\tilde g(e^{\eta(\alpha)}e^{i\alpha}) = g(\alpha)$ and $\tilde f(e^{\eta(\alpha)}e^{i\alpha}) = f(\alpha)$, 
so we have the identification $(\frac{1}{2}I \pm K^*)\tilde{g} = \tilde{f}$. We choose $V(re^{i\alpha}) = e^{i\alpha}\chi(r)$, where $\chi$ is a smooth cutoff function satisfying $\chi = 1$ for $r\leq e^{\|\eta\|_{L^\infty}} + 1$ and $\chi = 0$ for $r > e^{\|\eta\|_{L^\infty}} + 2$. Then  $\|V\|_{C^1}\leq C$ for some absolute constant $C$, $V\cdot n = e^{2\eta}(1 + (\partial_\alpha\eta)^2)^{-\frac{1}{2}}$, and 
\begin{equation}
|\partial\Omega|\leq Ce^{\|\eta\|_{L^\infty}}(1 + \|\partial_\alpha\eta\|_{L^\infty}^2)^{\frac{1}{2}}.
\end{equation}
Therefore, by Theorem \ref{thm:quantitative_bound_layer_potential} we have
\begin{equation}
\|\tilde g\|_{L^2(\partial\Omega)}\leq Ce^{2\|\eta\|_{L^\infty}}(1 + \|\partial_\alpha\eta\|_{L^\infty}^2)^{\frac{1}{2}}\big(1 + e^{2\|\eta\|_{L^\infty}}(1 + \|\partial_\alpha\eta\|_{L^\infty}^2)\big)^{\frac{1}{2}}\|\tilde f\|_{L^2(\partial\Omega)},
\end{equation}
and a straightforward change of variable gives
\begin{equation}
\begin{aligned}
\|g\|_{L^2} &\leq Ce^{4\|\eta\|_{L^\infty}}(1 + \|\partial_\alpha\eta\|_{L^\infty}^2)\big(1 + e^{2\|\eta\|_{L^\infty}}(1 + \|\partial_\alpha\eta\|_{L^\infty}^2)\big)^{\frac{1}{2}}\|f\|_{L^2}\\
&\leq Ce^{6\|\eta\|_{L^\infty} + 1}(1 + \|\partial_\alpha\eta\|_{L^\infty}^2)^{\frac{3}{2}}\|f\|_{L^2},
\end{aligned}
\end{equation}
as desired.
\end{proof}

Following the proof of Proposition \ref{inverse_layer_potential_star_shape}, we also have the following result on the $H^1(\T)\to L^2(\T)$ boundedness of the Dirichlet-to-Neumann operator.

\begin{proposition}[Order one $L^2$-boundedness]\label{prop:DN_order_1_bdd}
Suppose $f\in \lip(\T)$ and $\partial_\alpha g\in L^2(\T)$, then
\begin{equation}
\|G(e^f)g\|_{L^2}\leq Ce^{10(1+\|f\|_{L^\infty})}(1 + \|\partial_\alpha f\|_{L^\infty}^2)^{\frac{7}{2}}\|\partial_\alpha g\|_{L^2}
\end{equation}
\end{proposition}
\begin{proof}
For convenience of notation, we denote $H := \cS\theta$, where $\theta = (\frac{1}{2}I - K^*)^{-1}\partial_\alpha g$. We recall from \eqref{integral_dn} that 
\begin{equation}
G(e^f)g(\alpha) = (i + f'(\alpha))e^{i\alpha + f(\alpha)}\cdot \nabla H(e^{i\alpha+f(\alpha)}),
\end{equation}
so it follows that
\begin{equation}
\begin{aligned}
\|G(e^f)g\|_{L^2}^2 &\lesssim e^{2\|f\|_{L^\infty}}(1 + \|\partial_\alpha f\|_{L^\infty}^2)\int_\T |\partial_\tau H(e^{f(\alpha) + i\alpha})|^2\ d\alpha \\
&\leq e^{3\|f\|_{L^\infty}}(1 + \|\partial_\alpha f\|_{L^\infty}^2)\int_\T|\partial_\tau H(e^{f(\alpha) + i\alpha})|^2e^{f(\alpha)}\sqrt{1 + \partial_\alpha f(\alpha)^2}\ d\alpha \\
&\leq e^{3\|f\|_{L^\infty}}(1 + \|\partial_\alpha f\|_{L^\infty}^2)\|\partial_\tau H\|_{L^2(\partial\Omega)}^2.
\end{aligned}
\end{equation}
Choosing the vector field $V$ as in the proof of Proposition \ref{inverse_layer_potential_star_shape}, we know from \eqref{tangential_normal} that 
\begin{equation}
\begin{aligned}
\|\partial_\tau H\|_{L^2(\partial\Omega)} &\leq Ce^{4(1+\|f\|_{L^\infty})}(1 + \|\partial_\alpha f\|_{L^\infty}^2)^{\frac{3}{2}} \|\partial_nH\|_{L^2(\partial\Omega)}\\
&\leq Ce^{4(1+\|f\|_{L^\infty})}(1 + \|\partial_\alpha f\|_{L^\infty}^2)^{\frac{3}{2}} \bigg\|\bigg(-\frac{1}{2}I + K^*\bigg)\theta\bigg\|_{L^2}\\
&\leq Ce^{8(1+\|f\|_{L^\infty})}(1 + \|\partial_\alpha f\|_{L^\infty}^2)^{3}\|\partial_\alpha g\|_{L^2},
\end{aligned}
\end{equation}
where the second inequality follows from Theorem \ref{thm:jump_relations}, and the last inequality follows from Proposition \ref{inverse_layer_potential_star_shape}. In summary, we have
\begin{equation}
\|G(e^f)g\|_{L^2}\leq Ce^{10(1+\|f\|_{L^\infty})}(1 + \|\partial_\alpha f\|_{L^\infty}^2)^{\frac{7}{2}}\|\partial_\alpha g\|_{L^2}
\end{equation}
as desired.
\end{proof}

\section{Viscosity Regularization of the Interface Equation}\label{sec:visc_reg}
The overall framework of this section is inspired by \cite{Dong-Gancedo-Nguyen-23}. In this section, we study the following viscosity regularization of the interface equation, where $\varepsilon\geq 0$:
\begin{equation}\label{visc_interface_eqn}
\partial_t\eta^\varepsilon + {e^{-2\eta^\varepsilon}}(G(h^\varepsilon)\eta^\varepsilon - 1) - \varepsilon\partial_\alpha^2\eta^\varepsilon = 0,\qquad h^\varepsilon = e^{\eta^\varepsilon}
\end{equation}
Here we will primarily use the expressions
\begin{equation}\label{DN_mollified}
\begin{split}
G(h^\varepsilon)\eta^\varepsilon(\alpha, t) = \pv\frac{1}{2\pi} \re\ \partial_\alpha\int_\T \log(h^\varepsilon(\alpha)e^{i\alpha} - h^\varepsilon(\beta)e^{i\beta})\, \theta^\varepsilon(\beta)\ d\beta,
\end{split}
\end{equation}
\begin{equation}\label{theta_eqn_visc}
\begin{split}
\frac{1}{2}\theta^\varepsilon(\alpha) + \pv\frac{1}{2\pi}\im\ \partial_\alpha \int_\T \log(h^\varepsilon(\alpha)e^{i\alpha} - h^\varepsilon(\beta)e^{i\beta})\,\eta^\varepsilon(\beta)\ d\beta = \partial_\alpha \eta^\varepsilon(\alpha).
\end{split}
\end{equation}

\begin{definition}\label{def:classical_sol}
Let $h^\varepsilon:= e^{\eta^\varepsilon}$. Then $\eta^\varepsilon$ is a classical subsolution (supersolution) of (\ref{visc_interface_eqn}) on $[0, T]\times \T$, where $0<T\leq \infty$, if $\eta$ is $C^1$ in time and $C^2$ in space, and
\begin{equation}
\partial_t\eta^\varepsilon + {e^{-2\eta^\varepsilon}}(G(h^\varepsilon)\eta^\varepsilon - 1) - \varepsilon\partial_\alpha^2\eta^\varepsilon \leq  0\quad (\geq 0) 
\end{equation}
on $[0, T]\times \T$. $\eta^\varepsilon$ is a classical solution if it is both a classical subsolution and a classical supersolution.
\end{definition}

\subsection{Comparison principles} The advantage of the viscosity regularization \eqref{DN_mollified} is that it preserves the comparison principle and maximum principles of the unregularized equation. 

\begin{proposition}[Comparison principle]\label{visc_cp}
Let $\eta_-^\varepsilon$ be a subsolution of \eqref{visc_interface_eqn} and $\eta_+^\varepsilon$ be a supersolution of \eqref{visc_interface_eqn} on $[0,T]$. If $\eta_-^\varepsilon(\cdot, 0)\leq \eta_+^\varepsilon(\cdot, 0)$, then $\eta_-^\varepsilon(\cdot, t)\leq \eta_+^\varepsilon(\cdot, t)$ for all $t\in [0, T]$.
\end{proposition}

\begin{proof}
For convenience of notation, we suppress the superscript $\varepsilon$. Suppose on the contrary that
\begin{equation}
\max_{\T\times [0,T]} (\eta_- - \eta_+) = M_0 > 0
\end{equation}
is achieved at $(\alpha_0, t_0)$. Let $c > 0$ be sufficiently small such that
\begin{equation}
\eta_-(\alpha_0, t_0) - \eta_+(\alpha_0, t_0) - ct_0 > 0
\end{equation}
and we consider $\eta_- - ct - \eta_+$. Suppose $\eta_- - ct - \eta_+$ achieves its global maximum $M_*$ at $(\alpha_*, t_*)$. Note that $M_* > 0$, $t_* > 0$, and
\begin{equation}\label{comparison_global_max}
\eta_-(\alpha_*, t_*) - ct_* - M_* = \eta_+(\alpha_*, t_*).
\end{equation}
First of all, note that
\begin{equation}
\begin{split}
0 &\leq \partial_t (\eta_- - ct - M_* - \eta_+)(\alpha_*, t_*) \\
  & = \partial_t\eta_-(\alpha_*, t_*) - \partial_t\eta_+(\alpha_*, t_*) - c,
\end{split}
\end{equation}
which gives rise to 
\begin{equation}\label{fifth_eqn}
\partial_t \eta_- (\alpha_*, t_*) - \partial_t\eta_+(\alpha_*, t_*) \geq c.
\end{equation}
Since $\eta_-(\cdot, t_*) - ct_* - M_*\leq \eta_+(\cdot, t_*)$ and the equality holds at $\alpha_*$, so by pointwise comparison principle (Proposition \ref{pointwise_cp}) and scaling symmetry (Proposition \ref{scaling_symm}) of the Dirichlet-to-Neumann operator, we have
\begin{equation}\label{second_eqn}
\begin{split}
G(h_-)\eta_-(\alpha_*, t_*) &= G(e^{-ct_* - M_*}h_-)(\eta_- - ct_* - M_*)(\alpha_*, t_*)  \\
                            &\geq G(h_+)\eta_+(\alpha_*, t_*).
\end{split}
\end{equation}
At the mean time, by \eqref{comparison_global_max} we have
\begin{equation}
e^{2ct_* + M_*}e^{-2\eta_-} = e^{-2\eta_+}\quad\mathrm{at}\ (\alpha_*, t_*),
\end{equation}
and in particular
\begin{equation}\label{third_eqn}
0<e^{-2\eta_-}\leq e^{-2\eta_+}\quad\mathrm{at}\ (\alpha_*, t_*).
\end{equation}
In addition, we have the Taylor sign condition (Proposition \ref{taylor_sign})
\begin{equation}\label{fourth_eqn}
G(h_\pm)\eta_\pm - 1\leq 0.
\end{equation}
Using (\ref{second_eqn}), (\ref{third_eqn}), and (\ref{fourth_eqn}), we conclude that
\begin{equation}\label{sixth_eqn}
\begin{split}
e^{-2\eta_-}(G(h_-)\eta_- - 1)(\alpha_*, t_*) &\geq e^{-2\eta_+}(G(h_-)\eta_- - 1)(\alpha_*, t_*) \\
                                 &\geq e^{-2\eta_+}(G(h_+)\eta_+ - 1)(\alpha_*, t_*)
\end{split}
\end{equation}
Eventually, since $\alpha_*$ is a local max of $\eta_- - \eta_+ (\cdot, t_*)$, we have
\begin{equation}\label{seventh_eqn}
-\varepsilon \partial_\alpha^2 \eta_-(\alpha_*, t_*) \geq - \varepsilon\partial_\alpha^2\eta_+(\alpha_*, t_*)
\end{equation}
Now, combining inequalities (\ref{fifth_eqn}), (\ref{sixth_eqn}), and (\ref{seventh_eqn}), we have $c\leq 0$, leading to a contradiction. Our proof is thus complete. 
\end{proof}

\begin{proposition}[Preservation of modulus of continuity]\label{visc_reg_mod_cty}
Let $\omega:[0,\infty)\to[0,\infty)$ be a modulus of continuity, i.e.
$\omega$ is continuous, nondecreasing, and $\omega(0)=0$. Let
$\eta^\varepsilon$ be a classical solution of \eqref{visc_interface_eqn} on
$[0,T]$. If $\eta^\varepsilon(\cdot,0)$ has modulus of continuity $\omega$,
namely
\begin{equation}\label{modulus_of_cty_ineq}
-\omega(|\beta|)\leq \eta^\varepsilon(\alpha+\beta,0)-\eta^\varepsilon(\alpha,0)
\leq \omega(|\beta|)
\end{equation}
for all $\alpha,\beta\in\mathbb R$, then for every $t\in[0,T]$ we have
\begin{equation}
-\omega(|\beta|)\leq \eta^\varepsilon(\alpha+\beta,t)-\eta^\varepsilon(\alpha,t)
\leq \omega(|\beta|)
\end{equation}
for all $\alpha,\beta\in\mathbb R$. In particular, for any $0<\gamma\leq 1$, if
$\eta^\varepsilon(\cdot,0)\in C^\gamma(\mathbb T)$, then
\begin{equation}\label{eqn:mod_cty}
[\eta^\varepsilon(\cdot,t)]_{C^\gamma(\mathbb T)}
\leq
[\eta^\varepsilon(\cdot,0)]_{C^\gamma(\mathbb T)}
\qquad\text{for all } t\in[0,T].
\end{equation}
\end{proposition}

\begin{proof}
For convenience of notation, we suppress the superscript $\varepsilon$. Fix $\beta>0$. We claim that
\[
\tilde\eta(\alpha,t):=\eta(\alpha+\beta,t)-\omega(\beta)
\]
is a classical subsolution of \eqref{visc_interface_eqn}. Denote
\[
\tilde h(\alpha,t):=e^{\tilde\eta(\alpha,t)}
=e^{-\omega(\beta)}h(\alpha+\beta,t).
\]
We have
\begin{equation}\label{modulus_time_shift}
\partial_t\tilde\eta(\alpha,t)=\partial_t\eta(\alpha+\beta,t),
\qquad
\partial_\alpha^2\tilde\eta(\alpha,t)=\partial_\alpha^2\eta(\alpha+\beta,t).
\end{equation}
Moreover, by the translation invariance and scaling symmetry of the
unnormalized Dirichlet-to-Neumann operator, we obtain
\begin{equation}\label{modulus_dn_shift}
G(\tilde h)\tilde\eta(\alpha,t)=G(h)\eta(\alpha+\beta,t).
\end{equation}
Therefore,
\begin{equation}
\begin{split}
&\partial_t\tilde\eta(\alpha,t)
+e^{-2\tilde\eta(\alpha,t)}\bigl[G(\tilde h)\tilde\eta(\alpha,t)-1\bigr]
-\varepsilon\partial_\alpha^2\tilde\eta(\alpha,t) \\
&=
\partial_t\eta(\alpha+\beta,t)
+e^{2\omega(\beta)}e^{-2\eta(\alpha+\beta,t)}
\bigl[G(h)\eta(\alpha+\beta,t)-1\bigr]
-\varepsilon\partial_\alpha^2\eta(\alpha+\beta,t).
\end{split}
\end{equation}
Since $G(h)\eta-1\le 0$ pointwise by Proposition \ref{taylor_sign} and
$e^{2\omega(\beta)}\ge 1$, the right-hand side is bounded above by
\[
\partial_t\eta(\alpha+\beta,t)
+e^{-2\eta(\alpha+\beta,t)}\bigl[G(h)\eta(\alpha+\beta,t)-1\bigr]
-\varepsilon\partial_\alpha^2\eta(\alpha+\beta,t)
=0.
\]
Hence $\tilde\eta$ is a classical subsolution of
\eqref{visc_interface_eqn}. By \eqref{modulus_of_cty_ineq}, we have
\[
\tilde\eta(\alpha,0)
=
\eta(\alpha+\beta,0)-\omega(\beta)
\leq
\eta(\alpha,0).
\]
Applying the comparison principle (Proposition \ref{visc_cp}) to
$\tilde\eta$ and $\eta$, we obtain
\[
\eta(\alpha+\beta,t)-\eta(\alpha,t)\leq \omega(\beta)
\qquad\text{for all } \alpha\in\mathbb R,\ t\in[0,T].
\]

To prove the reverse inequality, we argue exactly as above, showing that $\alpha\mapsto \eta(\alpha + \beta, t) + \omega(\beta)$ is a classical supersolution. Then another application of the comparison principle yields
\[
-\omega(\beta)\leq \eta(\alpha+\beta,t)-\eta(\alpha,t).
\]

Combining the two inequalities gives the desired estimate for all $\beta>0$.
The case $\beta<0$ follows by replacing $\beta$ with $-\beta$. This completes
the proof.
\end{proof}

We have the following important corollaries that control the Lipschitz norm of $\eta(\cdot, t)$ in terms of $\eta_0$. The first follows from comparison with expanding circular solutions, and the second is a special case of \eqref{eqn:mod_cty} when $\gamma = 1$.
\begin{corollary}[Control of $L^\infty$]\label{cor_Linfty_bound}
Suppose $\eta^\varepsilon$ is a classical solution to \eqref{visc_interface_eqn} and $h^\varepsilon := e^{\eta^\varepsilon}$. If $0< r_0\leq h^\varepsilon(\cdot, 0)\leq R_0 <\infty$, then
\begin{equation}\label{h_l_infty}
\sqrt{r_0^2 + 2t} \leq h^\varepsilon(\cdot, t)\leq \sqrt{R_0^2 + 2t} 
\end{equation}
In particular, if $\eta^\varepsilon(\cdot, 0) = \eta_0^\varepsilon$, then
\begin{equation}\label{L_infty_bound}
|\eta^\varepsilon(\cdot, t)|\leq \log\sqrt{\exp(2\|\eta_0^\varepsilon\|_{L^\infty}) + 2t}
\end{equation}
\end{corollary}
for $t\in [0, T]$. 

\begin{corollary}[Control of Lipschitz constant]\label{cor_Lipschitz_bound}
Suppose $\eta^\varepsilon$ is a classical solution to \eqref{visc_interface_eqn}. Then
\begin{equation}\label{Lipschitz_bound}
\|\partial_\alpha\eta^\varepsilon(\cdot, t)\|_{L^\infty(\T)}\leq \|\partial_\alpha\eta^\varepsilon(\cdot, 0)\|_{L^\infty(\T)}
\end{equation}
for $t\in [0, T]$.
\end{corollary}

The maximum principle of $\|\partial_\alpha\eta(\cdot, t)\|_{L^\infty}$ has a natural geometric
interpretation as preservation of a uniform cone condition. Indeed, if
$z(\alpha,t)=e^{\eta(\alpha,t)}e^{i\alpha}$, $e_r:= e^{i\alpha}$, then the outward normal $n(\alpha, t)$ on $z(\alpha, t)$ satisfies
\[
n(\alpha,t)\cdot e_r(\alpha)
=\frac{1}{\sqrt{1+\eta_\alpha(\alpha,t)^2}}.
\]
Therefore, if $\|\partial_\alpha\eta(\cdot, t)\|_{L^\infty}\le L_0$, then
\[
n(\alpha,t)\cdot e_r(\alpha)
\ge \frac{1}{\sqrt{1+L_0^2}}
=\cos(\arctan L_0).
\]

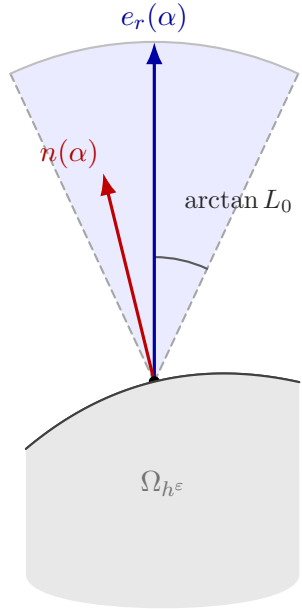
\begin{figure}[ht]
\centering
\begin{tikzpicture}[
  scale=3, >=Latex, line cap=round, line join=round, thick,
]
  \def\th{90}    
  \def\ph{25}    
  \def\R{1.5}
  \def\rA{0.55}

  \coordinate (P)  at (0,0);
  \coordinate (eR) at ({\R*cos(\th)},{\R*sin(\th)});
  \coordinate (BL) at ({\R*cos(\th+\ph)},{\R*sin(\th+\ph)});
  \coordinate (BR) at ({\R*cos(\th-\ph)},{\R*sin(\th-\ph)});
  \pgfmathsetmacro{\thn}{\th + 0.55*\ph}             
  \coordinate (N) at ({0.95*cos(\thn)},{0.95*sin(\thn)});

  \fill[blue!8] (P) -- (BL)
      arc[start angle=\th+\ph, end angle=\th-\ph, radius=\R] -- cycle;
  \draw[densely dashed, gray!70] (P) -- (BL);
  \draw[densely dashed, gray!70] (P) -- (BR);
  \draw[gray!50] (BR) arc[start angle=\th-\ph, end angle=\th+\ph, radius=\R];

  \draw[gray!70!black]
      ({\rA*cos(\th)},{\rA*sin(\th)})
      arc[start angle=\th, end angle=\th-\ph, radius=\rA];
  \node[gray!30!black]
      at ({1.60*\rA*cos(\th-\ph)},{1.60*\rA*sin(\th-\ph)})
      {\small $\arctan L_0$};

  \draw[->, very thick, blue!65!black] (P) -- (eR)
      node[above, blue!60!black, inner sep=3pt] {$e_r(\alpha)$};
  \draw[->, very thick, red!75!black]  (P) -- (N)
      node[above left, red!70!black, inner sep=1pt] {$n(\alpha)$};

  \fill (P) circle (0.7pt);

  \pgfmathsetmacro{\thn}{\th + 0.55*\ph}
  \pgfmathsetmacro{\ttan}{\thn - 90}        
  \coordinate (N) at ({0.95*cos(\thn)},{0.95*sin(\thn)});

  ...   

  \begin{scope}[shift={(P)}, rotate=\ttan]
    \draw[thick, black!75, smooth, samples=50, domain=-0.62:0.62]
        plot ({\x},{-0.40*\x*\x});
  \end{scope}

    \pgfmathsetmacro{\ct}{cos(\ttan)}
    \pgfmathsetmacro{\st}{sin(\ttan)}
    \pgfmathsetmacro{\PRx}{ 0.62*\ct + 0.40*0.62*0.62*\st}
    \pgfmathsetmacro{\PRy}{ 0.62*\st - 0.40*0.62*0.62*\ct}
    \pgfmathsetmacro{\PLx}{-0.62*\ct + 0.40*0.62*0.62*\st}
    \pgfmathsetmacro{\PLy}{-0.62*\st - 0.40*0.62*0.62*\ct}
    
    \fill[gray!18]
        plot[domain=-0.62:0.62, samples=50, smooth]
            ({\x*\ct + 0.40*\x*\x*\st},{\x*\st - 0.40*\x*\x*\ct})
        -- (\PRx,-0.85)
        .. controls (\PRx,-1.05) and (\PLx,-1.05) .. (\PLx,-0.85)
        -- cycle;
    \begin{scope}[shift={(P)}, rotate=\ttan]
      \draw[thick, black!75, smooth, samples=50, domain=-0.62:0.62]
          plot ({\x},{-0.40*\x*\x});
    \end{scope}
    \node[black!55] at (0.04,-0.45) {$\Omega_{h^\varepsilon}$};
\end{tikzpicture}
\caption{Uniform Cone Condition}
\label{fig:cone_condition}
\end{figure}

Equivalently, at every boundary point, the outward normal remains inside the
fixed cone of half-angle $\arctan L_0$ around the radial direction. Hence the
interface stays uniformly star-shaped: it remains uniformly transverse to every
ray from the origin, with the same cone aperture as initially.

\subsection{Global well-posedness of the regularized equations} We next prove that, once the regularization is introduced, equation \eqref{visc_interface_eqn} is globally well-posed for sufficiently smooth initial data in Sobolev spaces. 
\begin{proposition}[Global well-posedness of regularized equation]\label{mollified_gwp}
Let $\varepsilon>0$ and $s\geq 2$. For each initial data $\eta_0^\varepsilon\in H^s(\T)$, equation \eqref{visc_interface_eqn} admits a unique solution $\eta^\varepsilon\in C([0,\infty); H^s(\T))\cap L^2_{\mathrm{loc}}([0,\infty); H^{s+1}(\T))$. Moreover, $\eta^\varepsilon$ satisfies the bounds \eqref{L_infty_bound} and \eqref{Lipschitz_bound}.
\end{proposition}

The proof proceeds by a continuation argument. We first establish local well-posedness in \(H^s\), along with the usual blowup alternative: if the maximal existence time is finite, then \(\|\eta(t)\|_{H^s}\) must diverge as \(t\) approaches it. We then derive a priori estimates showing that \(\|\eta(t)\|_{H^s}\) stays bounded on every finite time interval. This excludes finite-time blowup and hence implies that the solution exists globally.

\subsubsection{Local well-posedness} Fix $s\geq 2$. We first outline the argument of local well-posedness in $H^s(\T)$ with maximal time of existence $T_*>0$.
\begin{proposition}[Local well-posedness]
Let \(s\ge 2\) and \(\eta_0^\varepsilon\in H^s(\T)\). Then there exists a maximal existence time \(T^*\in(0,\infty]\) such that, for every \(T<T^*\), equation \eqref{visc_interface_eqn} admits a unique solution
\[
\eta^\varepsilon\in C([0,T];H^s(\T))\cap L^2([0,T];H^{s+1}(\T))\cap C^\infty((0,T)\times\T).
\]
Moreover, if \(T^*<\infty\), then the blowup alternative holds:
\begin{equation}\label{eqn:blow_up}
\lim_{t\to T^*_-}\|\eta^\varepsilon(t)\|_{H^s}=\infty.
\end{equation}
\end{proposition}

Since $\varepsilon>0$, equation \eqref{visc_interface_eqn} is semilinear parabolic, and local well-posedness in $H^s$ follows from standard theory for semilinear parabolic equations as in \cite{TaylorPDE3}, together with the Dirichlet-to-Neumann estimates collected in Appendix \ref{appendix:DN}. 

\subsubsection{Standing notation and estimates} For convenience of notation, we often drop the dependence of $\varepsilon$. We write $\mathcal F$ for a generic continuous nonnegative function that is nondecreasing in all of its arguments. Its precise form may vary from line to line.

We also record several estimates that will be used repeatedly below. For Sobolev-embedding-type inequalities, see Appendix \ref{subsec:useful_ineq}. From the layer-potential estimates established earlier (see Proposition \ref{inverse_layer_potential_star_shape} and Proposition \ref{prop:DN_order_1_bdd}), we have
\begin{align}
\|\theta(t)\|_{L^2} &\leq \mathcal F(\|\eta(t)\|_{W^{1,\infty}})\|\partial_\alpha\eta(t)\|_{L^2}\leq \mathcal F(\|\eta(t)\|_{W^{1,\infty}}),\\
\|G(h)\eta(t)\|_{L^2} &\leq \mathcal F(\|\eta(t)\|_{W^{1,\infty}})\|\partial_\alpha\eta(t)\|_{L^2}\leq \mathcal F(\|\eta(t)\|_{W^{1,\infty}}).
\end{align}
We also have the maximum-principle-type bounds from Corollary \ref{cor_Linfty_bound} and Corollary \ref{cor_Lipschitz_bound}:
\begin{align}
\|\eta(t)\|_{L^\infty}&\leq \log\sqrt{\exp(2\|\eta_0\|_{L^\infty}) + 2t},\\
\|\partial_\alpha\eta(t)\|_{L^\infty} &\leq \|\partial_\alpha\eta_0\|_{L^\infty}.
\end{align}
Strictly speaking, Corollary \ref{cor_Linfty_bound} and Corollary \ref{cor_Lipschitz_bound} were established for solutions that are \(C^1\) in time and \(C^2\) in space. This is harmless here: since \(\varepsilon>0\), solutions to \eqref{visc_interface_eqn} are instantly smooth for positive times, so the bounds apply on \([\tau,T]\) for every \(\tau>0\). Passing to the limit \(\tau\to0^+\) using \(H^s(\T)\hookrightarrow C^1(\T)\) gives the desired estimates on \([0,T]\).

\subsubsection{$H^1$ estimates} We note that $L^2$ estimates follow from $L^\infty$ control of $\eta$, so we proceed to $\dot H^1$ estimates. Multiplying \eqref{visc_interface_eqn} by $-\partial_\alpha^2\eta$ and integrating by parts, we get
\begin{equation}
\begin{split}
\frac{1}{2}\frac{d}{dt}\|\partial_\alpha\eta\|_{L^2}^2 + \varepsilon\|\partial_\alpha^2\eta\|_{L^2}^2 &\leq \|\partial_\alpha^2\eta\|_{L^2}\|e^{-2\eta}(G(h)\eta - 1)\|_{L^2} \\
&\leq \|\partial_\alpha^2\eta\|_{L^2}\cF(\|\eta\|_{W^{1,\infty}})(1 + \|\partial_\alpha\eta\|_{L^2}),\\
\end{split}
\end{equation}
and a standard AM-GM inequality argument gives
\begin{equation}
\frac{d}{dt}\|\partial_\alpha\eta\|_{L^2}^2 + \varepsilon\|\partial_\alpha^2\eta\|_{L^2}^2 \leq \varepsilon^{-1}\cF(\|\eta\|_{W^{1,\infty}})(1 + \|\partial_\alpha\eta\|_{L^2}^2).
\end{equation}
By Gronwall's inequality \ref{Gronwall}, we then have
\begin{equation}
\begin{split}
&\|\partial_\alpha\eta\|_{L^2}^2(t) + \varepsilon\int_0^t\|\partial_\alpha^2\eta\|_{L^2}^2(s)\ ds\\
&\leq  \exp\bigg(\varepsilon^{-1}\int_0^t\cF(\|\eta\|_{W^{1,\infty}}(s))\ ds\bigg)\bigg(\|\partial_\alpha\eta_0\|_{L^2}^2 + \varepsilon^{-1}\int_0^t\cF(\|\eta\|_{W^{1,\infty}}(s))\ ds\bigg),
\end{split}
\end{equation}
and thus
\begin{equation}\label{H^1_Gronwall}
\|\partial_\alpha\eta\|_{L^2}^2(t) + \varepsilon\int_0^t\|\partial_\alpha^2\eta\|_{L^2}^2(s)\ ds \leq \cF(\|\eta_0\|_{W^{1,\infty}}, \varepsilon^{-1}, t),
\end{equation}
where the right-hand side is in good control. 

\subsubsection{$\dot H^2$ estimates} We differentiate equation \eqref{visc_interface_eqn} with respect to $\alpha$ and get 
\begin{equation}\label{H^2_eqn}
\partial_t\partial_\alpha\eta -2\partial_\alpha\eta e^{-2\eta}(G(h)\eta - 1) + e^{-2\eta}\partial_\alpha G(h)\eta - \varepsilon\partial_\alpha^3\eta = 0.
\end{equation}
Multiplying both sides of equation \eqref{H^2_eqn} by $\partial_\alpha^3\eta$ and integrating by parts, we get
\begin{equation}\label{H^2_ineq_1}
\frac{1}{2}\frac{d}{dt}\|\partial_\alpha^2\eta\|_{L^2}^2 + \varepsilon\|\partial_\alpha^3\eta\|_{L^2}^2 \leq \cF(\|\eta\|_{W^{1,\infty}})\|\partial_\alpha^2\eta\|_{L^2}(1 + \|\partial_\alpha G(h)\eta\|_{L^2}).
\end{equation}
To estimate $\|\partial_\alpha G(h)\eta\|_{L^2}$, we need the following lemma. We note that the same kind of estimates appear in Lemma 5.4 of \cite{Dong-Gancedo-Nguyen-23}, but we prove it using our layer-potential formulation \eqref{DN_mollified} and \eqref{theta_eqn_visc} on star-shaped domains.

\begin{lemma}\label{grad_estimate}
For any $\nu > 0$, we have
\begin{equation}\label{grad_estimate_ineq}
\|\partial_\alpha G(h)\eta\|_{L^2}\leq \nu\|\partial_\alpha^3\eta\|_{L^2} + \frac{1}{\nu}\cF(\|\eta\|_{W^{1,\infty}})(1 + \|\partial_\alpha^2\eta\|_{L^2} + \|\partial_\alpha^2\eta\|_{L^2}^2).
\end{equation}
\end{lemma}

Before proving Lemma \ref{grad_estimate}, we explain how it yields the desired \(\dot H^2\) control of \(\eta\). Combining \eqref{H^2_ineq_1} and \eqref{grad_estimate_ineq}, we obtain
\begin{equation}
\begin{aligned}
\frac12\frac{d}{dt}\|\partial_\alpha^2\eta\|_{L^2}^2+\varepsilon\|\partial_\alpha^3\eta\|_{L^2}^2
&\leq \mathcal F(\|\eta\|_{W^{1,\infty}})\|\partial_\alpha^2\eta\|_{L^2}
\bigl(1+\nu\|\partial_\alpha^3\eta\|_{L^2}\bigr) \\
&\quad + \frac{1}{\nu}\mathcal F(\|\eta\|_{W^{1,\infty}})\|\partial_\alpha^2\eta\|_{L^2}
\bigl(1+\|\partial_\alpha^2\eta\|_{L^2}+\|\partial_\alpha^2\eta\|_{L^2}^2\bigr).
\end{aligned}
\end{equation}
Choosing \(\nu=\varepsilon/4\) and using the AM--GM inequality, we get
\begin{equation}
\frac{d}{dt}\|\partial_\alpha^2\eta\|_{L^2}^2+\varepsilon\|\partial_\alpha^3\eta\|_{L^2}^2
\leq \varepsilon^{-3}\mathcal F(\|\eta\|_{W^{1,\infty}})\bigl(1+\|\partial_\alpha^2\eta\|_{L^2}^2\bigr)\|\partial_\alpha^2\eta\|_{L^2}^2.
\end{equation}
Applying Gronwall's inequality (Lemma \ref{Gronwall}), we obtain
\begin{equation}
\begin{aligned}
&\|\partial_\alpha^2\eta(t)\|_{L^2}^2+\varepsilon\int_0^t\|\partial_\alpha^3\eta(s)\|_{L^2}^2\,ds\\
&\leq \|\partial_\alpha^2\eta_0\|_{L^2}^2 \exp\bigg(\varepsilon^{-3}\int_0^t \mathcal F(\|\eta(s)\|_{W^{1,\infty}})
\bigl(1+\|\partial_\alpha^2\eta(s)\|_{L^2}^2\bigr)\,ds\bigg).
\end{aligned}
\end{equation}
Bounding \(\|\eta(s)\|_{W^{1,\infty}}\) in terms of \(\|\eta_0\|_{W^{1,\infty}}\) and using \eqref{H^1_Gronwall}, we conclude that
\begin{equation}
\|\partial_\alpha^2\eta(t)\|_{L^2}^2\leq \mathcal F(\|\eta_0\|_{W^{1,\infty}},\varepsilon^{-1},t)\|\partial_\alpha^2\eta_0\|_{L^2}^2,
\end{equation}
as desired.

\begin{proof}[Proof of Lemma \ref{grad_estimate}]
For convenience of notation, we suppress the dependence on $t$. Differentiating (\ref{DN_mollified}) with $h = e^\eta$ plugged in, we obtain
\begin{equation}
\begin{split}
\partial_\alpha G(h)\eta(\alpha) = G_1 + G_2 + G_3,
\end{split}
\end{equation}
where
\[\begin{split}
G_1(\alpha) &= \re\ \frac{\partial_\alpha^2 e^{\eta(\alpha) - i\alpha}}{2\pi}\int_\T \frac{e^{\eta(\alpha) - i\alpha} - e^{\eta(\alpha - \beta) - i(\alpha - \beta)}}{|e^{\eta(\alpha) + i\alpha} - e^{\eta(\alpha - \beta) + i(\alpha - \beta)}|^2}\ \theta(\alpha -\beta)  d\beta,\\
G_2(\alpha) &= -\re\ \frac{\partial_\alpha e^{\eta(\alpha) + i\alpha}}{2\pi}\int_\T \frac{\partial_\alpha e^{\eta(\alpha) - i\alpha} - \partial_\alpha e^{\eta(\alpha - \beta) - i(\alpha - \beta)}}{|e^{\eta(\alpha) + i\alpha} - e^{\eta(\alpha - \beta) + i(\alpha - \beta)}|^2} \theta(\alpha -\beta)\ d\beta,\\
G_3(\alpha) &= \int_\T\Lambda[e^\eta](\alpha,\alpha -\beta) \partial_\alpha\theta(\alpha -\beta)\ d\beta.
\end{split}\]
Before we proceed with estimates of $G_1$, $G_2$, and $G_3$, we note that straightforward calculations give us the bounds
\begin{equation}\label{eta_derivative_bound}
\begin{split}
|\partial_\alpha e^{\eta(\alpha) \pm i\alpha}| &= |e^{\eta(\alpha)}|\sqrt{1 + (\partial_\alpha\eta(\alpha))^2}, \\
|\partial_\alpha^2e^{\eta(\alpha) \pm i\alpha}| &\leq |e^{\eta(\alpha)}|\big(|\partial_\alpha^2\eta(\alpha)| + |\partial_\alpha\eta(\alpha)|^2 + 1\big), \\
|\partial_\alpha^3e^{\eta(\alpha) \pm i\alpha}|&\leq |e^{\eta(\alpha)}|\big(|\partial_\alpha^3\eta(\alpha)| + 3(1+|\partial_\alpha\eta(\alpha)|)|\partial_\alpha^2\eta(\alpha)| + (|\partial_\alpha\eta(\alpha)|+1)^3\big).
\end{split}
\end{equation}
We shall also frequently use the fact that
\begin{equation}\label{denom_lower_bound}
\begin{split}
|e^{\eta(\alpha) + i\alpha} - e^{\eta(\beta) + i\beta}|^2 &= [e^{\eta(\alpha) - \eta(\beta)}]^2 + e^{\eta(\alpha) + \eta(\beta)}\sin^2(\frac{\alpha - \beta}{2})\\
&\gtrsim \exp(-2\|\eta\|_{L^\infty})(\alpha-\beta)^2.
\end{split}
\end{equation}
\paragraph{\textbf{Estimates of $G_1$}}
We define
\begin{equation}
\begin{split}
T\theta(\alpha) &:=  \int_\T \frac{e^{\eta(\alpha) - i\alpha} - e^{\eta(\alpha - \beta) - i(\alpha - \beta)}}{|e^{\eta(\alpha) + i\alpha} - e^{\eta(\alpha - \beta) + i(\alpha - \beta)}|^2}  \theta(\alpha -\beta)\ d\beta\\
\end{split}
\end{equation}
\(T\) is a Cauchy integral operator on a Lipschitz curve. Therefore, by classical results in harmonic analysis, \(T\) is bounded on \(L^2(\T)\), with
operator norm depending only on $\|\eta\|_{W^{1,\infty}}$. More detailed discussions can be found in Appendix \ref{subsec:bdd_cauchy}. We then have
\begin{equation}
\begin{split}
\|G_1\|_{L^2} &\leq \|\partial_\alpha^2e^{\eta + i\cdot}\|_{L^\infty}\|T\theta\|_{L^2}
\\
&\leq \cF(\|\eta\|_{W^{1,\infty}})(1 + \|\partial_\alpha^2\eta\|_{L^\infty}).
\end{split}
\end{equation}
By Sobolev embedding,
\begin{equation}\label{second_derivative_sobolev}
\|\partial_\alpha^2\eta\|_{L^\infty}\leq C\|\partial_\alpha^3\eta\|_{L^2}^\frac{1}{2}\|\partial_\alpha^2\eta\|_{L^2}^{\frac{1}{2}},
\end{equation}
and an application of the AM-GM inequality gives
\begin{equation}\label{G_1_bound}
\|G_1\|_{L^2}\leq \nu\|\partial_\alpha^3\eta\|_{L^2} + \frac{1}{\nu}\cF(\|\eta\|_{W^{1,\infty}})(1 + \|\partial_\alpha^2\eta\|_{L^2} + \|\partial_\alpha^2\eta\|_{L^2}^2)
\end{equation}
as desired.

\vspace{2ex}

\paragraph{\textbf{Estimates of $G_2$}}
We write
\[G_2 = G_{2, 1} + G_{2, 2} + G_{2, 3} + G_{2, 4},\]
where, for $\delta > 0$ to be chosen,
\[\begin{split}
G_{2, 1}(\alpha) &= -\re\ \frac{\partial_\alpha e^{\eta(\alpha) + i\alpha}}{2\pi}\int_{|\beta| > \delta} \frac{\partial_\alpha e^{\eta(\alpha) - i\alpha} - \partial_\alpha e^{\eta(\alpha - \beta) - i(\alpha - \beta)}}{|e^{\eta(\alpha) + i\alpha} - e^{\eta(\alpha - \beta) + i(\alpha - \beta)}|^2}\theta(\alpha -\beta)\ d\beta,\\
G_{2, 2}(\alpha) &= -\re\ \frac{\partial_\alpha e^{\eta(\alpha) + i\alpha}}{2\pi}\int_{|\beta| < \delta} \frac{(\partial_\alpha e^{\eta(\alpha) - i\alpha} - \partial_\alpha e^{\eta(\alpha - \beta) - i(\alpha - \beta)}) - \beta\partial_\alpha^2 e^{\eta(\alpha) + i\alpha}}{|e^{\eta(\alpha) + i\alpha} - e^{\eta(\alpha - \beta) + i(\alpha - \beta)}|^2}\theta(\alpha - \beta)\ d\beta,
\end{split}\]
\[G_{2, 3}(\alpha) = -\re\ \frac{\partial_\alpha e^{\eta(\alpha) + i\alpha}\partial_\alpha^2e^{\eta(\alpha) + i\alpha}}{2\pi}\int_{|\beta| < \delta}  \beta\ A[h](\alpha, \beta)\  \theta(\alpha - \beta)\ d\beta \]
\[\mathrm{with}\quad A[h](\alpha, \beta) = \frac{1}{|e^{\eta(\alpha) + i\alpha} - e^{\eta(\alpha - \beta) + i(\alpha - \beta)}|^2}-\frac{1}{e^{2\eta(\alpha)}[1+(\partial_\alpha\eta(\alpha))^2]\beta^2},\]
and
\[G_{2, 4}(\alpha) = \frac{\partial_\alpha e^{\eta(\alpha) + i\alpha}\partial_\alpha^2 e^{\eta(\alpha)+i\alpha}}{e^{\eta(\alpha)}[1 + (\partial_\alpha\eta(\alpha))^2]}\int_{|\beta| < \delta}\frac{\theta(\alpha - \beta)}{\beta}\ d\beta .
\]

We begin with $G_{2, 1}$. We note from (\ref{denom_lower_bound}) that the kernel of $G_{2, 1}$ is bounded by
\begin{equation}
\begin{split}
\lesssim \frac{|\partial_\alpha e^{\eta(\alpha) - i\alpha} - \partial_\alpha e^{\eta(\alpha - \beta) - i(\alpha - \beta)}|}{|\beta|^2},
\end{split}
\end{equation}
which can be further bounded by
\begin{equation}
\begin{split}
&\lesssim \|\partial_\alpha e^{\eta - i\cdot}\|_{C^{\frac{1}{2}}}|\beta|^{-\frac{3}{2}}\lesssim \|\partial_\alpha^2 e^{\eta - i\cdot}\|_{L^2}|\beta|^{-\frac{3}{2}} \lesssim \cF(\|\eta\|_{W^{1,\infty}})(1 + \|\partial_\alpha^2\eta\|_{L^2})|\beta|^{-\frac{3}{2}},
\end{split}
\end{equation}
where the second inequality follows from Morrey's inequality. Therefore,
\begin{equation}
\begin{split}
\|G_{2, 1}\|_{L^2} &\leq \cF(\|\eta\|_{W^{1,\infty}})(1 + \|\partial_\alpha^2\eta\|_{L^2})\bigg\|\int_{|\beta|>\delta} \frac{\theta(\cdot - \beta)}{|\beta|^{\frac{3}{2}}}\bigg\|_{L^2}\\
&\leq \cF(\|\eta\|_{W^{1,\infty}})(1 + \|\partial_\alpha^2\eta\|_{L^2})\delta^{-\frac{1}{2}},
\end{split}
\end{equation}
where we used Young's convolution inequality to bound
\begin{equation}
\bigg\|\int_{|\beta|>\delta} \frac{\theta(\cdot - \beta)}{|\beta|^{\frac{3}{2}}}\bigg\|_{L^2}\leq \frac{\|\theta\|_{L^2}}{\delta^{\frac{1}{2}}}.
\end{equation}

For $G_{2, 2}$, we note that
\begin{equation}
\begin{aligned}
&\big|(\partial_\alpha e^{\eta(\alpha)-i\alpha}-\partial_\alpha e^{\eta(\alpha-\beta)-i(\alpha-\beta)})
-\beta\,\partial_\alpha^2 e^{\eta(\alpha)-i\alpha}\big| \\
&\qquad \lesssim |\beta|^{3/2}\|\partial_\alpha^2 e^{\eta-i\cdot}\|_{C^{1/2}}\\
&\qquad\lesssim |\beta|^{3/2}\|\partial_\alpha^3 e^{\eta-i\cdot}\|_{L^2} \\
&\qquad \le |\beta|^{3/2}\mathcal F(\|\eta\|_{W^{1,\infty}})
\bigl(1+\|\partial_\alpha^2\eta\|_{L^\infty}+\|\partial_\alpha^3\eta\|_{L^2}\bigr).
\end{aligned}
\end{equation}
where the second inequality again follows from Morrey's inequality. We can utilize Sobolev embedding and AM-GM inequality to further bound
\begin{equation}\label{sobolev_embedding_am_gm}
\begin{split}
\|\partial_\alpha^2\eta\|_{L^\infty}&\lesssim \|\partial_\alpha^3\eta\|_{L^2}^\frac{1}{2}\|\partial_\alpha^2\eta\|_{L^2}^{\frac{1}{2}} \lesssim \|\partial_\alpha^3\eta\|_{L^2}+\|\partial_\alpha^2\eta\|_{L^2}
\end{split}
\end{equation}
and obtain 
\begin{equation}\label{numerator_bound}
\begin{aligned}
|(\partial_\alpha e^{\eta(\alpha) - i\alpha} - \partial_\alpha e^{\eta(\alpha - \beta) - i(\alpha - \beta)}) - \beta\partial_\alpha^2 e^{\eta(\alpha) - i\alpha}|\\
\leq |\beta|^{\frac{3}{2}}\cF(\|\eta\|_{W^{1,\infty}})(1 + \|\partial_\alpha^2\eta\|_{L^2} + \|\partial_\alpha^3\eta\|_{L^2}).
\end{aligned}
\end{equation}
Using (\ref{numerator_bound}), (\ref{denom_lower_bound}), and (\ref{eta_derivative_bound}), we can bound the kernel of $G_{2, 2}$ by
\begin{equation}
\begin{split}
\leq \cF(\|\eta\|_{W^{1,\infty}})(1 + \|\partial_\alpha^2\eta\|_{L^2} + \|\partial_\alpha^3\eta\|_{L^2})|\beta|^{-\frac{1}{2}}
\end{split}
\end{equation}
and thus
\begin{equation}
\begin{split}
\|G_{2, 2}\|_{L^2}\leq \cF(\|\eta\|_{W^{1,\infty}})(1 + \|\partial_\alpha^2\eta\|_{L^2} + \|\partial_\alpha^3\eta\|_{L^2})\bigg\|\int\frac{\theta(\cdot - \beta)}{|\beta|^{\frac{1}{2}}}\bigg\|_{L^2}.
\end{split}
\end{equation}
Using Young's convolution inequality, we obtain 
\begin{equation}
\bigg\|\int\frac{\theta(\cdot - \beta)}{|\beta|^{\frac{1}{2}}}\bigg\|_{L^2} \lesssim \delta^{\frac{1}{2}}\|\theta\|_{L^2},
\end{equation}
and it follows that
\begin{equation}
\|G_{2, 2}\|_{L^2}\leq \cF(\|\eta\|_{W^{1,\infty}})(1 + \|\partial_\alpha^2\eta\|_{L^2} + \|\partial_\alpha^3\eta\|_{L^2})\delta^{\frac{1}{2}}.
\end{equation}
For $G_{2, 3}$, mean-value theorem gives
\begin{equation}
\begin{split}
|A[h](\alpha,\beta)\beta|&\lesssim \sup_\alpha \bigg|\partial_\alpha \big|\partial_\alpha e^{\eta(\alpha) + i\alpha}\big|^2\bigg|\\
&\leq \sup_\alpha|\partial_\alpha^2 e^{\eta(\alpha) + i\alpha}|\cdot \sup_\alpha|\partial_\alpha e^{\eta(\alpha) + i\alpha}|\\
&\leq \cF(\|\eta\|_{W^{1,\infty}})(1 + \|\partial_\alpha^2\eta\|_{L^\infty}).
\end{split}
\end{equation}
Therefore, we obtain
\begin{equation}
\begin{split}
\|G_{2, 3}\|_{L^2} &\leq \cF(\|\eta\|_{W^{1,\infty}})\|\partial_\alpha^2\eta\|_{L^\infty}(1 + \|\partial_\alpha^2\eta\|_{L^\infty})\bigg\|\int_{|\beta|<\delta}\theta(\cdot - \beta)\ d\beta\bigg\|_{L^2}\\
&\leq \cF(\|\eta\|_{W^{1,\infty}})\|\partial_\alpha^2\eta\|_{L^\infty}(1 + \|\partial_\alpha^2\eta\|_{L^\infty})\delta^\frac{1}{2},
\end{split}
\end{equation}
where we used Young's convolution inequality to bound
\begin{equation}
\bigg\|\int_{|\beta|<\delta}\theta(\cdot - \beta)\ d\beta\bigg\|_{L^2}\leq \delta^\frac{1}{2}\|\theta\|_{L^2}.
\end{equation}
Another application of (\ref{sobolev_embedding_am_gm}) then gives
\begin{equation}
\|G_{2, 3}\|_{L^2} \leq \cF(\|\eta\|_{W^{1,\infty}})(1 + \|\partial_\alpha^2\eta\|_{L^2} +\|\partial_\alpha^3\eta\|_{L^2} + \|\partial_\alpha^2\eta\|_{L^2} \|\partial_\alpha^3\eta\|_{L^2})\delta^\frac{1}{2}.
\end{equation}

For $G_{2, 4}$, standard Calderon-Zygmund operator theory, bounds of $\theta$, and (\ref{sobolev_embedding_am_gm}) together give
\begin{equation}
\begin{split}
\|G_{2, 4}\|_{L^2} &\leq \cF(\|\eta\|_{W^{1,\infty}})\|\partial_\alpha^2 e^{\eta + i\cdot}\|_{L^\infty}\|\theta\|_{L^2}\\
&\leq \cF(\|\eta\|_{W^{1,\infty}})(1 + \|\partial_\alpha^2\eta\|_{L^2} + \|\partial_\alpha^3\eta\|_{L^2}).
\end{split}
\end{equation}
Now, we choose 
\[\delta = \frac{c\nu^2}{\cF(\|\eta\|_{W^{1,\infty}})(1 + \|\partial_\alpha^2\eta\|_{L^2})}\]
for $c > 0$ sufficiently small, and we obtain
\begin{equation}\label{G_2_bound}
\|G_2\|_{L^2}\leq \nu\|\partial_\alpha^3\eta\|_{L^2} + \frac{1}{\nu}\cF(\|\eta\|_{W^{1,\infty}})(1 + \|\partial_\alpha^2\eta\|_{L^2} + \|\partial_\alpha^2\eta\|_{L^2}^2)
\end{equation}
as desired.

\vspace{2ex}

\paragraph{\textbf{Estimates of $G_3$}} First of all, we recall that
\begin{equation}
G_3 = N(e^{i\alpha}h(\alpha))\cdot\nabla(S[h]\partial_\alpha\theta)(e^{i\alpha}h(\alpha))
\end{equation}
and thus
\begin{equation}
\|G_3\|_{L^2}\leq \cF(\|\eta\|_{W^{1,\infty}})\|(\frac{1}{2}I - K^*)\partial_\alpha\theta\|_{L^2}.
\end{equation}
Therefore, to obtain the desired bound for $G_3$, we just need to obtain 
\begin{equation}
\|(\frac{1}{2}I - K^*)\partial_\alpha\theta\|_{L^2}\leq \nu\|\partial_\alpha^3\eta\|_{L^2} + \frac{1}{\nu}\cF(\|\eta\|_{W^{1,\infty}})(1 + \|\partial_\alpha^2\eta\|_{L^2} + \|\partial_\alpha^2\eta\|_{L^2}^2)
\end{equation}
Differentiating equation (\ref{theta_eqn}) with respect to $\alpha$ and rearranging the terms, we get
\begin{equation}
(\frac{1}{2}I - K^*)\partial_\alpha\theta = \partial_\alpha^2\eta + (\partial_\alpha K^*\theta(\alpha) - K^*\partial_\alpha\theta(\alpha)),
\end{equation}
and a direct computation using \eqref{K^*_expression_2} yields
\begin{equation}\label{theta_derivative_eqn_1}
\begin{split}
& \partial_\alpha K^*\theta(\alpha) - K^*\partial_\alpha\theta(\alpha) \\
& = \im\ \frac{1}{2\pi}\int\partial_\alpha\bigg(\frac{\partial_\alpha e^{\eta(\alpha) + i\alpha}(e^{\eta(\alpha) - i\alpha} - e^{\eta(\alpha - \beta) - i(\alpha-\beta)})}{|e^{\eta(\alpha) + i\alpha} - e^{\eta(\alpha - \beta) + i(\alpha - \beta)})|^2}\bigg)\theta(\alpha -\beta)\ d\beta.
\end{split}
\end{equation}
Since our analysis of 
\begin{equation}
\re\ \frac{1}{2\pi}\int\partial_\alpha\bigg(\frac{\partial_\alpha e^{\eta(\alpha) + i\alpha}(e^{\eta(\alpha) - i\alpha} - e^{\eta(\alpha - \beta) - i(\alpha-\beta)})}{|e^{\eta(\alpha) + i\alpha} - e^{\eta(\alpha - \beta) + i(\alpha - \beta)})|^2}\bigg)\theta(\alpha -\beta)\ d\beta = G_1 + G_2
\end{equation}
also works for the imaginary part of the same complex number, we can obtain the same bounds as \eqref{G_1_bound} and \eqref{G_2_bound}, and thus
\begin{equation}
\begin{split}
\|(\frac{1}{2}I - K^*)\partial_\alpha\theta\|_{L^2} &\leq \nu\|\partial_\alpha^3\eta\|_{L^2} + \frac{1}{\nu}\cF(\|\eta\|_{W^{1,\infty}})(1 + \|\partial_\alpha^2\eta\|_{L^2} + \|\partial_\alpha^2\eta\|_{L^2}^2)
\end{split}
\end{equation}
as desired. This finishes the proof of the lemma. 
\end{proof}
\paragraph{\textbf{Higher regularity}} For $s\geq 2$, we have
\begin{equation}
\begin{split}
\frac{1}{2}\frac{d}{dt}\|\eta\|_{H^s}^2 + \varepsilon\|\partial_\alpha\eta\|_{H^s}^2 &\leq \|\eta\|_{H^{s+1}}\|e^{-2\eta}G(h)\eta\|_{H^{s-1}},
\end{split}
\end{equation}
and by product lemma in Sobolev spaces (see Lemma \ref{Sobolev_prod}), 
\begin{equation}
\|e^{-2\eta}G(h)\eta\|_{H^{s-1}} \lesssim \|e^{-2\eta}\|_{L^\infty}\|G(h)\eta\|_{H^{s-1}} + \|e^{-2\eta}\|_{H^{s-1}}\|G(h)\eta\|_{L^\infty}.
\end{equation}
Note that
\begin{equation}
\|e^{-2\eta}\|_{L^\infty} \leq \exp(2\|\eta\|_{L^\infty}),
\end{equation}
and that by fractional chain rule in Sobolev spaces (Lemma \ref{fractional_chain_rule}) we have
\begin{equation}
\|e^{-2\eta}\|_{H^{s-1}}\lesssim \exp(2\|\eta\|_{L^\infty})\|\eta\|_{H^{s-1}}.
\end{equation}
Moreover, for $s_0\in (\frac{3}{2}, 2]$, we have from Proposition \ref{DN_boundedness} that
\begin{equation}
\|G(h)\eta\|_{H^{s-1}} \leq \cF(\|\eta\|_{H^{s_0}})\|\eta\|_{H^s}.
\end{equation}
By Sobolev embedding (Lemma \ref{lem:1d_sobolev_interpolation}),
\begin{equation}
\|G(h)\eta\|_{L^\infty} \lesssim \|G(h)\eta\|_{H^1}\leq \cF(\|\eta\|_{H^{s_0}})\|\eta\|_{H^2}.
\end{equation}
so it follows that for $s_0\in (\frac{3}{2}, 2]$,
\begin{equation}
\|e^{-2\eta}G(h)\eta\|_{H^{s-1}} \leq \cF(\|\eta\|_{H^{s_0}})\|\eta\|_{H^s}.
\end{equation}
Eventually, 
\begin{equation}
\|\eta\|_{H^{s+1}}\lesssim \|\eta\|_{L^2} + \|\partial_\alpha\eta\|_{H^s},
\end{equation}
so an AM-GM inequality argument as before gives
\begin{equation}
\frac{d}{dt}\|\eta\|_{H^s}^2 + \varepsilon\|\partial_\alpha\eta\|_{H^s}^2 \leq \cF(\varepsilon^{-1}, \|\eta\|_{H^{s-1}})\|\eta\|_{H^s}^2.
\end{equation}
By Gronwall's inequality \eqref{Gronwall}, we have
\begin{equation}
\|\eta(t)\|_{H^s}^2 + \varepsilon\int_0^t \|\partial_\alpha\eta(\tau)\|_{H^s}^2\ d\tau \leq \exp\bigg(\int_0^t\cF(\varepsilon^{-1}, \|\eta(\tau)\|_{H^{s-1}})\ d\tau\bigg)\|\eta_0\|_{H^s}^2.
\end{equation}
Inductively, we can show that for all $s\geq 2$, 
\begin{equation}
\|\eta(t)\|_{H^s}^2 + \varepsilon\int_0^t\|\partial_\alpha\eta(\tau)\|_{H^s}^2\ d\tau \leq\cF\big(\|\eta_0\|_{W^{1,\infty}}, \varepsilon^{-1}, t)\|\eta_0\|_{H^s}^2.
\end{equation}
It follows that finite-time blow up of $\|\eta(t)\|_{H^s}$ cannot happen, and we established global well-posedness of \eqref{visc_interface_eqn} with $H^s(\T)$ initial data. 

\section{Viscosity Solutions of the Interface Equation}\label{sec:visc_sol}
In this section, we introduce the notion of viscosity solutions for the interface equation \eqref{interface_eqn}. This definition is motivated by the comparison principle \ref{visc_cp} established earlier for classical solutions of \eqref{visc_interface_eqn}. We emphasize that the extrema in the definition must be global because of the nonlocal nature of the Dirichlet-to-Neumann operator, and must also be non-negative (respectively non-positive) in order to preserve the sign information needed in the comparison argument.

\begin{definition}[Viscosity solution of the boundary equation]\label{def:visc_sol_bdry}
A function $\eta: [0, T]\times\T\to \R$ is called a viscosity subsolution (resp. supersolution) of \eqref{interface_eqn} on $(0, T)$ if the following conditions hold:
\begin{enumerate}
\item $\eta$ is upper semicontinuous (resp. lower semicontinuous) on $\T\times [0, T]$, and
\item for every $\psi: \T\times(0, T)\to \R$ with $\partial_t\psi\in C(\T\times (0, T))$ and $\psi\in C((0, T); C^{1,1}(\T))$, if $\eta - \psi$ attains a \emph{non-negative global maximum (resp. non-positive global minimum)} over $\T\times [t_0 - r, t_0]$ at $(\alpha_0, t_0)\in \T\times (0, T)$ for some $r > 0$, then
\begin{equation}
\partial_t\psi(\alpha_0, t_0) + e^{-2\psi}(G(e^\psi)\psi - 1)(\alpha_0, t_0) \leq 0\quad (\mathrm{resp.}\ \geq 0).
\end{equation}
\end{enumerate}
A viscosity solution is both a viscosity subsolution and a viscosity supersolution.
\end{definition}

For the one-phase Muskat problem, there is by now an extensive literature on viscosity solutions of the interface equation; see, for example, \cite{ChangLaraGuillenSchwab2019}, \cite{Dong-Gancedo-Nguyen-23}, and \cite{SchwabTuTuranova2024} for more detailed discussions. For the horizontal Hele-Shaw problem, the notion of viscosity solution has also existed for a long time, although it has most often been formulated for the domain equation \eqref{domain_eqn_p}. In Section \ref{sec:visc_sol_domain}, we revisit this classical literature and examine its connection with Definition \ref{def:visc_sol_bdry}.

We now begin by showing that once a viscosity solution enjoys sufficient regularity, it upgrades to a classical solution.
\begin{proposition}[Consistency]\label{consistency}
Let $\eta$ be a viscosity subsolution (resp. supersolution) of \eqref{interface_eqn} on $(0, T)$ and let $h := e^\eta$. Assume that $\eta\in W^{1,\infty}(\T\times (0, T))$ and is $C^{1, 1}$ at $(\alpha_0, t_0)\in \T\times (0, T)$. Then $G(h)\eta$ is classically well-defined at $(\alpha_0, t_0)$ and
\begin{equation}\label{classical_eqn}
\partial_t\eta(\alpha_0, t_0) + e^{-2\eta}(G(h)\eta - 1)(\alpha_0, t_0) \leq 0\quad (\mathrm{resp.}\ \geq 0).
\end{equation}
\end{proposition}

\begin{proof}
The proof is very similar to the one for periodic one-phase Muskat problem (\cite{Dong-Gancedo-Nguyen-23}, Proposition 6.2), so we outline the main idea here. We only discuss viscosity subsolutions, since the arguments for supersolutions are similar. First of all we note from Proposition \ref{pointwise_cp} that $G(h)\eta(\alpha_0, t_0)$ is classically well-defined since $\eta(t_0)$ is $C^{1,1}$ . To show the equation \eqref{classical_eqn}, one constructs a family of functions $\{\psi_r\}_r\subset C^\infty(\T\times\R)$ such that 
\begin{enumerate}
    \item $\psi_r\geq \eta$ on $\T\times [0, T]$;
    \item $\psi_r(\alpha_0, t_0) = \eta(\alpha_0, t_0)$;
    \item there exists $\delta > 0$ sufficiently small such that $\psi_r \to \eta$ in $C(\T\times [\delta, T - \delta])$;
    \item $\lim_{r\to 0} G(e^{\psi_r})\psi_r(\alpha_0, t_0) = G(e^\eta)\eta(\alpha_0, t_0)$.
\end{enumerate}
The construction of $\psi_r$ that satisfies (1) - (3) can be taken exactly as in Appendix C of \cite{Dong-Gancedo-Nguyen-23}. Since $\eta-\psi_r$ attains a maximum at $(\alpha_0,t_0)$ and $\eta$ is differentiable at $(\alpha_0,t_0)$, we also have
\[
\partial_t\psi_r(\alpha_0,t_0)=\partial_t\eta(\alpha_0,t_0).
\]
Justification of (4) follows from equation (6.8) in \cite{Dong-Gancedo-Nguyen-23} and the equivalence we established in Appendix \ref{appendix:DN}. We remark that justification of (4) is highly non-trivial and crucially depends on quantitative $C^{1,\alpha}$ estimates established in Section 2.2 of \cite{Dong-Gancedo-Nguyen-23}. We refer to \cite{Dong-Gancedo-Nguyen-23} for the details.

Once we have the family $\{\psi_r\}_r$, we note that from (1) and (2) that every $\psi_r$, which is smooth, is a test function for viscosity subsolution at $(\alpha_0, t_0)$. Thus
\begin{equation}
\partial_t\psi_r(\alpha_0, t_0) + e^{-2\psi_r}(G(e^{\psi_r})\psi_r - 1)(\alpha_0, t_0) \leq 0
\end{equation}
for all sufficiently small $r > 0$. Sending $r\to 0$, equation \eqref{classical_eqn} then follows from (2), (3), (4), and the identity $\partial_t\psi_r(\alpha_0,t_0)=\partial_t\eta(\alpha_0,t_0)$.
\end{proof}

We proved a comparison principle as in Proposition \ref{visc_cp} for classical solutions of the regularized interface equation \eqref{visc_interface_eqn}, where $\varepsilon\geq 0$. Now we extend the comparison principle to viscosity solutions.
\begin{proposition}[Comparison principle]\label{visc_sol_cp}
Assume that $\eta_-,\ \eta+:\T\times [0, T]\to\R$ are bounded subsolution and supersolution of (\ref{interface_eqn}) on $(0, T)$. If $\eta_-(\alpha, 0)\leq \eta_+(\alpha, 0)$ for all $\alpha\in\T$, then $\eta_-(\alpha, t)\leq \eta_+(\alpha, t)$ for all $(\alpha, t)\in \T\times [0, T]$. 
\end{proposition}
\begin{proof}
The proof, where $\eta_-$ and $\eta_+$ are only bounded and semi-continuous, is a standard argument using sup- and inf-convolutions. For small $\delta >0$, we define the sup- and inf- convolutions 
\begin{equation}
\begin{aligned}
\eta^\delta(\alpha, t) &= \sup_{(\beta, s)\in\T\times[0, T]} \bigg[\eta(\beta, t) - \frac{1}{2\delta}(|\alpha - \beta|^2 + |t - s|^2)\bigg],\\
\eta_\delta(\alpha, t) &= \sup_{(\beta, s)\in\T\times[0, T]} \bigg[\eta(\beta, t) + \frac{1}{2\delta}(|\alpha - \beta|^2 + |t - s|^2)\bigg].
\end{aligned}
\end{equation}
for $(\alpha, t)\in \T\times[0, T]$. We recall the following standard properties of sup- and inf-convolutions
\cite[Appendix A]{CrandallIshiiLions1992}.
\begin{enumerate}[(i)]
    \item \(\eta^\delta, \eta_\delta \in \operatorname{Lip}(\mathbb T\times[0,T])\).

    \item \(\eta^\delta\) is semiconvex and \(\eta_\delta\) is semiconcave. More precisely,
    \(\eta^\delta\) is semiconvex with opening \(\delta^{-1}\) in the sense that, for every
    \(X_0=(\alpha_0,t_0)\in \mathbb T\times[0,T]\), there exists \(v_{X_0}\in\mathbb R^2\)
    such that, in local coordinates and for \(X=(\alpha,t)\) sufficiently close to \(X_0\),
    \[
        \eta^\delta(X)
        \geq
        \eta^\delta(X_0)+v_{X_0}\cdot (X-X_0)
        -\frac{1}{2\delta}|X-X_0|^2.
    \]
    Similarly, \(\eta_\delta\) semiconcave with opening \(\delta^{-1}\) means that for every \(X_0\), there exists \(w_{X_0}\in\mathbb R^2\) such that
    \[
        \eta_\delta(X)
        \leq
        \eta_\delta(X_0)+w_{X_0}\cdot (X-X_0)
        +\frac{1}{2\delta}|X-X_0|^2
    \]
    for \(X\) sufficiently close to \(X_0\).

    \item The half-relaxed limits
    \[
        \limsup_{\substack{\delta\to0\\ (\beta,s)\to(\alpha,t)}} \eta^\delta(\beta,s)
        =
        \eta(\alpha,t),
        \qquad
        \liminf_{\substack{\delta\to0\\ (\beta,s)\to(\alpha,t)}} \eta_\delta(\beta,s)
        =
        \eta(\alpha,t)
    \]
    hold for all \((\alpha,t)\in\mathbb T\times[0,T]\).
\end{enumerate}

We now claim that if $\eta$ is a viscosity subsolution of (\ref{interface_eqn}), each $\eta^\delta$ for $\delta > 0$ is also a viscosity subsolution of (\ref{interface_eqn}). Let $\psi$ be a test function as in Definition \ref{def:visc_sol_bdry} such that $\eta^\delta - \psi$ attains a global maximum $M\geq 0$ on $\T\times[t_0 - r, t_0]$ at $(\alpha_0, t_0)$, and suppose
\begin{equation}
\eta^\delta(\alpha_0, t_0) = \eta(\beta_0, s_0) - \frac{1}{2\delta}(|\alpha_0 - \beta_0|^2 + |t_0 - s_0|^2).
\end{equation}
We note that such $(\beta_0, s_0)$ can be found since $\eta$ is upper semicontinuous. Then 
\begin{equation}\label{conv_max}
\begin{aligned}
M &= \eta^\delta(\alpha_0, t_0) - \psi(\alpha_0, t_0)\\
  &= \eta(\beta_0, s_0) - \frac{1}{2\delta}(|\alpha_0 - \beta_0|^2 + |t_0 - s_0|^2) - \psi(\alpha_0, t_0),
\end{aligned}
\end{equation}
and for any $(\alpha, t)\in \T\times[t_0 - r, t_0]$, $(\beta, s)\in\T\times[0, T]$, one has
\begin{equation}\label{conv_ineq}
\begin{aligned}
M &\geq \eta^\delta(\alpha, t) - \psi(\alpha, t)\\
  &\geq \eta(\beta, s) - \frac{1}{2\delta}(|\alpha - \beta|^2 + |t - s|^2)-\psi(\alpha, t).
\end{aligned}
\end{equation}
Now, if we define the translated function
\[\tilde\psi(\alpha, t) := \psi(\alpha + \alpha_0 - \beta_0, t + t_0 - s_0) + \frac{1}{2\delta}(|\alpha_0 - \beta_0|^2 + |t_0 - s_0|^2)  \]
we have from \eqref{conv_max} that 
\begin{equation}
M = \eta(\beta_0, s_0) - \tilde\psi(\beta_0, s_0).
\end{equation}
If we use \eqref{conv_ineq} with $(\alpha, t) = (\beta + \alpha_0 - \beta_0, s + t_0 - s_0)$, where $(\beta, s)\in \T\times [s_0 - r, s_0]$, we have that
\begin{equation}
\begin{aligned}
M &\geq \eta(\beta, s) - \frac{1}{2\delta}(|\alpha_0 - \beta_0|^2 + |t_0 - s_0|^2) - \psi(\beta+\alpha_0 - \beta_0, s + t_0 - s_0)\\
  &= \eta(\beta, s) - \tilde\psi(\beta, s)
\end{aligned}
\end{equation}
for any $(\beta, s)\in \T\times [s_0 - r, s_0]$. Therefore, $\tilde\psi$ is a valid test function at $(\beta_0, s_0)$, and thus
\begin{equation}
\partial_t\tilde\psi + e^{-2\tilde\psi}(G(e^{\tilde\psi})\tilde\psi - 1) \leq 0\quad \mathrm{at}\ (\beta_0, s_0).
\end{equation}
We note that 
\begin{equation}
\partial_t\tilde\psi(\beta_0, s_0) = \partial_t\psi(\beta, s),
\end{equation}
and from Proposition \ref{scaling_symm}, Proposition \ref{taylor_sign} that
\begin{equation}
G(e^{\tilde\psi})\tilde\psi(\beta_0, s_0) - 1 = G(e^\psi)\psi(\alpha_0, t_0) - 1\leq 0,
\end{equation}
and that
\begin{equation}
\begin{aligned}
0 &\leq \exp(-2\tilde\psi(\beta_0, s_0))\\
&= \exp(-2\psi(\alpha_0, t_0))\exp\bigg(-\frac{1}{\delta}(|\alpha_0 - \beta_0|^2 + |t_0 - s_0|^2)\bigg)\\
&\leq \exp(-2\psi(\alpha_0, t_0)).
\end{aligned}
\end{equation}
Therefore, 
\begin{equation}
\partial_t\psi + e^{-2\psi}(G(e^\psi)\psi - 1) \leq 0\quad \mathrm{at}\ (\alpha_0, t_0), 
\end{equation}
and $\eta^\delta$ is a viscosity subsolution, showing the claim. Similarly, if $\eta$ is a viscosity subsolution of \eqref{interface_eqn}, $\eta_\delta$ is also a viscosity subsolution of \eqref{interface_eqn} for each $\delta > 0$. 

We next show that for every $\varepsilon > 0$, there exists a $\delta(\varepsilon) > 0$ such that for all $0< \delta \leq \delta(\varepsilon)$, $(\eta_-)^\delta\leq (\eta_+)_\delta + \varepsilon$ pointwise on $\T\times[0, T]$. The argument is much like in the proof of Proposition \ref{visc_cp}. Suppose on the contrary that there exists some $\varepsilon_0 > 0$ and a sequence $\delta_k\downarrow 0$ such that
\begin{equation}
M_n:= \max_{\T\times[0, T]}\bigg[(\eta_-)^{\delta_k} - (\eta_+)_{\delta_k}\bigg] = (\eta_-)^{\delta_k}(\alpha_k, t_k) -(\eta_+)_{\delta_k}(\alpha_k, t_k) \geq \varepsilon_0,
\end{equation}
where such $(\alpha_k, t_k)$ exist since $(\eta_-)^{\delta_k} - (\eta_+)_{\delta_k}$ is upper semicontinuous. Since $\eta_-(\cdot, 0) - \eta_+(\cdot, 0) \leq 0$, we have
\begin{equation}
\max_\T\bigg[(\eta_-)^{\delta_k}(\cdot, 0) - (\eta_+)_{\delta_k}(\cdot, 0)\bigg] < \frac{\varepsilon_0}{3}
\end{equation}
for all $k\geq k_0$ for some large enough $k_0$. Let $c > 0$ be sufficiently small such that
\begin{equation}
(\eta_-)^{\delta_{k_0}}(\alpha_{k_0}, t_{k_0}) - (\eta_+)_{\delta_{k_0}}(\alpha_{k_0}, t_{k_0}) - ct_{k_0} > \frac{2\varepsilon_0}{3},
\end{equation}
then
\begin{equation}
M_* :=\max_{\T\times[0, T]}\bigg[(\eta_-)^{\delta_{k_0}} - (\eta_+)_{\delta_{k_0}} - ct\bigg] = (\eta_-)^{\delta_{k_0}}(\alpha_*, t_*) - (\eta_+)_{\delta_{k_0}}(\alpha_*, t_*) - ct_* > \frac{2\varepsilon_0}{3}
\end{equation}
where again such $(\alpha_*, t_*)$ exists since the function we are maximizing is upper semicontinuous. In other words,
\begin{equation}\label{eqn:max_contact}
\max_{\T\times[0, T]}\bigg[(\eta_-)^{\delta_{k_0}} - (\eta_+)_{\delta_{k_0}} - ct - M_*\bigg] = 0
\end{equation}
is attained at $(\alpha_*, t_*)$. We note that both $(\eta_-)^{\delta_{k_0}}$ and $(\eta_+)_{\delta_{k_0}}$ are $C^{1, 1}$ at $(\alpha_*, t_*)$. 

To see this, write $\delta=\delta_{k_0}$, $X=(\alpha,t)$, $X_*=(\alpha_*,t_*)$. Recall that
\((\eta_-)^\delta\) and \(-(\eta_+)_\delta\) are semiconvex, so there exist
\(v_-,v_+\in\mathbb R^2\) and \(A\sim \delta^{-1}\) such that, for \(X\) sufficiently
close to \(X_*\),
\begin{align}
(\eta_-)^\delta(X)
&\geq
(\eta_-)^\delta(X_*)+v_-\cdot (X-X_*)-A|X-X_*|^2,
\label{eq:eta-minus-lower}\\
-(\eta_+)_\delta(X)
&\geq
-(\eta_+)_\delta(X_*)+v_+\cdot (X-X_*)-A|X-X_*|^2.
\label{eq:eta-plus-negative-lower}
\end{align}

By \eqref{eqn:max_contact} we have, for every \(X\),
\begin{equation}\label{eq:sum-upper}
(\eta_-)^\delta(X)-(\eta_+)_\delta(X)
\leq
(\eta_-)^\delta(X_*)-(\eta_+)_\delta(X_*)+(0,c)\cdot (X-X_*).
\end{equation}
We only prove that \((\eta_-)^\delta\) is \(C^{1,1}\) at \(X_*\), since the proof for $(\eta_+)_\delta$ is similar. Rearranging \eqref{eq:sum-upper}, we get
\[
(\eta_-)^\delta(X)
\leq
(\eta_-)^\delta(X_*)+(0,c)\cdot(X-X_*)
-\Big[-(\eta_+)_\delta(X)+(\eta_+)_\delta(X_*)\Big].
\]
Then we use \eqref{eq:eta-plus-negative-lower} to obtain 
\begin{equation}\label{eq:eta-minus-upper}
(\eta_-)^\delta(X)
\leq
(\eta_-)^\delta(X_*)+\big((0,c)-v_+\big)\cdot(X-X_*)
+A|X-X_*|^2.
\end{equation}
We combine \eqref{eq:eta-minus-lower} and \eqref{eq:eta-minus-upper} to obtain
\begin{equation}
|(v_- -((0,c)-v_+))\cdot (X-X_*)|
\leq 2A|X - X_*|^2.
\end{equation}
Then we divide both sides by $|X-X_*|$ and send $X\to X_*$ to get
\begin{equation}
v_-=(0,c)-v_+.
\end{equation}
Therefore, \eqref{eq:eta-minus-lower} and \eqref{eq:eta-minus-upper} imply
\[
\left|
(\eta_-)^\delta(X)-(\eta_-)^\delta(X_*)
-v_-\cdot(X-X_*)
\right|
\leq A|X-X_*|^2.
\]
which means that \((\eta_-)^\delta\) is \(C^{1,1}\) at \(X_*\). Therefore, by Proposition \ref{consistency}, $(\eta_-)^{\delta_{k_0}}$ and $(\eta_+)_{\delta_{k_0}}$ are classical subsolutions and supersolutions, repsectively, at $(\alpha_*, t_*)$. Then we finish the proof of the claim by deriving a contradiction exactly as in Proposition \ref{visc_cp}. 

Now, for any $(\alpha, t)\in \T\times[0, T]$, we have
\begin{equation}
\eta_- - \eta_+ = \limsup_{\delta\to 0} (\eta_-)^\delta - \liminf_{\delta\to 0} (\eta_+)_\delta =  \limsup_{\delta\to 0} \big[(\eta_-)^\delta - (\eta_+)_\delta\big] \leq \varepsilon
\end{equation}
and thus
\begin{equation}
(\eta_- - \eta_+)(\alpha, t)\leq 0
\end{equation}
since $\varepsilon > 0$ is arbitrary. This finishes the proof as desired.
\end{proof}

A straightforward corollary is that viscosity solutions as defined in Definition \ref{def:visc_sol_bdry} must be unique. 
\begin{corollary}[Uniqueness of viscosity solutions]
Suppose $\eta_1,\ \eta_2$ are viscosity solutions of \eqref{interface_eqn} on $(0, T)$ with $\eta_1(\cdot, 0) = \eta(\cdot, 0)$, then $\eta_1 \equiv \eta_2$ on $\T\times [0, T]$. 
\end{corollary}

Another corollary of Proposition \ref{visc_sol_cp}, just as how Proposition \ref{visc_reg_mod_cty} follows from Proposition \ref{visc_cp}, is that the modulus of continuity of a viscosity solution is preserved. 

\begin{corollary}[Modulus of continuity]
Proposition \ref{visc_reg_mod_cty} holds for all Lipschitz viscosity solutions of \eqref{interface_eqn}. 
\end{corollary}

\section{Global well-posedness of the Interface Equation}\label{sec:global_wp}
In this section, we use a vanishing viscosity argument to construct, for arbitrary Lipschitz initial data \(\eta_0\), a global strong solution \(\eta\) to the interface equation \eqref{interface_eqn}. As noted in section~\ref{sec:interface_eqn}, \(W^{1,\infty}(\T)\) is critical for the scaling of \eqref{interface_eqn}, and uniqueness at this regularity is usually delicate. The key point is that, as in \cite{Dong-Gancedo-Nguyen-23}, the resulting strong solution is also a viscosity solution in the sense of Definition~\ref{def:visc_sol_bdry}. Uniqueness then follows from the comparison principle in Proposition~\ref{visc_sol_cp}. We write our main theorem as follows.

\begin{theorem}[Global existence and uniqueness]\label{thm:gwp}
For all $\eta_0\in W^{1,\infty}(\T)$, there exist
\[\eta\in C(\T\times [0,\infty))\cap L_{\mathrm{loc}}^\infty([0, \infty); W^{1,\infty}(\T)),\qquad \partial_t\eta \in L_{\mathrm{loc}}^\infty([0,\infty); L^2(\T)) \]
such that $\eta(\cdot, 0) = \eta_0$ and $\eta$ satisfies the interface equation \eqref{interface_eqn} in $L_{\mathrm{loc}}^\infty([0,\infty); L^2(\T))$ sense. Moreover, $\eta$ is the unique viscosity solution in the sense of Definition \ref{def:visc_sol_bdry}. 
\end{theorem}

Comparison with explicit radial barriers also yields the following stability estimate.

\begin{proposition}[Large-time asymptotics]
Suppose $\eta$ is a solution as in Theorem \ref{thm:gwp}. Then there exist $0< C_1 < C_2$ such that
\begin{equation}
\sqrt{C_1 + 2t}\leq e^{\eta(\cdot, t)} \leq \sqrt{C_2 + 2t}.
\end{equation}
In other words, the fluid region $\Omega(t)$ in equations \eqref{domain_eqn_u} is asymptotically a disk with radius $\sim \sqrt{2t}$ for large $t$. 
\end{proposition}

In the remainder of this section, we unwind the proof of Theorem \ref{thm:gwp}. 
\subsection{Compactness} Let $\eta_0\in W^{1, \infty}(\T)$. For $\varepsilon > 0$ sufficiently small, consider $\eta_0^\varepsilon := \Gamma_\varepsilon *\eta_0$ where $\Gamma_\varepsilon$ is an approximation to identity. Then $\eta_0^\varepsilon\in H^s(\T)$ for any $s > 0$, and by Proposition \ref{mollified_gwp} we have a unique global solution $\eta^\varepsilon$ to the regularized interface equation (\ref{visc_interface_eqn}). The solution $\eta^\varepsilon$ satisfies the estimates
\begin{equation}
\begin{split}
\|\eta^\varepsilon(t)\|_{L^\infty} &\leq \log\sqrt{\exp(2\|\eta_0^\varepsilon\|_{L^\infty}) + {2t}}\leq \log\sqrt{\exp(2\|\eta_0\|_{L^\infty}) + {2t}};\\
\|\eta^\varepsilon(t)\|_{\lip}&\leq \|\eta_0^\varepsilon\|_{\lip} \leq \|\eta_0\|_{\lip},\\
\end{split}
\end{equation}
and it follows from layer potential estimates that
\begin{equation}\label{theta_unif_l2_bound}
\|\theta^\varepsilon(t)\|_{L^2}\leq \cF(\|\eta_0\|_{\lip}).
\end{equation}
Now, let $T > 0$. By Banach-Alaoglu theorem, upon extracting a subsequence $\varepsilon_n\to 0$ and denoting $\eta_n := \eta^{\varepsilon_n}$, $h_n := e^{\eta_n}$, $\theta_n := \theta^{\varepsilon_n}$, we obtain weak* convergences
\begin{equation}
\begin{split}
\eta_n &\wstar \eta\quad \mathrm{in}\quad L^\infty([0, T]; W^{1,\infty}(\T));\\
\theta_n &\wstar \theta\quad \mathrm{in}\quad L^\infty([0, T]; L^2(\T)).
\end{split}
\end{equation}

\subsection{Strong convergence of $\eta_n$ in $C(\T\times [0,T])$} We note that
\begin{equation}
\begin{split}
\|\partial_t\eta_n\|_{L^\infty([0, T]; H^{-1})}&\leq \|e^{-2\eta_n}(G(h_n)\eta_n - 1)\|_{L^\infty([0, T]; H^{-1})} + \|\partial_\alpha^2\eta_n\|_{L^\infty([0, T]; H^{-1})} \\
&\leq \|e^{-2\eta_n}(G(h_n)\eta_n - 1)\|_{L^\infty([0, T]; L^2)} + \|\partial_\alpha\eta_n\|_{L^\infty([0, T]; L^2)} \\
&\leq \cF\big(\|\eta_n\|_{L^\infty([0, T]; W^{1,\infty})}\big)\\
&\leq \cF\big(\|\eta_0\|_{L^\infty([0, T]; W^{1,\infty})}, T\big),
\end{split}
\end{equation}
and that we have the continuous embedding $W^{1,\infty}(\T)\subset C(\T)\subset H^{-1}(\T)$ where the first one is compact. By Aubin-Lions Lemma \ref{aubin-lions},
\begin{equation}
\eta_n\to \eta \quad \mathrm{in}\ C(\T\times [0, T])
\end{equation}
upon extracting a further subsequence if necessary. In particular, we also get
\begin{equation}
h_n\to h \quad \mathrm{in}\ C(\T\times [0, T]),
\end{equation}
where $h := e^\eta$.
\subsection{Equation \eqref{theta_eqn}} We use an argument analogous to the one in \cite{Dong-Gancedo-Nguyen-23} to obtain equation \eqref{theta_eqn}. For completeness, we also present it here. If we multiply a test function $\varphi\in C^\infty(\T\times[0, T])$ to both sides of equation (\ref{theta_eqn_visc}) and integrate by parts, we obtain
\begin{equation}\label{theta_eqn_visc_integral}
\frac{1}{2}\int_0^T\int_\T\varphi(\alpha,t)\theta_n(\alpha, t)\ d\alpha dt + I_1^n = \int_0^T\int_\T \varphi(\alpha, t)\partial_\alpha \eta_n(\alpha, t)\ d\alpha dt,
\end{equation}
where, noting the equivalence of \eqref{K^*_expression_1} and \eqref{K^*_expression_2}, we have
\begin{equation}
I_1^n = \frac{1}{2\pi}\int_0^T\int_\T\int_\T \partial_\alpha\varphi(\alpha, t)\arctan\bigg(\frac{h_n(\alpha)\sin(\alpha) - h_n(\beta)\sin(\beta)}{h_n(\alpha)\cos(\alpha) - h_n(\beta)\cos(\beta)}\bigg)\theta_n(\beta)\ d\beta d\alpha dt.
\end{equation}
The first and last terms of equation (\ref{theta_eqn_visc_integral}) converge to their corresponding limits due to weak* convergence of $\theta_n$ and $\partial_\alpha\eta_n$, so it remains to show that $I_1^n\to I_1$, where
\begin{equation}
I_1 = \frac{1}{2\pi}\int_0^T\int_\T\int_\T \partial_\alpha\varphi(\alpha, t)\arctan\bigg(\frac{h(\alpha)\sin(\alpha) - h(\beta)\sin(\beta)}{h(\alpha)\cos(\alpha) - h(\beta)\cos(\beta)}\bigg)\theta(\beta)\ d\beta d\alpha dt.
\end{equation}
Let $\delta > 0$ be sufficiently small. We write $I_1^n-I_{1}=D_1^n+D_2^n$ where
\begin{align*}
\begin{split}
D_1^n=&\int_0^T\int_\T\theta_n(\beta,t)\int_{\T\cap\{|\alpha|>\delta\}}\partial_\alpha\varphi(\alpha\!+\!\beta,t)\frac{1}{2\pi}\arctan\bigg(\frac{h_n(\alpha+\beta)\sin(\alpha+\beta) - h_n(\beta)\sin(\beta)}{h_n(\alpha+\beta)\cos(\alpha+\beta) - h_n(\beta)\cos(\beta)}\bigg)\ d\alpha d\beta dt\\
&-\int_0^T\!\!\int_\T\theta(\beta,t)\int_{\T\cap\{|\alpha|>\delta\}}\partial_\alpha\varphi(\alpha\!+\!\beta,t)\frac{1}{2\pi}\arctan\bigg(\frac{h(\alpha+\beta)\sin(\alpha+\beta) - h(\beta)\sin(\beta)}{h(\alpha+\beta)\cos(\alpha+\beta) - h(\beta)\cos(\beta)}\bigg)\ d\alpha d\beta dt
\end{split}
\end{align*}
and
\begin{align*}
\begin{split}	
	D_2^n=&\int_0^T\!\!\int_\T\theta_n(\beta,t)\int_{\T\cap\{|\alpha|<\delta\}}\partial_\alpha\varphi(\alpha\!+\!\beta,t)\frac{1}{2\pi}\arctan\bigg(\frac{h_n(\alpha+\beta)\sin(\alpha+\beta) - h_n(\beta)\sin(\beta)}{h_n(\alpha+\beta)\cos(\alpha+\beta) - h_n(\beta)\cos(\beta)}\bigg)\ d\alpha d\beta dt\\
	&-\int_0^T\!\!\int_\T\theta(\beta,t)\int_{\T\cap\{|\alpha|<\delta\}}\partial_\alpha\varphi(\alpha\!+\!\beta,t)\frac{1}{2\pi}\arctan\bigg(\frac{h(\alpha+\beta)\sin(\alpha+\beta) - h(\beta)\sin(\beta)}{h(\alpha+\beta)\cos(\alpha+\beta) - h(\beta)\cos(\beta)}\bigg)\ d\alpha d\beta dt.
\end{split}
\end{align*}

Since $|\arctan(\cdot)|\leq \frac{\pi}{2}$ and $\theta_n$ enjoys the uniform bound (\ref{theta_unif_l2_bound}),
\begin{equation}
|D_2^n|\leq \delta C(\|\eta_0\|_{\lip})\|\partial_\alpha\varphi\|_{L^\infty}
\end{equation}
The strong convergence in $C(\T\times [0, T])$ of $h_n$ combined with the weak* convergence of $\theta_n$ implies that for any $\delta\in (0, 1)$,
\begin{equation}
\lim_{n\to \infty}|D_1^{n}|=0.
\end{equation}
It follows that 
\begin{equation}
\limsup_{n\to\infty} |I_1^n - I_1| \leq \delta C(\|\eta_0\|_{\lip})\|\partial_\alpha\varphi\|_{L^\infty}
\end{equation}
and we get the desired convergence $I_1^n\to I_1$ since $\delta > 0$ is arbitrary. Now, since $\eta\in L^\infty([0, T]; W^{1,\infty}(\T))$, we have $K^*[h]\theta$ well-defined on $L^\infty([0, T]; L^2(\T))$, so we can integrate by parts back to obtain 
\begin{equation}
\begin{aligned}
&\frac{1}{2}\int_0^T\int_\T\varphi(\alpha,t)\theta(\alpha, t)\ d\alpha dt + \int_0^T\int_\T \varphi(\alpha, t)K^*[h]\theta(\alpha, t)\ d\alpha dt \\
=& \int_0^T\int_\T \varphi(\alpha, t)\partial_\alpha \eta(\alpha, t)\ d\alpha dt.
\end{aligned}
\end{equation}
By density of $C^\infty(\T\times[0, T])$ in $L^1([0, T]; L^2(\T))$, we obtain equation \eqref{theta_eqn} as desired.

\subsection{Weak* convergence of $e^{-2\eta_n}G(h_n)\eta_n$ to $e^{-2\eta}G(h)\eta$ in $L^\infty([0, T]; L^2(\T))$} Let $\varphi\in C^\infty(\T\times [0, T])$. Using the integral expression of $G(h_n)\eta_n$ in equation \eqref{DN_mollified} and integrating by parts, we get
\begin{equation}
\begin{split}
&\int_0^T\int_\T G(h_n)\eta_n(\alpha)\varphi(\alpha)\ d\alpha dt\\
&= \frac{1}{4\pi}\int_0^T\int_\T\int_\T \partial_\alpha\varphi(\alpha)\log (h_n(\alpha)^2 + h_n(\beta)^2 - 2h_n(\alpha)h_n(\beta)\cos(\alpha-\beta))\theta_n(\beta)\ d\beta d\alpha dt\\
&=  \frac{1}{4\pi}\int_0^T\int_\T\int_\T \partial_\alpha\varphi(\alpha + \beta)\log (h_n(\alpha+\beta)^2 + h_n(\beta)^2 - 2h_n(\alpha + \beta)h_n(\beta)\cos(\alpha))\theta_n(\beta)\ d\beta d\alpha dt,
\end{split}
\end{equation}
where the second equality follows from a simple change of variable $\alpha \mapsto \alpha + \beta$. Using trigonometric identities, we can write
\begin{equation}
\begin{split}
&h_n(\alpha+\beta)^2 + h_n(\beta)^2 - 2h_n(\alpha + \beta)h_n(\beta)\cos(\alpha)\\
&= h_n(\alpha+\beta)^2 + h_n(\beta)^2 - 2h_n(\alpha + \beta)h_n(\beta)(1 - 2\sin^2(\frac{\alpha}{2}))\\
&= [h_n(\alpha+\beta) - h_n(\beta)]^2 + 4h_n(\alpha + \beta)h_n(\beta)\sin^2(\frac{\alpha}{2}).
\end{split}
\end{equation}
Now, we decompose 
\begin{equation}
\begin{split}
&\log (h_n(\alpha+\beta)^2 + h_n(\beta)^2 - 2h_n(\alpha + \beta)h_n(\beta)\cos(\alpha))\\
&=\log \bigg[(h_n(\alpha+\beta) - h_n(\beta))^2 + 4h_n(\alpha + \beta)h_n(\beta)\sin^2(\frac{\alpha}{2})\bigg]\\
&= \log\bigg[\bigg(\frac{h_n(\alpha+\beta) - h_n(\beta)}{\sin(\frac{\alpha}{2})}\bigg)^2 + 4h_n(\alpha + \beta)h_n(\beta)\bigg] + \log\bigg[\sin^2(\frac{\alpha}{2})\bigg]
\end{split}
\end{equation}
and we write
\begin{equation}
\int_0^T\int_\T G(h_n)\eta_n(\alpha)\varphi(\alpha)\ d\alpha dt = I_2^n + I_3^n,
\end{equation}
where
\begin{equation}
I_2^n =  \frac{1}{4\pi}\int_0^T\int_\T\int_\T \partial_\alpha\varphi(\alpha + \beta)\log\bigg[\bigg(\frac{h_n(\alpha+\beta) - h_n(\beta)}{\sin(\frac{\alpha}{2})}\bigg)^2 + 4h_n(\alpha + \beta)h_n(\beta)\bigg]\theta_n(\beta)\ d\beta d\alpha dt;
\end{equation}
\begin{equation}
I_3^n = \int_0^T\int_\T\int_\T \partial_\alpha\varphi(\alpha + \beta)\log\bigg[\sin^2(\frac{\alpha}{2})\bigg]\theta_n(\beta)\ d\beta d\alpha dt.
\end{equation}
For $I_2^n$, note that
\begin{equation}
\bigg|\frac{h_n(\alpha+\beta) - h_n(\beta)}{\sin(\frac{\alpha}{2})}\bigg|\lesssim \|h_n\|_{\lip} \leq \exp(\|\eta_n\|_{L^\infty})\|\eta_n\|_{\lip},
\end{equation}
and that in view of the bounds (\ref{h_l_infty}) we have
\begin{equation}
\exp\big(-\|\eta_0^{\varepsilon_n}\|_{L^\infty}^2 -{2T}\big)\leq h_n(\alpha + \beta) h_n(\beta) \leq \exp\big(\|\eta_0^{\varepsilon_n}\|_{L^\infty}^2 +{2T}\big).
\end{equation}
Therefore, we can obtain convergence of $I_2^n$ to its respective limit by decomposing into small and large scales as we did for $I_1^n$. We then note that
\begin{equation}
\begin{split}
I_3^n =  -\frac{1}{2}\int_0^T\int_\T \theta_n(\beta)H\varphi(\beta)\ d\beta dt,
\end{split} 
\end{equation}
where $H$ is the Hilbert transform on $\T$. Since $H$ is bounded on $L^2(\T)$, the convergence of $I_3^n$ follows from weak* convergence of $\theta_n$. In summary, we obtain
\begin{equation}\label{DN_limit_integral}
\begin{aligned}
&\lim_{n\to\infty} \int_0^T\int_\T G(h_n)\eta_n(\alpha)\varphi(\alpha)\ d\alpha dt\\
&= \frac{1}{4\pi}\int_0^T\int_\T\int_\T \partial_\alpha\varphi(\alpha)\log (h(\alpha)^2 + h(\beta)^2 - 2h(\alpha)h(\beta)\cos(\alpha-\beta))\theta(\beta)\ d\beta d\alpha dt.\\
\end{aligned}
\end{equation}
Since $\eta\in L^\infty([0, T];W^{1,\infty}(\T))$, $G(h)\eta$ is well-defined in $L^\infty([0, T]; L^2(\T))$. Therefore, integrating by parts on the right-hand side of \eqref{DN_limit_integral}, we get
\begin{equation}
\lim_{n\to\infty} \int_0^T\int_\T G(h_n)\eta_n(\alpha)\varphi(\alpha)\ d\alpha dt = \int_0^T\int_\T G(h)\eta(\alpha)\varphi(\alpha)\ d\alpha dt.
\end{equation}
Again, by density of $C^\infty(\T\times[0, T])$ in $L^1([0, T]; L^2(\T))$, we conclude that $G(h_n)\eta_n \wstar G(h)\eta$ in $L^\infty([0, T]; L^2(\T))$. 
Since $\eta_n\to \eta$ in $C(\T\times[0, T])$, we get that $e^{-2\eta_n}G(h_n)\eta_n \wstar e^{-2\eta}G(h)\eta$ in $L^\infty([0, T]; L^2(\T))$, as desired.

\subsection{Strong solution to the interface equation \eqref{interface_eqn}}
Now, suppose $\varphi\in C_c^\infty(\T\times[0, T))$. Then from \eqref{visc_interface_eqn} we get
\begin{equation}
\begin{aligned}
0 &= -\int_0^T\int_\T \partial_t\varphi(\alpha, t)\eta_n(\alpha, t)\ d\alpha dt + \int_\T\Gamma_{\varepsilon_n}*\eta_0(\alpha)\varphi(\alpha, 0)\ d\alpha - \int_0^T\int_\T e^{-2\eta_n}(\alpha, t)\varphi(\alpha, t)\ d\alpha dt\\
 &+ \int_0^T\int_\T [e^{-2\eta_n}G(h_n)\eta_n](\alpha, t)\varphi(\alpha, t)\ d\alpha dt + \varepsilon_n\int_0^T\int_\T\partial_\alpha\varphi(\alpha, t)\partial_\alpha\eta_n(\alpha, t)\ d\alpha dt
\end{aligned}
\end{equation}
Sending $n\to \infty$, we have
\begin{equation}\label{eqn:int_by_parts_eqn}
\begin{aligned}
0 &= -\int_0^T\int_\T \partial_t\varphi(\alpha, t)\eta(\alpha, t)\ d\alpha dt + \int_\T\eta_0(\alpha)\varphi(\alpha, 0)\ d\alpha\\
 &+ \int_0^T\int_\T [e^{-2\eta}G(h)\eta(\alpha, t)\varphi(\alpha, t)]\ d\alpha dt - \int_0^T\int_\T e^{-2\eta}(\alpha, t)\varphi(\alpha, t)\ d\alpha dt
\end{aligned}
\end{equation}
as desired. Since $\eta\in L^\infty([0, T];W^{1,\infty}(\T))$, $G(h)\eta\in L^\infty([0, T]; L^2(\T))$. By a density argument we can deduce from \eqref{eqn:int_by_parts_eqn} that the interface equation \eqref{interface_eqn} is satisfied in a $L^\infty([0, T); L^2(\T))$ sense.

\subsection{$\eta$ is the Unique Viscosity Solution} Given Lipschitz initial data $\eta_0$, we have shown that for every $T>0$, there exists a solution $\eta$ to the interface equation \eqref{interface_eqn} in a strong $L^\infty([0,T];L^2(\T))$ sense. This alone does not yet yield a full well-posedness theory: without a uniqueness statement, the solution constructed on a given interval $[0,T]$ need not be compatible with the solutions obtained on larger or smaller time intervals. To close this gap, we prove that $\eta$ is the unique viscosity solution with initial data $\eta_0$ in the sense of Definition~\ref{def:visc_sol_bdry}.

Let $\psi$ be a test function as given in Definition \ref{def:visc_sol_bdry} such that $\eta - \psi$ achieves a non-negative global maximum over $\T\times[t_0-r, t_0]$ at $(\alpha_0, t_0)$. Let $\Gamma_\delta$ be an approximation to identity supported on $\T$, and we define 
\[\psi^\delta(t) := \Gamma_\delta * \psi(t) - 2c_\delta, \qquad c_\delta:= \|\Gamma_\delta*\psi - \psi\|_{L^\infty(\T\times[t_0-r, t_0])}.\] Since $\psi\in C([t_0 - r, t_0]; C^{1, 1}(\T))$, we know that $c_\delta \downarrow 0$ as $\delta\downarrow 0$.  Moreover, we note that
\begin{equation}
\max_{\T\times[t_0-r, t_0]} \big(\eta - \psi^\delta\big)\geq \eta(\alpha_0, t_0) - \psi^\delta(\alpha_0, t_0) \geq c_\delta > 0.
\end{equation}
We consider 
\[\tilde{\psi^\delta}(\alpha, t):= \psi^\delta(\alpha, t) + |t -t_0|^2,\quad \tilde{\psi}(\alpha, t):= \psi(\alpha, t) + |t -t_0|^2 \]
then we note that $(\alpha_0, t_0)$ is a strict global maximum of $\eta - \tilde\psi$ over $\T\times [t_0 - r, t_0]$, and  $\eta - \tilde\psi^\delta$ also achieves a strict global maximum 
\begin{equation}
\max_{\T\times[t_0-r, t_0]} \big(\eta - \tilde\psi^\delta\big) \geq c_\delta > 0.
\end{equation}
at $(\alpha_{0,\delta}, t_{0,\delta})$. By uniform convergence, we have $(\alpha_{0,\delta}, t_{0,\delta})\to (a_0, t_0)$ as $\delta \downarrow 0$. We fix $\delta > 0$ for the moment. Since $\eta_n\to \eta$ uniformly, there exists a sequence $(\alpha_{n,\delta}, t_{n,\delta})\to (\alpha_{0,\delta}, t_{0,\delta})$ where $\eta_n - \tilde\psi^\delta$ attains a global maximum 
\[M_{n, \delta}:= \max_{\T\times[t_0 - r, t_0]}\eta_n - \tilde\psi^\delta\]
at $(a_{n, \delta}, t_{n,\delta})$. We assume $n$ is sufficiently large such that $M_{n, \delta} > 0$. We define $\tilde\psi_n^\delta := \tilde\psi^\delta + M_{n, \delta}$, then $\eta_n - \tilde\psi_n^\delta$ attains a maximum $0$ at $(a_{n, \delta}, t_{n,\delta})$. Since $\eta_n - \tilde\psi_n^\delta$ is smooth, at the maximum $(a_{n, \delta}, t_{n,\delta})$ we have the optimality conditions
\begin{equation}\label{second_optimality_1}
e^{-2\tilde\psi_n^\delta} \leq e^{-2\eta_n},\quad\partial_t\tilde\psi_n^\delta\leq \partial_t\eta_n,\quad -\partial_\alpha^2\tilde\psi_n^\delta\leq -\partial_\alpha^2\eta_n
\end{equation}
and by Proposition \ref{taylor_sign} and Proposition \ref{pointwise_cp} we have 
\begin{equation}\label{second_optimality_2}
G(e^{\tilde\psi_n^\delta})\tilde\psi_n^\delta -1 \leq G(\eta_n)\eta_n - 1 \leq 0
\end{equation}
Therefore, since 
\begin{equation}
\partial_t\eta_n + e^{-2\eta_n}(G(e^{\eta_n})\eta_n - 1) - \frac{1}{n}\partial_\alpha^2\eta_n = 0
\end{equation}
we have from \eqref{second_optimality_1} and \eqref{second_optimality_2} that 
\begin{equation}
\partial_t\tilde\psi_n^\delta + e^{-2\tilde\psi_n^\delta}(G(e^{\tilde\psi_n^\delta})\tilde\psi_n^\delta - 1) - \frac{1}{n}\partial_\alpha^2\tilde\psi_n^\delta \leq 0\quad \mathrm{at}\ (\alpha_{n,\delta}, t_{n,\delta})
\end{equation}
Furthermore, since $M_{n,\delta} > 0$, $\tilde\psi^\delta \leq \tilde\psi_n^\delta$, and thus $e^{-2\tilde\psi^\delta}\geq e^{-2\tilde\psi_n^\delta}$. By Proposition \ref{scaling_symm}, we have 
\begin{equation}
\partial_t\tilde\psi^\delta + e^{-2\tilde\psi^\delta}(G(e^{\tilde\psi^\delta})\tilde\psi^\delta - 1) - \frac{1}{n}\partial_\alpha^2\tilde\psi^\delta \leq 0\quad \mathrm{at}\ (\alpha_{n,\delta}, t_{n,\delta})
\end{equation}
for all sufficiently large $n$ depending on $\delta > 0$. In order to send $n\to\infty$ and then $\delta\downarrow 0$, we need the following lemma on pointwise convergence of the Dirichlet-to-Neumann operators. 
\begin{lemma}\label{DN_ptwise_conv}
Let $f:\T\times[t_1, t_2]\to \R$ be such that $\partial_tf\in C(\T\times[t_1, t_2])$ and $f\in C([t_1, t_2]; C^{1, 1}(\T))$. Suppose $\Gamma_n:\T\to\R$ is an approximation to identity and $f_n:=\Gamma_n * f$. If $(\alpha_n, t_n)\to (\alpha_0, t_0)$ in $\T\times [t_1, t_2]$, then 
\begin{equation}
\lim_{n\to\infty}G(e^{f_n(t_n)})f_n(t_n)(\alpha_n) = G(e^{f(t_0)})f(t_0)(\alpha_0)
\end{equation}
where both sides of the equation are classically defined. 
\end{lemma}
\begin{proof}
To begin, we write
\begin{equation}
\begin{aligned}
G(e^{f_n(t_n)})f_n(t_n)(\alpha_n) - G(e^{f(t_0)})f(t_0)(\alpha_0) &= G(e^{f_n(t_n)})f_n(t_n)(\alpha_n) - G(e^{f_n(t_0)})f_n(t_0)(\alpha_n)\\
&+ G(e^{f_n(t_0)})f_n(t_0)(\alpha_n) - G(e^{f(t_0)})f(t_0)(\alpha_n)\\
&+ G(e^{f(t_0)})f(t_0)(\alpha_n)) - G(e^{f(t_0)})f(t_0)(\alpha_0)\\
&=: A_n + B_n + C_n
\end{aligned}
\end{equation}
We note that $C_n\to 0$ since $f\in C([t_1, t_2]; C^{1, 1}(\T))$. For $A_n$, we note that 
\begin{equation}
\begin{aligned}
&|G(e^{f_n(t_n)})f_n(t_n)(\alpha_n) - G(e^{f_n(t_0)})f_n(t_0)(\alpha_n)|\\
&\leq  \|G(e^{f_n(t_n)})(f_n(t_n) - f_n(t_0))\|_{L^\infty} +\|[G(e^{f_n(t_n)}) - G(e^{f_n(t_0)})]f_n(t_0)\|_{L^\infty}\\
&\leq \|G(e^{f_n(t_n)})(f_n(t_n) - f_n(t_0))\|_{H^2} +\|[G(e^{f_n(t_n)}) - G(e^{f_n(t_0)})]f_n(t_0)\|_{H^2}\\
&\leq \cF(\|f_n(t_n)\|_{H^2}, \|f_n(t_0)\|_{H^2})\|f_n(t_n) - f_n(t_0)\|_{H^2}\\
&\leq \cF(\|f(t_n)\|_{H^2}, \|f(t_0)\|_{H^2})\|f(t_n) - f(t_0)\|_{H^2}
\end{aligned}
\end{equation}
where we used Sobolev embedding for the second inequality, Proposition \ref{DN_boundedness} and Proposition \ref{DN_contraction} for the third inequality, and properties of approximation to identity in the last inequality. Then $A_n\to 0$ since $f\in C([t_1, t_2]; C^{1, 1}(\T))$. For $B_n$, we employ a similar strategy and obtain
\begin{equation}
\begin{aligned}
&\|G(e^{f_n(t_0)})f_n(t_0) - G(e^{f(t_0)})f(t_0)\|_{L^\infty}\\
&\leq \|G(e^{f_n(t_0)})(f_n(t_0) - f(t_0))\|_{H^2} +\|[G(e^{f_n(t_0)}) - G(e^{f(t_0)})]f(t_0)\|_{H^2}\\
&\leq \cF(\|f_n(t_0)\|_{H^2}, \|f(t_0)\|_{H^2})\|f_n(t_0) - f(t_0)\|_{H^2}\\
\end{aligned}
\end{equation}
Then $B_n\to 0$ as $n\to\infty$ using properties of approximation to identity. This finishes the proof of the lemma.
\end{proof}

Now, sending $n\to\infty$ and using Lemma \ref{DN_ptwise_conv}, we get
\begin{equation}
\partial_t\tilde\psi^\delta + e^{-2\tilde\psi^\delta}(G(e^{\tilde\psi^\delta})\tilde\psi^\delta - 1)\leq 0\quad \mathrm{at}\ (\alpha_{0,\delta}, t_{0,\delta})
\end{equation}
Moreover, since $\partial_t\tilde\psi^\delta (\alpha_{0,\delta}, t_{0,\delta}) = \partial_t\psi^\delta(\alpha_{0,\delta}, t_{0,\delta})$ $e^{-2\psi^\delta}\geq e^{-2\tilde\psi^\delta}$, by Proposition \ref{scaling_symm} and Proposition \ref{taylor_sign}, we have
\begin{equation}
\partial_t\psi^\delta + e^{-2\psi^\delta}(G(e^{\psi^\delta})\psi^\delta - 1)\leq 0\quad \mathrm{at}\ (\alpha_{0,\delta}, t_{0,\delta})
\end{equation}
Since $\psi\in C((0, T); C^{1, 1}(\T)$ and $\partial_t\psi\in C(\T\times(0, T))$, sending $\delta\to 0$ and applying Lemma \ref{DN_ptwise_conv}, we get
\begin{equation}
\partial_t\psi + e^{-2\psi}(G(e^\psi)(\psi) - 1) \leq 0 \quad\mathrm{at}\ (a_0, t_0)
\end{equation}
as desired. The proof that $\eta$ is a viscosity supersolution is similar. 

\section{Viscosity Solutions in Eulerian Coordinates}\label{sec:visc_sol_domain}
In this section, we bring our focus back to the Hele-Shaw equations in Eulerian coordinates. We observe from \eqref{domain_eqn_p} that the Hele-Shaw problem with an injection at the origin can be equivalently formulated as follows:
\begin{equation}\label{Hele-Shaw_Kim}
\left\{
\begin{array}{ll}
      -\Delta p = 2\pi \delta & \mathrm{in}\ \{p> 0\}\\
      V = \partial_tp/|\nabla p| = -(\nabla p)\cdot n = |\nabla p| & \mathrm{on}\ \partial\{p > 0\}
\end{array}
\right.
\end{equation}

Given a general initial data $p(\cdot, 0) = p_0$, classical solutions to \eqref{Hele-Shaw_Kim} might not exist globally, so the notion of viscosity solution corresponding to \eqref{Hele-Shaw_Kim} was introduced first by Kim in \cite{Kim2003}.
\begin{definition}[Viscosity solutions in the domain]\label{def:visc_sol_domain} Let $Q := \R^2\bs\{0\}\times (0, \infty)$. 
A nonnegative function
\[
p:\R^2\bs\{0\}\times[0,\infty)\to [0,\infty)
\]
is called a viscosity subsolution of \eqref{Hele-Shaw_Kim} with initial data $p_0$ if the following hold:
\begin{enumerate}
    \item $p(\cdot,0)=p_0$ on $\R^2\bs\{0\}$;
    \item for every $t_0 > 0$, 
    \[\limsup_{(x, t)\to (0, t_0)} \big(p(x, t) + \log|x|\big) \leq 0; \]
    \item $\overline{\{p>0\}}\cap\{t=0\}=\overline{\{p_0>0\}}$;
    \item for each $T\ge 0$, the set $\overline{\{p>0\}}\cap\{t\le T\}$ is bounded;
    \item let $q\in C^{2,1}(Q)$ be such that $p(x, t) - q(x,t)$ attains a local maximum in $\overline{\{p>0\}}\cap\{t\le t_0\}\cap Q$ at $(x_0,t_0)$, the following hold:
    \begin{enumerate}[(a)]
        \item $-\Delta q(x_0,t_0)\le 0$ if $p(x_0,t_0)>0$;        
        \item $\min\bigl(-\Delta q,\ q_t-|\nabla q|^2\bigr)(x_0,t_0)\le 0$ if $(x_0,t_0)\in \partial\{p>0\}$.
    \end{enumerate}
\end{enumerate}
A nonnegative function
\[
p:\R^2\bs\{0\}\times[0,\infty)\to [0,\infty)
\]
is called a viscosity supersolution of \eqref{Hele-Shaw_Kim} with initial data $p_0$ if the following hold:
\begin{enumerate}
    \item $p(\cdot,0)=p_0$ on $\R^2$;
    \item for every $t_0 > 0$, 
    \[\liminf_{(x, t)\to (0, t_0)} \big(p(x, t) + \log|x|\big) \geq 0 \]
    \item $\overline{\{p>0\}}\cap\{t=0\}=\overline{\{p_0>0\}}$;
    \item for each $T\ge 0$, the set $\overline{\{p>0\}}\cap\{t\le T\}$ is bounded;
    \item let $q\in C^{2,1}(Q)$ be such that $p(x, t)-q(x, t)$ attains a local minimum in $\overline{\{p>0\}}\cap\{t\le t_0\}\cap Q$ at $(x_0,t_0)$, the following hold:
    \begin{enumerate}[(a)]
        \item $-\Delta q(x_0,t_0)\geq 0$ if $p(x_0,t_0)>0$;        
        \item $\max\bigl(-\Delta q,\ q_t-|\nabla q|^2\bigr)(x_0,t_0)\ge 0$ if $(x_0,t_0)\in \partial\{p>0\}$, $|\nabla q(x_0, t_0)|\neq 0$, and 
        \[\{q > 0\}\cap \{p > 0\}\cap B\]
        is non-empty for every ball $B$ centered at $(x_0, t_0)$.
    \end{enumerate}
\end{enumerate}

A nonnegative function $p: \R^2\bs\{0\}\to [0,\infty)$ is called a viscosity solution of \eqref{Hele-Shaw_Kim} if it is both a viscosity subsolution and a viscosity supersolution.
\end{definition}

\subsection{Relationship between different notions of viscosity solutions} In section \ref{sec:visc_sol}, we also defined our notion of viscosity solutions for the interface equation
\[\partial_t\eta + e^{-2\eta}(G(h)\eta - 1) = 0,\qquad h = e^\eta.\] 
One natural question to ask is how viscosity solutions of the interface equation, as in Definition \ref{def:visc_sol_bdry}, relate to viscosity solutions of the domain equations, as in Definition \ref{def:visc_sol_domain}. We give a positive answer to this question by giving a direct proof that a viscosity solution in the sense of Definition \ref{def:visc_sol_bdry}, under some regularity assumptions, correspond to a viscosity solution in the sense of Definition \ref{def:visc_sol_domain}. We want to again mention that such observations are by no means unprecedented; Chang-Lara, Guillen, and Schwab \cite{ChangLaraGuillenSchwab2019} have an in-depth discussion of how viscosity solutions of the interface equations and viscosity solutions in the domain are related for a general class of free boundary problems with graph domains. 

We now state our main proposition. The additional regularity is to deal with possible degeneracy of the test function.

\begin{proposition}\label{prop:visc_sol_equiv}
Suppose $\eta$ is a viscosity solution as in
Definition~\ref{def:visc_sol_bdry}. Assume in addition that
\begin{equation}\label{regularity_eta}
\eta\in C(\T\times [0,\infty))
\cap L_{\mathrm{loc}}^\infty([0,\infty); W^{1,\infty}(\T)),
\qquad
\partial_t\eta \in L_{\mathrm{loc}}^\infty([0,\infty); L^2(\T)).
\end{equation}
Let $\phi$ be the
unique harmonic extension of $\eta$ in
\[
\Omega_\eta(t):=\{re^{i\alpha}: 0\leq r < e^{\eta(\alpha,t)},\
\alpha\in\T\}.
\]
Define
\[
p(x,t):=\max\{\phi(x,t)-\log|x|,0\}.
\]
Then $p$ is a viscosity solution as in
Definition~\ref{def:visc_sol_domain}.
\end{proposition}
\begin{proof}
We only verify that $p$ is a viscosity subsolution, since the case of supersolution is analogous. Let $q$ be a test function as in Definition \ref{def:visc_sol_domain}. Suppose $p - q$ has a local maximum at $(x_0, t_0)\in \R^2\bs\{0\}\times (0, \infty)$.

\noindent\emph{Case 1: interior point.} If $p(x_0, t_0) > 0$, $x_0\in \Omega_\eta(t_0)$, so
\[-\Delta (p(x_0, t_0) - q(x_0, t_0)) \geq 0 \]
by the standard second optimality condition. Note here that $p(\cdot, t_0)$ is $C^\infty$ at $x_0$ by elliptic regularity. Then, since $p$ is harmonic in $\Omega_\eta\bs\{0\}$, we have
\[-\Delta q(x_0, t_0)\leq 0\]
as desired. 

\noindent\emph{Case 2: boundary point with non-degenerate test function.} Now we suppose $p(x_0, t_0) = 0$, so that $x_0\in \partial\Omega_\eta(t_0)$. We want to show that
\begin{equation}\label{domain_visc_sub_sol_condition}
\min[-\Delta q(x_0, t_0), (q_t - |\nabla q|^2)(x_0, t_0)] \leq 0.
\end{equation}
Since our discussions concern only local properties of $p$ and $q$ around $(x_0, t_0)$, we may modify $q$ away from $(x_0, t_0)$ and assume that $(x_0, t_0)$ is actually a global maximum of $p - q$. Translating $q$ if necessary, we may also assume that $p-q = 0$ at $(x_0, t_0)$, and it follows that $p\leq q$ and $p = q$ at $(x_0, t_0)$. In this case, we impose the non-degeneracy condition $|\nabla q(x_0, t_0)| > 0$ so we can also assume $\{x:q(\cdot, t) = 0\}$ is star-shaped for every $t\in [t_0 - \delta, t_0]$. The reason is more involved so we phrase it as Lemma \ref{lem:starshaped_reduction_nondeg}. Let $\psi$ be such that
\[\{x: q(x, t) = 0\} = \{e^{\psi(\alpha, t)}e^{i\alpha}: \alpha\in \T\},\]
where $t\in [t_0 - \delta, t_0]$, and let $x_0 = e^{\psi(\alpha_0, t_0)}e^{i\alpha_0}$. We know that $\psi$ is in particular $C^{1, 1}$ at $(\alpha_0, t_0)$ since $q$ is $C^{2, 1}$ at $(x_0, t_0)$. Moreover, $\eta - \psi$ achieves a non-positive global maximum at $(\alpha_0, t_0)$ over $\T\times [t_0-\delta, t_0]$. Since $\eta$ is a viscosity solution, we have
\begin{equation}\label{visc_sol_condition}
[\partial_t\psi + e^{-2\psi}(G(e^\psi)\psi - 1)](\alpha_0, t_0) \leq 0.
\end{equation}

At the mean time, differentiating both sides of $q(e^{i\alpha_0 + \psi(\alpha_0, t)}, t) = 0$ with respect to $t$ at $t_0$, we get
\begin{equation}\label{q_time_derivative_eqn}
\partial_t q(e^{i\alpha_0 + \psi(\alpha_0, t_0)},t_0) + (\nabla q\cdot e^{i\alpha_0 + \psi(\alpha_0, t_0)})  \partial_t \psi(\alpha_0, t_0) = 0.
\end{equation}
We let $\Phi(\cdot, t)$ be the harmonic extension of $\psi(\cdot, t)$ inside the domain $\{q(\cdot, t) > 0\}$, where $t\in [t_0-\delta, t_0]$, and $\tilde{q}(x, t) := \Phi(x, t) -\log |x|$. Then straightforward calculation yields 
\begin{equation}\label{Taylor_sign_q_tilde}
G(e^\psi)\psi(\alpha_0, t_0) - 1 = N(e^{i\alpha_0 + \psi(\alpha_0, t_0)})\cdot \nabla \tilde{q}(x_0, t_0)
\end{equation}
where we recall that $N$ is the unnormalized vector field $N := -i\partial_\alpha(e^{i\alpha + \psi(\alpha, t)})$. If $-\Delta q(x_0, t_0)\leq 0$, \eqref{domain_visc_sub_sol_condition} is automatically satisfied, so we may assume $-\Delta q(x_0, t_0) > 0$. Then $q - \tilde q$ is superharmonic in $\{q(\cdot, t_0) > 0\}$. Since $q - \tilde{q} = 0$ on $\{q(\cdot, t_0) = 0\}$, we have $q - \tilde{q} \geq 0$ in $\{q(\cdot, t_0) > 0\}$ and the outward normal derivative
\begin{equation}
[n\cdot\nabla (q - \tilde{q})](x_0, t_0) \leq 0.
\end{equation}
Together with equation \eqref{Taylor_sign_q_tilde} we get
\begin{equation}
G(e^\psi)\psi(\alpha_0, t_0) - 1 \geq N(e^{i\alpha_0 + \psi(\alpha_0, t_0)})\cdot \nabla {q}(x_0,t_0).
\end{equation}
We note that
\begin{equation}\label{normal_dot_product_radial}
\begin{aligned}
N(e^{i\alpha + \psi(\alpha, t)})\cdot e^{i\alpha + \psi(\alpha, t)} = e^{2\psi(\alpha, t)}.
\end{aligned}
\end{equation}
Moreover, since $\nabla q$ only has a negative normal component, we have from \eqref{normal_dot_product_radial} that $\nabla q(x_0, t_0) \cdot e^{i\alpha_0 + \psi(\alpha_0, t_0)} \leq 0$. Using \eqref{visc_sol_condition}, we then have
\begin{equation}
\begin{aligned}
\partial_t\psi(\alpha_0, t_0)\nabla q(x_0, t_0) \cdot e^{i\alpha_0 + \psi(\alpha_0, t_0)} &\geq -\nabla q(x_0, t_0) \cdot e^{i\alpha_0 + \psi(\alpha_0, t_0)}[e^{-2\psi}(G(e^\psi)\psi -1)](\alpha_0, t_0). 
\end{aligned}
\end{equation}
Since $G(e^\psi)\psi - 1 \geq N\cdot\nabla q$, we may use the identity $\nabla q = -|\nabla q|n$ to reorganize the right-hand side as
\begin{equation}
\begin{aligned}
-|\nabla q(x_0, t_0)|^2e^{-2\psi(\alpha_0, t_0)} N(e^{i\alpha_0 + \psi(\alpha_0, t_0)})\cdot e^{i\alpha_0 + \psi(\alpha_0, t_0)} = -|\nabla q|^2,
\end{aligned}
\end{equation}
where the equality is given by \eqref{normal_dot_product_radial}. Now, recalling \eqref{q_time_derivative_eqn}, we have
\begin{equation}
\begin{aligned}
0 &= \partial_t q(x_0, t_0) + \nabla q(x_0, t_0)\cdot e^{i\alpha_0 + \psi(\alpha_0, t_0)}\partial_t\psi(\alpha_0, t_0)\\
  &\geq \partial_t q(x_0, t_0) - |\nabla q|^2(x_0, t_0)
\end{aligned}
\end{equation}

\noindent\emph{Case 3: boundary point with degenerate test function.} We are left with the degenerate case $\nabla q(x_0, t_0) = 0$. Again, we may assume that $p\leq q$ and $p = q$ at $(x_0, t_0)$. The only bad situation is when $-\Delta q(x_0, t_0) > 0$ and $\partial_t q(x_0, t_0) > 0$, and we show by contradiction that it cannot happen. By regularity assumptions \eqref{regularity_eta}, we have for $h: = e^\eta$ that $\partial_t h\in L^\infty([0, t_0]; L^2(\T))$ and $h\in L^\infty([0, t_0]; W^{1,\infty}(\T))$. For convenience of notation, we assume
\begin{equation}
\sup_{0\leq t\leq t_0}\Big(\|\partial_th(\cdot, t)\|_{L^2} + \|h(\cdot, t)\|_{W^{1,\infty}}\Big)\leq L.
\end{equation}
Then for $0< s < t\leq t_0$,
\begin{equation}
\begin{aligned}
\|h(\cdot, t) - h(\cdot, s)\|_{L^2}^2 &= \int_\T \bigg|\int_s^t \partial_\tau h(\alpha, \tau)\ d\tau\bigg|^2\ d\alpha \\
&\leq |t - s|\int_\T\int_s^t |\partial_\tau h(\alpha, \tau)|^2\ d\tau\\
&\leq M^2|t - s|^2
\end{aligned}
\end{equation}
where the first inequality follows from Cauchy--Schwarz. Then by interpolation inequalities in Lemma \ref{lem:1d_sobolev_interpolation} we have
\begin{equation}\label{eqn: h_time_diff}
\|h(\cdot, t) - h(\cdot, s)\|_{L^\infty} \leq C(|t - s|^{\frac{2}{3}} + |t - s|)
\end{equation}
for some $C$ depending only on $L$. Now, since $c := \partial_t q(x_0, t_0)> 0$ and $\nabla q(x_0, t_0) = 0$, we obtain from Taylor expansion that, for sufficiently small $\delta >0$ and $x$ near $x_0$,
\begin{equation}
\begin{aligned}
q(x, t_0 - \delta) &= q(x, t_0 - \delta) - q(x_0, t_0)\\
&\leq -\frac{c}{2}\delta + M|x-x_0|^2
\end{aligned}
\end{equation}
Note that $M$ here can be chosen to only depend on the local $C^{2,1}$ constant of $q$ and not relying specifically on $\delta$. Therefore, if $|x - x_0| < \varepsilon\sqrt{\delta}$ for sufficiently small $\varepsilon > 0$ depending on $M$ and $c$, we have $q(x, t_0 - \delta) < 0$. Define $x' := h(\alpha_0, t_0 - \delta)e^{i\alpha_0}$, then for all $\delta > 0$ sufficiently small depending on $\varepsilon$ and $C$, we have from \eqref{eqn: h_time_diff} that
\begin{equation}
|x' - x| \leq C\delta^{\frac{2}{3}} < \varepsilon\sqrt{\delta}
\end{equation}
and thus $q(x', t_0-\delta) < 0$. However, $p(x', t_0-\delta) = 0$, contradicting $p\leq q$. 
\end{proof}

\begin{lemma}[Star-shaped reduction for nondegenerate tests]
\label{lem:starshaped_reduction_nondeg}
Let $p$ be as in Proposition \ref{prop:visc_sol_equiv}. Let
$q\in C^{2,1}$ be such that
\[
p-q \le 0
\quad\text{in } \overline{\{p>0\}}\cap\{t\le t_0\}\cap Q,
\quad
(p-q)(x_0,t_0)=0, \quad \nabla_x q(x_0,t_0)\neq 0,
\]
for some $(x_0,t_0)\in \partial\{p>0\}\cap Q$, with $x_0\neq 0$. Then there exists some $\tilde q \in C^{2,1}$ such that 
\begin{enumerate}
    \item $\tilde{q} = q$ in a neighborhood of $(t_0, x_0)$;
    \item $\{x:\tilde q(x,t)>0\}$ is a star-shaped domain for every $t$ sufficiently close to $t_0$.
\end{enumerate} 
\end{lemma}

\begin{proof}
Write $h=e^\eta$, $x_0=h(\alpha_0,t_0)e^{i\alpha_0}$, and
$y_0=\eta(\alpha_0,t_0)$. Denote $q^\sharp(\alpha,y,t):=q(e^{y+i\alpha},t)$. We first show that
$\partial_y q^\sharp(\alpha_0,y_0,t_0)\neq0$. Suppose on the contrary that $\partial_y q^\sharp(\alpha_0,y_0,t_0)=0$, then since
$p\le q$ and $p=0$ on the free boundary, the trace $q_\eta^\sharp(\alpha, t_0) := q^\sharp(\alpha,\eta(\alpha,t_0),t_0)$ has a local minimum at $\alpha_0$. Then
\begin{equation}
\begin{aligned}
0 &= \partial_\alpha q_\eta^\sharp(\alpha_0, t_0)\\
  &= \partial_\alpha q^\sharp(\alpha_0, \eta(\alpha_0, t_0),t_0) + \partial_\alpha\eta(\alpha_0, t_0)\partial_y q^\sharp(\alpha_0, \eta(\alpha_0, t_0),t_0)\\
  &= \partial_\alpha q^\sharp(\alpha_0, \eta(\alpha_0, t_0),t_0),
\end{aligned}
\end{equation}
which would then imply $\nabla_xq(x_0,t_0) = 0$ after a change of variable. In fact, $\partial_y q^\sharp(\alpha_0,y_0,t_0)<0$. To see this, we note that for $y<y_0$ close to $y_0$, $x := e^{y+i\alpha_0}$, we have $q(x, t_0) \geq p(x, t_0) > 0$ while
$q(x_0,t_0)=0$. Thus the implicit function theorem gives neighborhoods $I$ of $\alpha_0$
and $J$ of $t_0$, and a function $\rho\in C^{2,1}(I\times J)$, such that near $(x_0,t_0)$,
\begin{equation}
\{q=0\}=\{e^{\rho(\alpha,t)}e^{i\alpha}:(\alpha,t)\in I\times J\}.
\end{equation}
Choose $I'\Subset I$, $J'\Subset J$, and extend $\rho|_{I'\times J'}$ to
a positive function $\widetilde\rho\in C^{2,1}(\mathbb T\times J')$ with
$\widetilde\rho=\rho$ on $I'\times J'$. Define
$g(e^{y+i\alpha},t):=\widetilde\rho(\alpha,t)-y$. Then
\begin{equation}
\{g(\cdot,t)>0\} = \{re^{i\alpha}:0<r<e^{\widetilde\rho(\alpha,t)}\}
\end{equation}
is smooth and star-shaped. Since $q$ and $g$ have the same zero set and the same sign near $(x_0,t_0)$, we choose $\chi\in C_c^\infty$ supported in this neighborhood with $\chi\equiv1$ near $(x_0,t_0)$, and define $\widetilde q:=\chi q+(1-\chi)g$. Then $\widetilde q=q$ near
$(x_0,t_0)$. Also, $\widetilde q$ has the same sign as $g$, so for
$t$ close to $t_0$,
\begin{equation}
\{\widetilde q(\cdot,t)>0\}=\{g(\cdot,t)>0\},
\end{equation}
which is star-shaped.
\end{proof}

\subsection{Past literature on viscosity solutions} Viscosity solutions provide another robust framework for Hele-Shaw flows
beyond classical regularity. Kim~\cite{Kim2003} introduced a viscosity
formulation for the one-phase Hele-Shaw and Stefan problems and proved
existence and uniqueness of viscosity solutions. Choi and Kim~\cite{ChoiKim2006}
used this framework to study waiting-time phenomena, giving criteria in
terms of the growth of the initial harmonic pressure near a boundary point
which distinguish waiting from immediate motion. 

A second theme in the viscosity literature is regularization. Kim~\cite{Kim2006Reg}
proved that a Lipschitz free boundary becomes smooth near points where the
solution is Lipschitz and nondegenerate, while Kim~\cite{Kim2006Long} showed
that, after finite time, viscosity solutions become Lipschitz with
nondegenerate free-boundary speed and hence eventually smooth. Choi, Jerison,
and Kim further developed this theory: in \cite{ChoiJerisonKim2007} they
constructed global viscosity solutions starting from star-shaped Lipschitz
initial domains, proved that the positive phase remains star-shaped and
Lipschitz in space for all time, and obtained quantitative arrival-time and
speed estimates for the free boundary; in \cite{ChoiJerisonKim2009} they
proved scale-invariant local regularization results and derived further
existence, uniqueness, and regularity consequences for global solutions with
Lipschitz initial free boundary. 

We revisit several classical results on viscosity solutions (in the sense of Definition \ref{def:visc_sol_domain}) of the Hele-Shaw problem. The injection considered in these results are usually from a fixed boundary (see, for example, \cites{Kim2003}), but the proofs are local in nature and can be adapted to our point injection setting. Then, by Proposition \ref{prop:visc_sol_equiv}, they apply to our solution in Theorem \ref{thm:gwp}. To fix notations, let $p(x, 0) = p_0(x)$, and assume that $\Omega_0 := \{p_0 > 0\}$ is Lipschitz. We also define $\Gamma_0:=\partial\Omega_0$.

\subsubsection{Corners on the initial interface} 
The first results we present concerns the waiting time phenomenon of acute-angled corners and immediate movement of obtuse angled corners on the initial interface.

\begin{theorem}[Choi--Jerison--Kim \cites{ChoiKim2006, ChoiJerisonKim2007, ChoiJerisonKim2009}] Let $\eta_0\in \lip(\T)$ and $\eta$ be the unique global solution given by Theorem \ref{thm:gwp} such that $\eta(\cdot, 0) = \eta_0$. 
\begin{enumerate}
    \item \emph{(Waiting time.)} Suppose that for some $x_0 = e^{\eta_0(\alpha_0)}e^{i\alpha_0}$, there exists $\delta > 0$ such that $B(x_0, \delta)\cap \Omega_{\eta_0}$ is contained in a cone centered at $x_0$ with opening angle strictly less than $\pi/ 2$. Then there is some $T > 0$ such that $x_0\in \Gamma_t$ for all $0\leq t\leq T$.
    \item \emph{(Immediate movement.)} Suppose that for some $x_0 = e^{\eta_0(\alpha_0)}e^{i\alpha_0}$, there exists $\delta > 0$ such that $B(x_0, \delta)\cap \Omega_{\eta_0}$ contains a cone centered at $x_0$ with opening angle strictly larger than $\pi/ 2$. Then for all sufficiently small $t > 0$, $x_0\not\in \Gamma_t$.
\end{enumerate}
\end{theorem} 

The corner-dynamics results above should also be viewed in relation to the
classical explicit-solution literature for source-driven Hele-Shaw flow.
Using complex-variable and self-similar methods, Howison~\cite{Howison1986},
Howison--King~\cite{HowisonKing1989}, and King--Lacey--V\'azquez
\cite{KingLaceyVazquez1995} studied the evolution of singular free-boundary
geometries such as cusps and corners. In particular, King--Lacey--V\'azquez
\cite{KingLaceyVazquez1995} used special self-similar solutions to describe
corner dynamics, already revealing the acute/obtuse dichotomy.

For the gravity-driven Muskat problem, there are also extensive literature on the angle dynamics of Lipschitz interfaces. In the one-phase case, Agrawal--Patel--Wu \cite{agrawal-patel-wu} proved local-well-posedness of the system \eqref{domain_eqn_u} in a function space that allows the initial interface to have acute-angled corners and cusps. They also showed that for such cornered initial data the angle remains unchanged for at least a short time. Their methodology, which involves Riemann mapping parametrization of the boundary and energy estimates in weighted Sobolev spaces, first appeared in the work \cite{Wu2019WaterWaves} of Wu for water waves. On the other hand, in \cite{AlazardKoch2025HeleShaw}, Alazard--Koch studied the one-phase Muskat problem as a semi-flow. They showed that on initial interfaces of these semi-flow solutions, acute-angled corners exhibit a waiting-time phenomenon while obtuse-angled corners have immediate movement. Their argument utilizes monotonicity/comparison principle of semi-flows. Regarding the two-phase Muskat problem with equal viscosities, Garc\'ia-Ju\'arez--G\'omez-Serrano--Haziot--Pausader \cite{GarciaJuarezGomezSerranoHaziotPausader2024} showed desingularization of obtuse corners with small slope.

\subsubsection{Lipschitz star-shaped viscosity solution and regularity estimates} Lastly, we want to mention the following important result of Choi--Jerison--Kim \cite{ChoiJerisonKim2007}, which is analogous to Theorem \ref{thm:gwp}.

\begin{proposition}[Lemma 2.4, \cite{ChoiJerisonKim2007}]\label{lem:Lipschitz_visc_sol}
Let $\Omega_0$ be Lipschitz and star-shaped. Then there exists a viscosity solution (as in Definition \ref{def:visc_sol_domain}) $p(x,t)$ of the Hele-Shaw equations with initial positive phase $\Omega_0$. Moreover, $\Omega(t) := \{x: p(x, t) > 0\}$ is star-shaped Lipschitz for all $t>0$. 
\end{proposition}

In their work, existence of viscosity solutions with more general initial data was first proven using Perron's method (see \cites{Kim2003, JerisonKim2005}), and preservation of Lipschitz star-shapedness follows from comparison principle with sub- and supersolutions. They also proved several regularity results regarding the free boundary. 

The first result they obtained has to do with speed of the free boundary. Define 
\[T(x) := \sup\{t > 0: p(x, t) = 0\}\]
which is the time it takes for the free boundary to reach $x$. In our setting, consider $x_\pm:=h(\alpha, 0)e^{i\alpha} \pm \delta n(\alpha,0)$ for some sufficiently small $\delta > 0$.  It was shown in \cite{ChoiJerisonKim2007} that
\begin{equation}
\frac{\delta}{T(x_+)}\sim |\nabla p(x_-, 0)|.
\end{equation}
In other words, the average normal speed of the free boundary across a $\delta$-neighborhood is comparable to the initial gradient of $p$ at that scale.  

The second result is given in the following proposition.
\begin{proposition}[Cone Monotoniticy, \cite{ChoiJerisonKim2007}]
\label{thm:CJK_cone_monotonicity}
Let \(p\) be a viscosity solution (in the sense of Definition \ref{def:visc_sol_domain}) of the Hele-Shaw equations, and let \(x_0\in\Gamma_0\). Suppose that $\Gamma_0$ has Lipschitz constant \(M<1\) in a coordinate neighborhood near $x_0$, and the initial pressure
\(p_0\) is monotone in a cone
\[
W(\theta, e):=\{q\in\mathbb R^2:(q-x_0)\cdot e\ge |q - x_0|\cos\theta\},
\quad
\mathrm{where\ }\frac{\pi}{4}<\theta\le \frac{\pi}{2},\ e\in \R^2
\]
near \(x_0\). Then there exist a smaller cone aperture $\theta_* \in \left(\frac{\pi}{4},\theta\right)$, a radius \(r_*>0\), and a time \(T_*>0\), depending only on the local Lipschitz geometry and on \(\theta\), such that $p(\cdot,t)$ is monotone in the cone $W(\theta_*, e)\cap B(x_0, r_*)$, for every \(0\le t\le T_*\).
\end{proposition}

In star-shaped variables, this cone monotonicity gives a local version of
our bound Corollary \ref{cor_Lipschitz_bound}. At a regular boundary point,
we know that $\nabla p$ only has normal component, so
\[
\frac{\nabla p}{|\nabla p|}=-\nu,\quad \mathrm{where}\ \nu=\frac{e_r-\eta_\alpha e_\theta}{\sqrt{1+\eta_\alpha^2}}.
\]
Monotonicity of \(p\) in the cone \(W(\theta_*,-e_r(\alpha_0))\) means that
\(\nabla p\cdot q\ge 0\) for every direction \(q\) in this cone. Direct calculation then gives
\begin{equation}
|\eta_\alpha|\le \cot\theta_* .
\end{equation}

The third result concerns regularization of free boundary when the local Lipschitz constant is small. 
\begin{theorem}[Local-in-time regularization, \cite{ChoiJerisonKim2007}]
\label{thm:CJK_classical_regularization}
Let \(p\) be a viscosity solution (in the sense of Definition \ref{def:visc_sol_domain}) of the Hele-Shaw equations, and suppose that
\(\Gamma_0\) has sufficiently small local Lipschitz constant. Then there exists \(T>0\), depending only on the Lipschitz geometry of \(\Omega_0\), such that for every \(0<t\le T\), the free boundary \(\Gamma(t)\) is smooth in space and time, \(p\) is smooth up to the free
boundary from the positive phase, and the Hele-Shaw free-boundary condition
\[
\partial_tp=|\nabla p|^2
\quad\text{on }\Gamma(t)
\]
holds in the classical sense.
\end{theorem}
In particular, our solution $\eta$ as in Theorem \ref{thm:gwp} is instantaneously smooth for small $t > 0$ given that $\|\partial_\alpha\eta\|_{L^\infty}$ is sufficiently small. 

\newpage
\appendix

\section{Dirichlet-to-Neumann Operators on Different Domains}\label{appendix:DN}
We recall the following definitions of Dirichlet-to-Neumann operators in graph domain and star-shaped domain, respectively. They should be understood in a weak sense if necessary. We use slightly different notations here to highlight their structural similarities while emphasizing the type of domains on which they are defined.

\begin{definition}[Dirichlet-to-Neumann operator, graph domain]\label{graph_DN}
Suppose $\Omega_\eta^g:=\{(x, y): y < \eta(x)\}$, where $\eta\in \lip(\T)$. We define the graph domain Dirichlet-to-Neumann operator $G_g(\eta)f(x)$ as 
\begin{equation}
(G_g(\eta)f)(x)=N_g\cdot\nabla \phi(x,\eta(x)),
\end{equation}
where $N_g(x) := (-\partial_x\eta, 1)$, and $\phi$ solves the boundary value problem
\begin{equation}
\left\{
\begin{array}{ll}
 \Delta\phi = 0 & \mathrm{in}\ \Omega_\eta^g  \\
 \phi(x, \eta(x)) = f(x) \\
 \nabla\phi \in L^2(\Omega_f^g)
\end{array}
\right.
\end{equation}
\end{definition}

\begin{definition}[Dirichlet-to-Neumann operator, star-shaped domain]\label{star_DN}
Suppose $\Omega_\eta^s:=\{0\}\cup\{e^{y + i\alpha}: y < \eta(\alpha)\}$, where $\eta\in \lip(\T)$. We define the star-shaped domain Dirichlet-to-Neumann operator $G_s(\eta)f$ as 
\begin{equation}
(G_s(\eta)f)(\alpha)=N_s\cdot\nabla\phi(e^{\eta(\alpha)+i\alpha}),
\end{equation}
where $N_s(\alpha) := -i(\eta'(\alpha) + i)e^{\eta(\alpha) + i\alpha}$, and $\phi$ solves the boundary value problem
\begin{equation}
\left\{
\begin{array}{ll}
 \Delta\phi = 0 & \mathrm{in}\ \Omega_\eta^s  \\
 \phi(e^{\eta(\alpha)+i\alpha}) = f(\alpha) \\
 \nabla\phi \in L^2(\Omega_\eta^s)
\end{array}
\right.
\end{equation}
\end{definition}

We show that the two definitions are actually equivalent under appropriate regularity assumptions. Before we dive into technicalities, we perform the change of variables formally. Since $\phi_g(\alpha,y) = \phi_s(e^y\cos \alpha, e^y\sin \alpha)$, we obtain
\begin{equation}
\nabla \phi_g = (-e^y\sin \alpha\, \partial_\alpha\phi_s + e^y\cos \alpha\, \partial_y\phi_s,\ e^y\cos \alpha\, \partial_\alpha\phi_s + e^y\sin \alpha\, \partial_y \phi_s  )
\end{equation}
\begin{equation}
\begin{aligned}
N_g \cdot \nabla\phi_g = (e^\eta \partial_\alpha\eta\, \sin\alpha\, + e^\eta \cos\alpha )\partial_\alpha\phi_s + (-e^\eta\partial_\alpha\eta\, \cos\alpha + e^\eta\sin\alpha)\partial_y\phi_s.
\end{aligned}
\end{equation}
On the other hand, we have
\begin{equation}
N_s = (e^\eta \cos\alpha + e^\eta\partial_\alpha\eta\, \sin\alpha, \ e^\eta \sin\alpha\, - e^\eta \partial_\alpha \eta\, \cos\alpha)
\end{equation}
\begin{equation}
N_s \cdot \nabla\phi_s = (e^\eta \partial_\alpha\eta\, \sin\alpha\, + e^\eta \cos\alpha )\partial_\alpha\phi_s + (-e^\eta\partial_\alpha\eta\, \cos\alpha + e^\eta\sin\alpha)\partial_y\phi_s.
\end{equation}
It follows that
\begin{equation}
N_g(\alpha)\cdot \nabla\phi_g(\alpha, \eta(\alpha)) = N_s(\alpha)\cdot\nabla\phi_s(e^{\eta(\alpha)}e^{i\alpha}).
\end{equation}

\begin{proposition}\label{prop:DNs_equivalence}
Suppose $\eta\in \lip(\T)$ and $f \in H^{\frac{1}{2}}(\T)$, then
\begin{equation}
G_g(\eta)f = G_s(\eta)f
\end{equation}
are both well-defined in $\dot{H}^{-\frac{1}{2}}(\T)$. 
\end{proposition}

\begin{proof}
Let $\phi_s\in \dot{H}^1(\Omega_\eta^s)$ be the weak harmonic extension of $f$ in $\Omega_\eta^s$, namely $\operatorname{Tr}\phi_s=f$ and
\begin{equation}
\int_{\Omega_\eta^s}\nabla \phi_s\cdot \nabla \varphi\ dA=0
\end{equation}
for all $\varphi\in H^1_0(\Omega_\eta^s)$. We define $E(x,y):=e^{y+ix}$ and ${\phi}_g(x,y):=\phi_s\circ E(x, y)$. We note that, viewing $x\in\T$, $E$ is a bijection $\Omega_\eta^g\to \Omega_\eta^s\setminus\{0\}$. Moreover, suppose $u_s,v_s\in \dot H^1(\Omega_\eta^s)$ and ${u_g}:=u_s\circ E$, ${v_g}:=v_s\circ E$, then writing $r=e^y$ and $\alpha=x$ and regarding $u_s = u_s(r, \alpha)$, we have 
\begin{equation}
\partial_y{u_g}=r\,\partial_r u_s,\quad \partial_x{u_g}=\partial_\alpha u_s,
\end{equation}
and similar equalities hold for $v_g$ and $v_s$. Since
\begin{equation}
dy=\frac{dr}{r},
\qquad
dx=d\alpha,
\qquad
dA=r\,dr\,d\alpha,
\end{equation}
a direct computation gives
\begin{equation}\label{eq:Dirichlet-form-conformal}
\int_{\Omega_\eta^g}\nabla {u_g}\cdot \nabla {v_g}\,dx\,dy
=
\int_{\Omega_\eta^s}\nabla u_s\cdot \nabla v_s\,dA.
\end{equation}
We show that $\phi_g$ is a weak harmonic extension of $f$ in $\Omega_\eta^g$. Let $\chi\in C_c^\infty(\Omega_\eta^g)$ and we define
\begin{equation}
\varphi(z):=
\begin{cases}
\chi(E^{-1}(z)), & z\in \Omega_\eta^s\setminus\{0\},\\
0, & z=0.
\end{cases}
\end{equation}
Since $\chi$ is compactly supported in $\Omega_\eta^g$, $\varphi$ is compactly supported in $\Omega_\eta^s\setminus\{0\}$. Then it follows from \eqref{eq:Dirichlet-form-conformal} that 
\begin{equation}
\int_{\Omega_\eta^g} \nabla\chi\cdot\nabla\phi_g\ dxdy = \int_{\Omega_\eta^s} \nabla\varphi\cdot\nabla\phi_s\ dA =0,
\end{equation}
so $\phi_g$ is weakly harmonic in $\Omega_{\eta}^g$. It also follows from \eqref{eq:Dirichlet-form-conformal} that $\phi_g\in \dot{H}^1(\Omega_\eta^g)$. Since we also have $\tr\ \phi_g = f$, by density of $C_c^\infty(\Omega_\eta^g)$ in $\dot H_0^1(\Omega_\eta^g)$ we get that $\phi_g$ is the weak harmonic extension of $f$ in $\Omega_\eta^g$.

Now let $\psi\in \dot H^{1/2}(\mathbb{T})$, and choose any $\Psi_s\in \dot H^1(\Omega_\eta^s)$ with $\tr\, \Psi_s = \psi$. Set $\Psi_g := \Psi_s \circ E$, we have from \eqref{eq:Dirichlet-form-conformal} that $\Psi_g\in \dot H^1(\Omega_\eta^g)$ and that $\tr\, \Psi_g = \psi$. Now, using the variational definitions of Dirichlet-to-Neumann operators and \eqref{eq:Dirichlet-form-conformal}, we obtain
\begin{equation}
\bigl\langle G_{\mathrm{g}}(\eta)f,\psi\bigr\rangle
=
\int_{\Omega_\eta^g}\nabla \phi_g\cdot \nabla \Psi_g\,dx\,dy
=
\int_{\Omega_\eta^s}\nabla \phi_s\cdot \nabla \Psi_s\,dA
=
\bigl\langle G_{\mathrm{s}}(\eta)f,\psi\bigr\rangle.
\end{equation}
Since $\psi\in \dot H^{1/2}(\mathbb{T})$ is arbitrary, we get
\begin{equation}
G_{\mathrm{g}}(\eta)f=G_{\mathrm{s}}(\eta)f
\qquad
\text{in }\dot H^{-1/2}(\mathbb{T}).
\end{equation}
as desired.
\end{proof}

The relationship uncovered in Proposition \ref{prop:DNs_equivalence} allow us to immediately apply existing results on $G_{\mathrm{g}}$ to $G_{\mathrm{s}}$, such as the following. We go back to our notation in Definition \ref{def:DN_star_shaped}. 

\begin{proposition}[$H^s$ boundedness {\cite[Theorem~3.8]{AlazardBurqZuily2014}, \cite[Proposition~3.7]{nguyen-pausader}}]
\label{DN_boundedness}
Let $s_0 > \frac{3}{2}$ and $\sigma > \frac{1}{2}$. Then there exists a nondecreasing function $\mathcal{F} : \mathbb{R}^+ \to \mathbb{R}^+$ such that
\begin{equation}
\|G(e^\eta)f\|_{H^{\sigma-1}(\mathbb T)}
\le
\mathcal F\bigl(\|\eta\|_{H^{s_0}(\mathbb T)}\bigr)
\left(
\|f\|_{H^\sigma(\mathbb T)}
+
\|\eta\|_{H^\sigma(\mathbb T)}\|f\|_{H^{s_0}(\mathbb T)}
\right).
\end{equation}
for all $\eta, f\in H^{\max\{s_0, \sigma\}}(\T)$
\end{proposition}

\begin{proposition}[Contraction estimates {\cite[Proposition~2.13]{dePoyferreNguyen2017}}]
\label{DN_contraction}
Let $s_0 > \frac{3}{2}$ and $\sigma \in \left[\frac{1}{2}, s_0\right]$.
Then there exists a nondecreasing function
$\mathcal{F} : \mathbb{R}^+ \to \mathbb{R}^+$ such that
\begin{equation}\label{eq:DN_contraction_estimate}
\big\|G(e^{\eta_1})f - G(e^{\eta_2})f\big\|_{H^{\sigma-1}(\mathbb{T})}
\leq
\mathcal{F}\!\left(\big\|(\eta_1,\eta_2)\big\|_{H^{s_0}(\mathbb{T})}\right)
\big\|\eta_1 - \eta_2\big\|_{H^\sigma(\mathbb{T})}
\big\|f\big\|_{H^{s_0}(\mathbb{T})}
\end{equation}
for all $\eta_j, f \in H^{s_0}(\mathbb{T})$.
\end{proposition}

Finally, we have the following proposition that upgrades the equality in Proposition \ref{prop:DNs_equivalence} to a pointwise agreement under sufficient regularity. 
\begin{proposition}[Pointwise equivalence at a \(C^{1,1}\) point]
\label{cor:DNs_pointwise_equivalence}
Suppose $\eta\in \lip(\T)$ and $\eta$ is $C^{1,1}$ at $\alpha_0\in\T$.
Let $h=e^\eta$. Then
\begin{equation}\label{eq:DNs_pointwise_equivalence}
G_g(\eta)\eta(\alpha_0)=G_s(\eta)\eta(\alpha_0),
\end{equation}
where both Dirichlet-to-Neumann operators are classically well-defined.
\end{proposition}
\begin{proof}
Let $N_s, N_g, \phi_s, \phi_g$, and $E$ be as in the proof of Proposition \ref{prop:DNs_equivalence}. Let $z_0 = e^{\eta(\alpha_0)}e^{i\alpha_0}$, $\mathbf x_0:= (\alpha_0, \eta(\alpha_0))$, and set 
\[p_s(z)=\phi_s(z)-\log|z|, \qquad p_g(x,y)=\phi_g(x, y) - y.\] 
Since $\eta$ is $C^{1,1}$ at $\alpha_0$, Lemma~11.17 of
\cite{CaffarelliSalsa2005} implies that the normal derivatives $\partial_{N_s}p_s(z_0)$ and
$\partial_{N_g}p_g(\mathbf{x}_0)$ are both classically well-defined. It follows that $\partial_{N_s}\phi_s(z_0)$ and $\partial_{N_g}\phi_g(\mathbf{x}_0)$ are classically well-defined. Direct calculation using the identity $\phi_s\circ E = \phi_g$ gives 
\begin{equation}
\partial_{N_g}\phi_g(\mathbf{x}_0) = \partial_{N_s}\phi_s(z_0)
\end{equation}
and thus $G_g(\eta)\eta(\alpha_0)=G_s(\eta)\eta(\alpha_0)$ as desired.
\end{proof}

\section{Analysis Preparations}
\subsection{Useful Inequalities}\label{subsec:useful_ineq}
We collect useful inequalities in this subsection. First of all, we have the classical Gronwall's inequality.
\begin{lemma}[Gronwall's inequality]\label{Gronwall}
Suppose $u:[0,T]\to \R$ is differentiable and satisfies
\begin{equation}
u'(t) \leq a(t)u(t) + b(t)\quad \mathrm{for\ all}\ 0\leq t\leq T
\end{equation}
where $a$ and $b$ are integrable on $[0,T]$. Define the integrating factor
\begin{equation}
\mu(t) := \exp\bigg(-\int_0^t a(s)\ ds\bigg)
\end{equation}
Then we have the inequality
\begin{equation}
u(t)\leq \exp\bigg(\int_0^t a(s)\ ds\bigg)\bigg[u(0) + \int_0^t b(s)\mu(s)\ ds\bigg]
\end{equation}
\end{lemma}
\begin{proof}
By definition of \(\mu\), we have $\mu'(t)=-a(t)\mu(t)$. Then
\begin{equation}\label{eq:gronwall_product_derivative}
\mu(t)u'(t)
\le
a(t)\mu(t)u(t)+\mu(t)b(t),
\end{equation}
and hence
\begin{equation}\label{eq:mu_u_derivative}
\frac{d}{dt}\bigl(\mu(t)u(t)\bigr)
=
\mu(t)u'(t)+\mu'(t)u(t)
\le
\mu(t)b(t).
\end{equation}
Integrating \eqref{eq:mu_u_derivative} from \(0\) to \(t\), and using
\(\mu(0)=1\), we have
\begin{equation}\label{eq:integrated_gronwall}
\mu(t)u(t)
\le
u(0)+\int_0^t \mu(s)b(s)\,ds.
\end{equation}
Finally, multiplying \eqref{eq:integrated_gronwall} by \(\mu(t)^{-1}\) yields
\begin{equation}
u(t)
\le
\exp\left(\int_0^t a(s)\,ds\right)
\left[
u(0)+\int_0^t b(s)\mu(s)\,ds
\right].
\end{equation}
\end{proof}

We then collect some results about inequalities in Sobolev spaces. 
\begin{lemma}[Sobolev interpolation]\label{lem:1d_sobolev_interpolation}
Let $f\in W^{1,\infty}(\mathbb T)$ be real-valued. Then
\begin{equation}\label{eq:sobolev_interp_L2_H1_nonzero}
\|f\|_{L^\infty(\mathbb T)}
\le C\|f\|_{L^2(\mathbb T)}^{1/2}
\|\partial_\alpha f\|_{L^2(\mathbb T)}^{1/2}
+C\|f\|_{L^2(\mathbb T)},
\end{equation}
\begin{equation}\label{eq:sobolev_interp_L2_W1inf_nonzero}
\|f\|_{L^\infty(\mathbb T)}
\le C\|f\|_{L^2(\mathbb T)}^{2/3}
\|\partial_\alpha f\|_{L^\infty(\mathbb T)}^{1/3}
+C\|f\|_{L^2(\mathbb T)} .
\end{equation}
\end{lemma}

\begin{proof}
Both inequalities are special cases of Gagliardo-Nirenberg inequalities, but we provide a self-contained elementary proof here. Let $\overline f:=|\mathbb T|^{-1}\int_{\mathbb T}f\,d\alpha$ and
$g:=f-\overline f$. Then $\int_{\mathbb T}g\,d\alpha=0$,
$\partial_\alpha g=\partial_\alpha f$, and
\begin{equation}\label{eq:mean_subtraction_bounds}
\|g\|_{L^2}\le C\|f\|_{L^2},
\qquad
|\overline f|\le C\|f\|_{L^2}.
\end{equation}
It is therefore enough to prove the homogeneous estimates for $g$. Since $g$ has mean zero, there exists $\alpha_*\in\mathbb T$ with
$g(\alpha_*)=0$. Hence, for any $\alpha\in\mathbb T$,
\begin{equation}\label{eq:mean_zero_fundamental_identity}
g(\alpha)^2
=
2\int_{\alpha_*}^{\alpha}g(\beta)\partial_\beta g(\beta)\,d\beta .
\end{equation}
By Cauchy--Schwarz we have $|g(\alpha)|^2 \le 2\|g\|_{L^2} \|\partial_\alpha g\|_{L^2}$,  and taking the supremum gives
\begin{equation}\label{eq:mean_zero_first_interp}
\|g\|_{L^\infty}
\le
C\|g\|_{L^2}^{1/2}
\|\partial_\alpha g\|_{L^2}^{1/2}.
\end{equation}
For the second inequality, we may assume $\|\partial_\alpha g\|_{L^\infty} > 0$. We write $A_rg(\alpha) := \frac{1}{r}\int_{\alpha-r/2}^{\alpha + r/2} g(\beta)\ d\beta$. For small $r > 0$, we have the pointwise bound
\begin{equation}
|g(\alpha)| \leq |A_rg(\alpha)| + |g(\alpha) - A_rg(\alpha)|.
\end{equation}
By Cauchy-Schwarz, the first term has the bound $|A_rg(\alpha)| \leq r^{-1/2}\|g\|_{L^2}$; using mean-value theorem, the second term can be bounded as $|g(\alpha) - A_rg(\alpha)|\leq r\|\partial_\alpha g\|_{L^\infty}$. It follows that
\begin{equation}
\|g\|_{L^\infty}\leq r^{-1/2}\|g\|_{L^2} + r\|\partial_\alpha g\|_{L^\infty}
\end{equation}
Choosing 
\[r = \min\bigg\{\pi, \bigg(\frac{\|g\|_{L^2}}{\|\partial_\alpha g\|_{L^\infty}}\bigg)^\frac{2}{3}\bigg\} \]
we obtain
\begin{equation}
\|g\|_{L^\infty}
\le C\|g\|_{L^2}^{2/3}
\|\partial_\alpha g\|_{L^\infty}^{1/3}
+C\|g\|_{L^2(\mathbb T)}
\end{equation}
as desired. Now the proof is complete.
\end{proof}

\begin{lemma}[One-dimensional Morrey inequality]\label{lem:morrey_general_T}
Let \(k\in\mathbb N_0\) and \(1<p\le \infty\). If
\(f\in W^{k+1,p}(\mathbb T)\), then \(\partial_\alpha^k f\) admits a
Hölder-continuous representative and
\begin{equation}\label{eq:morrey_general_seminorm}
[\partial_\alpha^k f]_{C^{0,1-\frac1p}(\mathbb T)}
\le
\|\partial_\alpha^{k+1}f\|_{L^p(\mathbb T)}.
\end{equation}
\end{lemma}

\begin{proof}
Let \(\alpha,\beta\in\mathbb T\), and let \(I_{\alpha,\beta}\) be the shorter
arc joining \(\beta\) to \(\alpha\). Then
\begin{equation}\label{eq:morrey_general_ftc}
\partial_\alpha^k f(\alpha)-\partial_\alpha^k f(\beta)
=
\int_{I_{\alpha,\beta}}\partial_\alpha^{k+1}f(\gamma)\,d\gamma .
\end{equation}
By Hölder's inequality,
\begin{equation}\label{eq:morrey_general_holder}
|\partial_\alpha^k f(\alpha)-\partial_\alpha^k f(\beta)|
\le
|I_{\alpha,\beta}|^{1-\frac1p}
\|\partial_\alpha^{k+1}f\|_{L^p(\mathbb T)}
=
|\alpha-\beta|_{\mathbb T}^{1-\frac1p}
\|\partial_\alpha^{k+1}f\|_{L^p(\mathbb T)}.
\end{equation}
which proves \eqref{eq:morrey_general_seminorm}. 
\end{proof}

\begin{lemma}[Sobolev product inequalities \cite{BahouriCheminDanchin2011}]\label{Sobolev_prod}
Let \(s\ge 0\). Then there exists \(C_s>0\) such that
\begin{equation}\label{eq:Hs_Linfty_product_T}
\|uv\|_{H^s(\mathbb T)}
\le
C_s\Big(
\|u\|_{L^\infty(\mathbb T)}\|v\|_{H^s(\mathbb T)}
+
\|v\|_{L^\infty(\mathbb T)}\|u\|_{H^s(\mathbb T)}
\Big).
\end{equation}
In particular, if \(s>\frac12\), then \(H^s(\mathbb T)\) is an algebra:
\begin{equation}\label{eq:Hs_algebra_T}
\|uv\|_{H^s(\mathbb T)}
\le
C_s\|u\|_{H^s(\mathbb T)}\|v\|_{H^s(\mathbb T)}.
\end{equation}
\end{lemma}

\begin{lemma}[Fractional chain rule]\label{fractional_chain_rule}
Suppose $s > 0$ and $f$ takes values in $[a, b]$. Here $-\infty \leq a\leq b\leq \infty$. Let $\phi$ be $C^1$. Then 
\begin{equation}
\|\phi(f)\|_{\dot{H}^s} \leq \sup_{y\in [a, b]}|\phi'(y)|\cdot \|f\|_{\dot{H}^s(\T)}
\end{equation}
\end{lemma}
\begin{proof}
Using the finite difference expression of Sobolev norms, we have
\begin{equation}
\begin{split}
\|\phi(f)\|_{\dot{H}^s}^2 &= \iint \frac{|\phi(f(\alpha')) - \phi(f(\beta'))|^2}{|\sin^2(\frac{\alpha'-\beta'}{2})|^{1+2s}}\ d\alpha'd\beta' \\
&\leq \sup_{y\in [a, b]}|\phi'(y)|^2\iint \frac{|f(\alpha') - f(\beta')|^2}{|\sin^2(\frac{\alpha'-\beta'}{2})|^{1+2s}}\ d\alpha'd\beta' \\
&= \sup_{y\in [a, b]}|\phi'(y)|^2\cdot \|f\|_{\dot{H}^s(\T)}^2
\end{split}
\end{equation}
as desired. 
\end{proof}

\subsection{Boundedness of Cauchy integral operators}\label{subsec:bdd_cauchy} We introduce some classical harmonic analysis results on Cauchy integral operators. 
\begin{lemma}\label{lem:log_radial_chord_arc}
Let $\eta\in W^{1,\infty}(\mathbb T)$ and set $z(\alpha):=e^{\eta(\alpha)}e^{i\alpha}$. Denote $M:=\|\eta\|_{L^\infty(\mathbb T)}$ and $L:=\|\eta_\alpha\|_{L^\infty(\mathbb T)}$. Then for all $\alpha,\beta\in\mathbb T$, we have
\begin{equation}\label{eq:z-bi-lip}
\frac{2}{\pi}e^{-M}|\alpha-\beta|_{\mathbb T}\leq |z(\alpha)-z(\beta)|
\le
e^M\sqrt{1+L^2}\,|\alpha-\beta|_{\mathbb T}.
\end{equation}
In particular, the chord-arc constant of $z(\T)$ is bounded by a
constant depending only on $\|\eta\|_{W^{1,\infty}}$.
\end{lemma}

\begin{proof}
Let $\delta:=|\alpha-\beta|_{\mathbb T}\in[0,\pi]$. Direct calculations give
\begin{equation}
|z(\alpha)-z(\beta)|^2
=
e^{\eta(\alpha)+\eta(\beta)}
\bigl(2\cosh(\eta(\alpha)-\eta(\beta))-2\cos(\alpha-\beta)\bigr).
\end{equation}
and thus
\begin{equation}
|z(\alpha)-z(\beta)|^2
\ge
e^{-2M}4\sin^2\frac{\delta}{2}
\ge
e^{-2M}\frac{4}{\pi^2}\delta^2,  
\end{equation}
which gives the lower bound. The upper bound immediately follows from
\begin{equation}
|z_\alpha|\le e^M\sqrt{1+L^2}.    
\end{equation}
\end{proof}

The following theorem is classical. See, for example, \cites{Calderon1977Cauchy,CMM1982Cauchy,David1984Cauchy}.
\begin{theorem}[$L^2$ boundedness of Cauchy integral operator]\label{lem:cauchy_chord_arc_l2}
Let $\Gamma\subset\mathbb C$ be a chord-arc curve. Define the Cauchy singular
integral
\[
\mathcal C_\Gamma F(z)
:=
\pv \int_\Gamma \frac{F(\zeta)}{\zeta-z}\,d\zeta,
\quad z\in\Gamma.
\]
Then $\mathcal C_\Gamma$ extends to a bounded operator on $L^2(\Gamma)$ with 
\begin{equation}
\|\mathcal C_\Gamma F\|_{L^2(\Gamma)}
\le L_\Gamma \|F\|_{L^2(\Gamma)}.
\end{equation}
Moreover, if $\Gamma$ is parametrized by $z:\T\to \C$, then the constant \(L_\Gamma\) depends only on the chord-arc constant of $z$.
\end{theorem}

Now we have the following lemma on Cauchy integral operator on star-shaped domains.  

\begin{lemma}\label{lem:cauchy_star_shaped_l2}
Let $\eta\in W^{1,\infty}(\mathbb T)$ and $z(\alpha)=e^{\eta(\alpha)}e^{i\alpha}$. Define $\Gamma:= z(\T)$ and
\[T[\eta]\theta(\alpha)
:=
\pv\int_{\mathbb T}
\frac{\theta(\beta)}{z(\alpha)-z(\beta)}\,d\beta .
\]
Then $T[\eta]$ is bounded on $L^2(\mathbb T)$, and
\begin{equation}\label{eq:cauchy-star-shaped-bound}
\|T[\eta]\theta\|_{L^2(\mathbb T)}
\le
\cF\bigl(\|\eta\|_{W^{1,\infty}}\bigr)
\|\theta\|_{L^2(\mathbb T)}.
\end{equation}
\end{lemma}

\begin{proof}
We have from Lemma \ref{lem:cauchy_chord_arc_l2} and Lemma \ref{lem:log_radial_chord_arc} that
\begin{equation}\label{eq:cauchy-standard-bound}
\begin{aligned}
\|\mathcal C_\Gamma F\|_{L^2(\Gamma)}
& \le
L_\Gamma \|F\|_{L^2(\Gamma)}\\
& \leq \cF\bigl(\|\eta\|_{W^{1,\infty}}\bigr)\|F\|_{L^2(\Gamma)}
\end{aligned}
\end{equation}
Now, we write
\begin{equation}
T[\eta]\theta(\alpha)
=
-\pv\int_\Gamma
\frac{F(\zeta)}{\zeta-z(\alpha)}\,d\zeta
\quad \text{where} \ F(z(\beta)):=\frac{\theta(\beta)}{z_\beta(\beta)},
\end{equation}
and note from \eqref{eq:z-bi-lip} that
\begin{equation}
\|F\|_{L^2(\Gamma)}
\le
C\bigl(\|\eta\|_{W^{1,\infty}}\bigr)\|\theta\|_{L^2(\mathbb T)},
\end{equation}
Then we have
\begin{equation}
\begin{aligned}
\|T[\eta]\theta\|_{L^2(\mathbb T)}
&\le
\cF\bigl(\|\eta\|_{W^{1,\infty}}\bigr)
\|\mathcal C_\Gamma F\|_{L^2(\Gamma)}\\
&\leq \cF\bigl(\|\eta\|_{W^{1,\infty}}\bigr)\|F\|_{L^2(\Gamma)}\\
&\leq \cF\bigl(\|\eta\|_{W^{1,\infty}}\bigr)\|\theta\|_{L^2(\mathbb T)}
\end{aligned}
\end{equation}
as desired.
\end{proof}

\subsection{Functional analysis} We have the following Aubin-Lions compactness lemma.
\begin{lemma}[Aubin-Lions \cites{Aubin1963Compactness, Lions1969QuelquesMethodes, BedrossianVicol2022}]\label{aubin-lions}
Suppose $B_0\subset B_1\subset B_2$ are Banach spaces. We assume that the embeddings are all continuous and that $B_0\subset B_1$ compactly. Let $p, r\in (1,\infty]\times (1,\infty)$. For $T>0$, we define
\[E_{p,r} := \{f\in L^p([0,T]; B_0): \partial_t f\in L^r([0,T]; B_2)\} \]
Then
\begin{enumerate}
    \item if $p<\infty$, the embedding $E_{p,r}\subset L^p([0,T]; B_1)$ is compact;
    \item if $p = \infty$, the embedding $E_{p,r}\subset C([0,T]; B_1)$ is compact.
\end{enumerate}
\end{lemma}

\newpage
\bibliography{refs_v5}

@book{Darcy1856,
  author    = {Darcy, Henry},
  title     = {Les fontaines publiques de la ville de {Dijon}},
  publisher = {Dalmont},
  address   = {Paris},
  year      = {1856}
}

@article{HeleShaw1898,
  author  = {Hele-Shaw, Henry S.},
  title   = {The flow of water},
  journal = {Nature},
  volume  = {58},
  pages   = {34--36},
  year    = {1898},
  doi     = {10.1038/058034a0}
}

@book{Muskat1937,
  author    = {Muskat, Morris},
  title     = {The Flow of Homogeneous Fluids Through Porous Media},
  publisher = {McGraw-Hill},
  address   = {New York},
  year      = {1937}
}

@article{SaffmanTaylor1958,
  author  = {Saffman, Philip G. and Taylor, Geoffrey I.},
  title   = {The penetration of a fluid into a porous medium or {Hele-Shaw} cell containing a more viscous liquid},
  journal = {Proceedings of the Royal Society of London. Series A. Mathematical and Physical Sciences},
  volume  = {245},
  number  = {1242},
  pages   = {312--329},
  year    = {1958},
  doi     = {10.1098/rspa.1958.0085}
}

@article{Aubin1963Compactness,
  author  = {Aubin, Jean-Pierre},
  title   = {Un th{\'e}or{\`e}me de compacit{\'e}},
  journal = {Comptes Rendus de l'Acad{\'e}mie des Sciences de Paris},
  volume  = {256},
  pages   = {5042--5044},
  year    = {1963}
}

@book{Lions1969QuelquesMethodes,
  author    = {Lions, Jacques-Louis},
  title     = {Quelques m{\'e}thodes de r{\'e}solution des probl{\`e}mes aux limites non lin{\'e}aires},
  publisher = {Dunod and Gauthier-Villars},
  address   = {Paris},
  year      = {1969}
}

@article{Baiocchi1972,
  author  = {Baiocchi, Claudio},
  title   = {Su un problema di frontiera libera connesso a questioni di idraulica},
  journal = {Annali di Matematica Pura ed Applicata},
  series  = {4},
  volume  = {92},
  pages   = {107--127},
  year    = {1972},
  doi     = {10.1007/BF02417940}
}

@article{Richardson1972,
  author  = {Richardson, S.},
  title   = {{Hele-Shaw} flows with a free boundary produced by the injection of fluid into a narrow channel},
  journal = {Journal of Fluid Mechanics},
  volume  = {56},
  number  = {4},
  pages   = {609--618},
  year    = {1972},
  doi     = {10.1017/S0022112072002624}
}

@article{Caffarelli1976,
  author  = {Caffarelli, Luis A.},
  title   = {The smoothness of the free surface in a filtration problem},
  journal = {Archive for Rational Mechanics and Analysis},
  volume  = {63},
  number  = {1},
  pages   = {77--86},
  year    = {1976},
  doi     = {10.1007/BF00251578}
}

@article{Caffarelli1977,
  author  = {Caffarelli, Luis A.},
  title   = {The regularity of free boundaries in higher dimensions},
  journal = {Acta Mathematica},
  volume  = {139},
  pages   = {155--184},
  year    = {1977},
  doi     = {10.1007/BF02392236}
}

@article{Calderon1977Cauchy,
  author  = {Calder{\'o}n, Alberto P.},
  title   = {{Cauchy} integrals on {Lipschitz} curves and related operators},
  journal = {Proceedings of the National Academy of Sciences of the United States of America},
  volume  = {74},
  number  = {4},
  pages   = {1324--1327},
  year    = {1977},
  doi     = {10.1073/pnas.74.4.1324}
}

@article{FabesJodeitRiviere1978,
  author  = {Fabes, Eugene B. and Jodeit, Manfred, Jr. and Rivi{\`e}re, Nicolas M.},
  title   = {Potential techniques for boundary value problems on {$C^1$}-domains},
  journal = {Acta Mathematica},
  volume  = {141},
  pages   = {165--186},
  year    = {1978},
  doi     = {10.1007/BF02545747}
}

@article{ElliottJanovsky1981,
  author  = {Elliott, Charles M. and Janovsk{\'y}, Vladim{\'i}r},
  title   = {A variational inequality approach to {Hele-Shaw} flow with a moving boundary},
  journal = {Proceedings of the Royal Society of Edinburgh. Section A. Mathematics},
  volume  = {88},
  number  = {1--2},
  pages   = {93--107},
  year    = {1981},
  doi     = {10.1017/S0308210500017300}
}

@book{Friedman1982,
  author    = {Friedman, Avner},
  title     = {Variational Principles and Free-Boundary Problems},
  publisher = {Wiley},
  address   = {New York},
  year      = {1982}
}

@article{CMM1982Cauchy,
  author  = {Coifman, Ronald R. and McIntosh, Alan and Meyer, Yves},
  title   = {L'int{\'e}grale de {Cauchy} d{\'e}finit un op{\'e}rateur born{\'e} sur \(L^2\) pour les courbes lipschitziennes},
  journal = {Annals of Mathematics. Second Series},
  volume  = {116},
  number  = {2},
  pages   = {361--387},
  year    = {1982},
  doi     = {10.2307/2007065}
}

@article{David1984Cauchy,
  author  = {David, Guy},
  title   = {Op{\'e}rateurs int{\'e}graux singuliers sur certaines courbes du plan complexe},
  journal = {Annales Scientifiques de l'{\'E}cole Normale Sup{\'e}rieure},
  series  = {4},
  volume  = {17},
  number  = {1},
  pages   = {157--189},
  year    = {1984},
  doi     = {10.24033/asens.1469}
}

@article{DuchonRobert1984,
  author  = {Duchon, Jean and Robert, Raoul},
  title   = {{\'E}volution d'une interface par capillarit{\'e} et diffusion de volume. {I}. {E}xistence locale en temps},
  journal = {Annales de l'Institut Henri Poincar{\'e}. C, Analyse non lin{\'e}aire},
  volume  = {1},
  number  = {5},
  pages   = {361--378},
  year    = {1984},
  doi     = {10.1016/S0294-1449(16)30418-8}
}

@article{Verchota1984,
  author  = {Verchota, Gregory C.},
  title   = {Layer potentials and regularity for the {Dirichlet} problem for {Laplace}'s equation in {Lipschitz} domains},
  journal = {Journal of Functional Analysis},
  volume  = {59},
  number  = {3},
  pages   = {572--611},
  year    = {1984},
  doi     = {10.1016/0022-1236(84)90066-1}
}

@book{Grisvard1985,
  author    = {Grisvard, Pierre},
  title     = {Elliptic Problems in Nonsmooth Domains},
  publisher = {Pitman},
  address   = {Boston},
  year      = {1985}
}

@incollection{Kenig1986,
  author    = {Kenig, Carlos E.},
  title     = {Elliptic Boundary Value Problems on {Lipschitz} Domains},
  booktitle = {Beijing Lectures in Harmonic Analysis},
  editor    = {Stein, Elias M.},
  series    = {Annals of Mathematics Studies},
  volume    = {112},
  pages     = {131--184},
  publisher = {Princeton University Press},
  address   = {Princeton, NJ},
  year      = {1986}
}

@article{Howison1986,
  author  = {Howison, Sam D.},
  title   = {Cusp development in {Hele-Shaw} flow with a free surface},
  journal = {SIAM Journal on Applied Mathematics},
  volume  = {46},
  number  = {1},
  pages   = {20--26},
  year    = {1986},
  doi     = {10.1137/0146003}
}

@article{Costabel1988,
  author  = {Costabel, Martin},
  title   = {Boundary integral operators on {Lipschitz} domains: elementary results},
  journal = {SIAM Journal on Mathematical Analysis},
  volume  = {19},
  number  = {3},
  pages   = {613--626},
  year    = {1988},
  doi     = {10.1137/0519043}
}

@article{CrandallIshiiLions1992,
  author  = {Crandall, Michael G. and Ishii, Hitoshi and Lions, Pierre-Louis},
  title   = {User's guide to viscosity solutions of second order partial differential equations},
  journal = {Bulletin of the American Mathematical Society},
  volume  = {27},
  number  = {1},
  pages   = {1--67},
  year    = {1992},
  doi     = {10.1090/S0273-0979-1992-00266-5}
}

@article{Chen1993,
  author  = {Chen, Xinfu},
  title   = {The {Hele-Shaw} problem and area-preserving curve-shortening motions},
  journal = {Archive for Rational Mechanics and Analysis},
  volume  = {123},
  number  = {2},
  pages   = {117--151},
  year    = {1993},
  doi     = {10.1007/BF00375152}
}

@article{ConstantinPugh1993,
  author  = {Constantin, Peter and Pugh, Mary},
  title   = {Global solutions for small data to the {Hele-Shaw} problem},
  journal = {Nonlinearity},
  volume  = {6},
  number  = {3},
  pages   = {393--415},
  year    = {1993},
  doi     = {10.1088/0951-7715/6/3/004}
}

@article{Wu1997WaterWaves2D,
  author  = {Wu, Sijue},
  title   = {Well-posedness in {Sobolev} spaces of the full water wave problem in 2-{D}},
  journal = {Inventiones Mathematicae},
  volume  = {130},
  number  = {1},
  pages   = {39--72},
  year    = {1997},
  doi     = {10.1007/s002220050177}
}

@article{EscherSimonett1997SurfaceTension,
  author  = {Escher, Joachim and Simonett, Gieri},
  title   = {Classical solutions for {Hele-Shaw} models with surface tension},
  journal = {Advances in Differential Equations},
  volume  = {2},
  number  = {4},
  pages   = {619--642},
  year    = {1997}
}

@article{EscherSimonett1997Multidimensional,
  author  = {Escher, Joachim and Simonett, Gieri},
  title   = {Classical solutions of multidimensional {Hele-Shaw} models},
  journal = {SIAM Journal on Mathematical Analysis},
  volume  = {28},
  number  = {5},
  pages   = {1028--1047},
  year    = {1997},
  doi     = {10.1137/S0036141095291919}
}

@article{Wu1999WaterWaves3D,
  author  = {Wu, Sijue},
  title   = {Well-posedness in {Sobolev} spaces of the full water wave problem in 3-{D}},
  journal = {Journal of the American Mathematical Society},
  volume  = {12},
  number  = {2},
  pages   = {445--495},
  year    = {1999},
  doi     = {10.1090/S0894-0347-99-00290-8}
}

@book{McLean2000,
  author    = {McLean, William},
  title     = {Strongly Elliptic Systems and Boundary Integral Equations},
  publisher = {Cambridge University Press},
  address   = {Cambridge},
  year      = {2000}
}

@article{Kim2003,
  author  = {Kim, Inwon C.},
  title   = {Uniqueness and existence results on the {Hele--Shaw} and the {Stefan} problems},
  journal = {Archive for Rational Mechanics and Analysis},
  volume  = {168},
  number  = {4},
  pages   = {299--328},
  year    = {2003},
  doi     = {10.1007/s00205-003-0251-z}
}

@article{SiegelCaflischHowison2004,
  author  = {Siegel, Michael and Caflisch, Russel E. and Howison, Sam},
  title   = {Global existence, singular solutions, and ill-posedness for the {Muskat} problem},
  journal = {Communications on Pure and Applied Mathematics},
  volume  = {57},
  number  = {10},
  pages   = {1374--1411},
  year    = {2004},
  doi     = {10.1002/cpa.20040}
}

@book{CaffarelliSalsa2005,
  author    = {Caffarelli, Luis A. and Salsa, Sandro},
  title     = {A Geometric Approach to Free Boundary Problems},
  series    = {Graduate Studies in Mathematics},
  volume    = {68},
  publisher = {American Mathematical Society},
  address   = {Providence, RI},
  year      = {2005},
  doi       = {10.1090/gsm/068}
}

@article{JerisonKim2005,
  author  = {Jerison, David and Kim, Inwon C.},
  title   = {The one-phase {Hele--Shaw} problem with singularities},
  journal = {Journal of Geometric Analysis},
  volume  = {15},
  number  = {4},
  pages   = {641--667},
  year    = {2005},
  doi     = {10.1007/BF02922248}
}

@article{ChoiKim2006,
  author  = {Choi, Sunhi and Kim, Inwon C.},
  title   = {Waiting time phenomena of the {Hele--Shaw} and the {Stefan} problem},
  journal = {Indiana University Mathematics Journal},
  volume  = {55},
  number  = {2},
  pages   = {525--551},
  year    = {2006},
  url     = {https://www.jstor.org/stable/24902363}
}

@article{Kim2006Reg,
  author  = {Kim, Inwon C.},
  title   = {Regularity of the free boundary for the one phase {Hele--Shaw} problem},
  journal = {Journal of Differential Equations},
  volume  = {223},
  number  = {1},
  pages   = {161--184},
  year    = {2006},
  doi     = {10.1016/j.jde.2005.07.003}
}

@article{Kim2006Long,
  author  = {Kim, Inwon C.},
  title   = {Long time regularity of solutions of the {Hele--Shaw} problem},
  journal = {Nonlinear Analysis: Theory, Methods \& Applications},
  volume  = {64},
  number  = {12},
  pages   = {2817--2831},
  year    = {2006},
  doi     = {10.1016/j.na.2005.09.021}
}

@book{GustafssonVasiliev2006,
  author    = {Gustafsson, Bj{\"o}rn and Vasil'ev, Alexander},
  title     = {Conformal and Potential Analysis in {Hele-Shaw} Cells},
  series    = {Advances in Mathematical Fluid Mechanics},
  publisher = {Birkh{\"a}user},
  address   = {Basel},
  year      = {2006},
  doi       = {10.1007/3-7643-7704-8}
}

@article{ChoiJerisonKim2007,
  author  = {Choi, Sunhi and Jerison, David and Kim, Inwon C.},
  title   = {Regularity for the one-phase {Hele--Shaw} problem from a {Lipschitz} initial surface},
  journal = {American Journal of Mathematics},
  volume  = {129},
  number  = {2},
  pages   = {527--582},
  year    = {2007},
  doi     = {10.1353/ajm.2007.0008}
}

@article{ChoiJerisonKim2009,
  author  = {Choi, Sunhi and Jerison, David and Kim, Inwon C.},
  title   = {Local regularization of the one-phase {Hele--Shaw} flow},
  journal = {Indiana University Mathematics Journal},
  volume  = {58},
  number  = {6},
  pages   = {2765--2804},
  year    = {2009},
  doi     = {10.1512/iumj.2009.58.3802}
}

@book{BahouriCheminDanchin2011,
  author    = {Bahouri, Hajer and Chemin, Jean-Yves and Danchin, Rapha{\"e}l},
  title     = {Fourier Analysis and Nonlinear Partial Differential Equations},
  series    = {Grundlehren der mathematischen Wissenschaften},
  volume    = {343},
  publisher = {Springer},
  address   = {Berlin},
  year      = {2011},
  doi       = {10.1007/978-3-642-16830-7}
}

@book{TaylorPDE3,
  author    = {Taylor, Michael E.},
  title     = {Partial Differential Equations {III}: Nonlinear Equations},
  series    = {Applied Mathematical Sciences},
  volume    = {117},
  edition   = {2},
  publisher = {Springer},
  address   = {New York},
  year      = {2011},
  doi       = {10.1007/978-1-4419-7049-7},
  isbn      = {978-1-4419-7048-0}
}

@article{CordobaCordobaGancedo2011,
  author  = {C{\'o}rdoba, Antonio and C{\'o}rdoba, Diego and Gancedo, Francisco},
  title   = {Interface evolution: the {Hele-Shaw} and {Muskat} problems},
  journal = {Annals of Mathematics. Second Series},
  volume  = {173},
  number  = {1},
  pages   = {477--542},
  year    = {2011},
  doi     = {10.4007/annals.2011.173.1.10}
}

@article{CastroCordobaFeffermanGancedo2012,
  author  = {Castro, {\'A}ngel and C{\'o}rdoba, Diego and Fefferman, Charles and Gancedo, Francisco and L{\'o}pez-Fern{\'a}ndez, Mar{\'i}a},
  title   = {{Rayleigh--Taylor} breakdown for the {Muskat} problem with applications to water waves},
  journal = {Annals of Mathematics. Second Series},
  volume  = {175},
  number  = {2},
  pages   = {909--948},
  year    = {2012},
  doi     = {10.4007/annals.2012.175.2.9}
}

@article{CastroCordobaFeffermanGancedo2013,
  author  = {Castro, {\'A}ngel and C{\'o}rdoba, Diego and Fefferman, Charles and Gancedo, Francisco},
  title   = {Breakdown of smoothness for the {Muskat} problem},
  journal = {Archive for Rational Mechanics and Analysis},
  volume  = {208},
  number  = {3},
  pages   = {805--909},
  year    = {2013},
  doi     = {10.1007/s00205-013-0614-3}
}

@article{ConstantinCordobaGancedoStrain2013,
  author  = {Constantin, Peter and C{\'o}rdoba, Diego and Gancedo, Francisco and Strain, Robert M.},
  title   = {On the global existence for the {Muskat} problem},
  journal = {Journal of the European Mathematical Society},
  volume  = {15},
  number  = {1},
  pages   = {201--227},
  year    = {2013},
  doi     = {10.4171/JEMS/358}
}

@article{AlazardBurqZuily2014,
  author  = {Alazard, Thomas and Burq, Nicolas and Zuily, Claude},
  title   = {On the {Cauchy} problem for gravity water waves},
  journal = {Inventiones Mathematicae},
  volume  = {198},
  number  = {1},
  pages   = {71--163},
  year    = {2014},
  doi     = {10.1007/s00222-014-0498-z}
}

@article{ChengCoutandShkoller2014,
  author  = {Cheng, C. H. Arthur and Coutand, Daniel and Shkoller, Steve},
  title   = {Global existence and decay for solutions of the {Hele-Shaw} flow with injection},
  journal = {Interfaces and Free Boundaries},
  volume  = {16},
  number  = {3},
  pages   = {297--338},
  year    = {2014},
  doi     = {10.4171/IFB/321}
}

@article{ConstantinCordobaGancedoRodriguezPiazzaStrain2016,
  author  = {Constantin, Peter and C{\'o}rdoba, Diego and Gancedo, Francisco and Rodr{\'i}guez-Piazza, Luis and Strain, Robert M.},
  title   = {On the {Muskat} problem: global in time results in 2{D} and 3{D}},
  journal = {American Journal of Mathematics},
  volume  = {138},
  number  = {6},
  pages   = {1455--1494},
  year    = {2016},
  doi     = {10.1353/ajm.2016.0042}
}

@article{ConstantinGancedoShvydkoyVicol2017,
  author  = {Constantin, Peter and Gancedo, Francisco and Shvydkoy, Roman and Vicol, Vlad},
  title   = {Global regularity for 2{D} {Muskat} equations with finite slope},
  journal = {Annales de l'Institut Henri Poincar{\'e} C, Analyse non lin{\'e}aire},
  volume  = {34},
  number  = {4},
  pages   = {1041--1074},
  year    = {2017},
  doi     = {10.1016/j.anihpc.2016.09.001}
}

@article{dePoyferreNguyen2017,
  author  = {de Poyferr{\'e}, Thibault and Nguyen, Quang-Huy},
  title   = {A paradifferential reduction for the gravity-capillary waves system at low regularity and applications},
  journal = {Bulletin de la Soci{\'e}t{\'e} Math{\'e}matique de France},
  volume  = {145},
  number  = {4},
  pages   = {643--710},
  year    = {2017},
  doi     = {10.24033/bsmf.2750}
}

@article{Cameron2019,
  author  = {Cameron, Stephen},
  title   = {Global well-posedness for the two-dimensional {Muskat} problem with slope less than 1},
  journal = {Analysis \& PDE},
  volume  = {12},
  number  = {4},
  pages   = {997--1022},
  year    = {2019},
  doi     = {10.2140/apde.2019.12.997}
}

@article{ChangLaraGuillenSchwab2019,
  author  = {Chang-Lara, H{\'e}ctor A. and Guillen, Nestor and Schwab, Russell W.},
  title   = {Some free boundary problems recast as nonlocal parabolic equations},
  journal = {Nonlinear Analysis},
  volume  = {189},
  pages   = {111538},
  year    = {2019},
  doi     = {10.1016/j.na.2019.05.019}
}

@article{GancedoGarciaJuarezPatelStrain2019,
  author  = {Gancedo, Francisco and Garc{\'i}a-Ju{\'a}rez, Eduardo and Patel, Neel and Strain, Robert M.},
  title   = {On the {Muskat} problem with viscosity jump: global in time results},
  journal = {Advances in Mathematics},
  volume  = {345},
  pages   = {552--597},
  year    = {2019},
  doi     = {10.1016/j.aim.2019.01.017}
}

@article{Matioc2019,
  author  = {Matioc, Bogdan-Vasile},
  title   = {The {Muskat} problem in two dimensions: equivalence of formulations, well-posedness, and regularity results},
  journal = {Analysis \& PDE},
  volume  = {12},
  number  = {2},
  pages   = {281--332},
  year    = {2019},
  doi     = {10.2140/apde.2019.12.281}
}

@article{Wu2019WaterWaves,
  author  = {Wu, Sijue},
  title   = {Wellposedness of the 2{D} full water wave equation in a regime that allows for non-{$C^1$} interfaces},
  journal = {Inventiones Mathematicae},
  volume  = {217},
  number  = {1},
  pages   = {241--375},
  year    = {2019},
  doi     = {10.1007/s00222-019-00867-4}
}

@article{AlazardLazar2020,
  author  = {Alazard, Thomas and Lazar, Omar},
  title   = {Paralinearization of the {Muskat} equation and application to the {Cauchy} problem},
  journal = {Archive for Rational Mechanics and Analysis},
  volume  = {237},
  number  = {2},
  pages   = {545--583},
  year    = {2020},
  doi     = {10.1007/s00205-020-01514-6}
}

@article{nguyen-pausader,
  author  = {Nguyen, Huy Q. and Pausader, Beno{\^{\i}}t},
  title   = {A paradifferential approach for well-posedness of the {Muskat} problem},
  journal = {Archive for Rational Mechanics and Analysis},
  volume  = {237},
  number  = {1},
  pages   = {35--100},
  year    = {2020},
  doi     = {10.1007/s00205-020-01494-7}
}

@article{nguyen20,
  author  = {Nguyen, Huy Q.},
  title   = {On well-posedness of the {Muskat} problem with surface tension},
  journal = {Advances in Mathematics},
  volume  = {374},
  pages   = {107344},
  year    = {2020},
  doi     = {10.1016/j.aim.2020.107344}
}

@article{flynn-nguyen,
  author  = {Flynn, Patrick T. and Nguyen, Huy Q.},
  title   = {The vanishing surface tension limit of the {Muskat} problem},
  journal = {Communications in Mathematical Physics},
  volume  = {382},
  number  = {2},
  pages   = {1205--1241},
  year    = {2021},
  doi     = {10.1007/s00220-021-03980-9}
}

@article{AlazardNguyen2021Critical,
  author  = {Alazard, Thomas and Nguyen, Quoc-Hung},
  title   = {On the {Cauchy} problem for the {Muskat} equation. {II}: critical initial data},
  journal = {Annals of PDE},
  volume  = {7},
  number  = {1},
  pages   = {Paper No. 7, 25 pp.},
  year    = {2021},
  doi     = {10.1007/s40818-021-00099-x}
}

@article{CordobaLazar2021MuskatH32,
  author  = {C{\'o}rdoba, Diego and Lazar, Omar},
  title   = {Global well-posedness for the {2D} stable {Muskat} problem in {$H^{3/2}$}},
  journal = {Annales Scientifiques de l'{\'E}cole Normale Sup{\'e}rieure},
  series  = {4},
  volume  = {54},
  number  = {5},
  pages   = {1315--1351},
  year    = {2021},
  doi     = {10.24033/asens.2483}
}

@article{AlazardNguyen2022Muskat3D,
  author  = {Alazard, Thomas and Nguyen, Quoc-Hung},
  title   = {Quasilinearization of the {3D} {Muskat} equation, and applications to the critical {Cauchy} problem},
  journal = {Advances in Mathematics},
  volume  = {399},
  pages   = {Paper No. 108278, 52 pp.},
  year    = {2022},
  doi     = {10.1016/j.aim.2022.108278}
}

@book{BedrossianVicol2022,
  author    = {Bedrossian, Jacob and Vicol, Vlad},
  title     = {The Mathematical Analysis of the Incompressible {Euler} and {Navier--Stokes} Equations: An Introduction},
  series    = {Graduate Studies in Mathematics},
  volume    = {225},
  publisher = {American Mathematical Society},
  address   = {Providence, RI},
  year      = {2022},
  pages     = {218},
  isbn      = {978-1-4704-7178-1}
}

@article{ChenNguyenXu2022C1Muskat,
  author  = {Chen, Ke and Nguyen, Quoc-Hung and Xu, Yiran},
  title   = {The {Muskat} problem with {$C^1$} data},
  journal = {Transactions of the American Mathematical Society},
  volume  = {375},
  number  = {5},
  pages   = {3039--3060},
  year    = {2022},
  doi     = {10.1090/tran/8559}
}

@article{GancedoLazar2022Muskat3D,
  author  = {Gancedo, Francisco and Lazar, Omar},
  title   = {Global well-posedness for the three dimensional {Muskat} problem in the critical {Sobolev} space},
  journal = {Archive for Rational Mechanics and Analysis},
  volume  = {246},
  number  = {1},
  pages   = {141--207},
  year    = {2022},
  doi     = {10.1007/s00205-022-01808-x}
}

@article{Nguyen2022BesovMuskat,
  author  = {Nguyen, Huy Q.},
  title   = {Global solutions for the {Muskat} problem in the scaling invariant {Besov} space {$\dot B^1_{\infty,1}$}},
  journal = {Advances in Mathematics},
  volume  = {394},
  pages   = {Paper No. 108122, 28 pp.},
  year    = {2022},
  doi     = {10.1016/j.aim.2021.108122}
}

@article{agrawal-patel-wu,
  author  = {Agrawal, Siddhant and Patel, Neel and Wu, Sijue},
  title   = {Rigidity of acute angled corners for one phase {Muskat} interfaces},
  journal = {Advances in Mathematics},
  volume  = {412},
  pages   = {Paper No. 108801},
  year    = {2023},
  doi     = {10.1016/j.aim.2022.108801}
}

@article{AgrawalAlazard2023Rellich,
  author  = {Agrawal, Siddhant and Alazard, Thomas},
  title   = {Refined {Rellich} boundary inequalities for the derivatives of a harmonic function},
  journal = {Proceedings of the American Mathematical Society},
  volume  = {151},
  number  = {5},
  pages   = {2103--2113},
  year    = {2023},
  doi     = {10.1090/proc/16277}
}

@article{AlazardNguyen2023Endpoint,
  author  = {Alazard, Thomas and Nguyen, Quoc-Hung},
  title   = {Endpoint {Sobolev} theory for the {Muskat} equation},
  journal = {Communications in Mathematical Physics},
  volume  = {397},
  number  = {3},
  pages   = {1043--1102},
  year    = {2023},
  doi     = {10.1007/s00220-022-04514-7}
}

@article{Dong-Gancedo-Nguyen-23,
  author  = {Dong, Hongjie and Gancedo, Francisco and Nguyen, Huy Q.},
  title   = {Global well-posedness for the one-phase {Muskat} problem},
  journal = {Communications on Pure and Applied Mathematics},
  volume  = {76},
  number  = {12},
  pages   = {3912--3967},
  year    = {2023},
  doi     = {10.1002/cpa.22124}
}

@article{GarciaJuarezGomezSerranoHaziotPausader2024,
  author  = {Garc{\'i}a-Ju{\'a}rez, Eduardo and G{\'o}mez-Serrano, Javier and Haziot, Susanna V. and Pausader, Beno{\^{\i}}t},
  title   = {Desingularization of small moving corners for the {Muskat} equation},
  journal = {Annals of PDE},
  volume  = {10},
  number  = {2},
  pages   = {17},
  year    = {2024},
  doi     = {10.1007/s40818-024-00175-y}
}

@article{KimZhang2024SourceDrift,
  author  = {Kim, Inwon and Zhang, Yuming Paul},
  title   = {Regularity of {Hele-Shaw} flow with source and drift},
  journal = {Annals of PDE},
  volume  = {10},
  number  = {2},
  pages   = {20},
  year    = {2024},
  doi     = {10.1007/s40818-024-00184-x}
}

@incollection{Alazard2024Paralinearization,
  author    = {Alazard, Thomas},
  title     = {Paralinearization of free boundary problems in fluid dynamics},
  booktitle = {Partial Differential Equations: Waves, Nonlinearities and Nonlocalities},
  editor    = {Ehrnstr{\"o}m, Mats and Holden, Helge and Jakobsen, Espen R.},
  series    = {Abel Symposia},
  volume    = {18},
  pages     = {1--31},
  publisher = {Springer Nature Switzerland},
  address   = {Cham},
  year      = {2025},
  doi       = {10.1007/978-3-031-91282-5_1}
}

@article{Dong-Gancedo-Nguyen-23-2,
  author  = {Dong, Hongjie and Gancedo, Francisco and Nguyen, Huy Q.},
  title   = {Global well-posedness for the one-phase {Muskat} problem in 3{D}},
  journal = {Preprint, arXiv:2308.14230 [math.AP]},
  year    = {2023},
  url     = {https://arxiv.org/abs/2308.14230}
}

@article{AlazardKoch2025HeleShaw,
  author  = {Alazard, Thomas and Koch, Herbert},
  title   = {The {Hele-Shaw} semi-flow},
  journal = {Preprint, arXiv:2312.13678 [math.AP]},
  year    = {2023},
  url     = {https://arxiv.org/abs/2312.13678}
}

@article{Lazar2024SurfaceTensionMuskat,
  author  = {Lazar, Omar},
  title   = {Global well-posedness of arbitrarily large {Lipschitz} solutions for the {Muskat} problem with surface tension},
  journal = {Preprint, arXiv:2407.09444 [math.AP]},
  year    = {2024},
  url     = {https://arxiv.org/abs/2407.09444}
}

@article{SchwabTuTuranova2024,
  author  = {Schwab, Russell W. and Tu, Son and Turanova, Olga},
  title   = {Well-posedness for viscosity solutions of the one-phase {Muskat} problem in all dimensions},
  journal = {Preprint, arXiv:2404.10972 [math.AP]},
  year    = {2024},
  url     = {https://arxiv.org/abs/2404.10972}
}

@article{MatiocWalker2025Capillary,
  author  = {Matioc, Bogdan-Vasile and Walker, Christoph},
  title   = {A potential theory approach to the capillarity-driven {Hele-Shaw} problem},
  journal = {Preprint, arXiv:2508.15491 [math.AP]},
  year    = {2025},
  url     = {https://arxiv.org/abs/2508.15491}
}

@article{DongKwon2026SurfaceTensionMuskat,
  author  = {Dong, Hongjie and Kwon, Hyunwoo},
  title   = {Global well-posedness of the one-phase {Muskat} problem with surface tension},
  journal = {Preprint, arXiv:2604.06545 [math.AP]},
  year    = {2026},
  url     = {https://arxiv.org/abs/2604.06545}
}

@article{Zhang2026C1,
  author  = {Zhang, Yuming Paul},
  title   = {{$C^1$}-regularity of the free boundary for {Hele-Shaw} flow with source and drift},
  journal = {Preprint, arXiv:2604.26906 [math.AP]},
  year    = {2026},
  url     = {https://arxiv.org/abs/2604.26906}
}

@article{AlazardMeunierSmets2020,
  author  = {Alazard, Thomas and Meunier, Nicolas and Smets, Didier},
  title   = {Lyapunov functions, identities and the Cauchy problem for the Hele-Shaw equation},
  journal = {Communications in Mathematical Physics},
  volume  = {377},
  number  = {2},
  pages   = {1421--1459},
  year    = {2020},
}

@article{HowisonKing1989,
  author  = {Howison, S. D. and King, J. R.},
  title   = {Explicit solution to six free-boundary problems in fluid flow and diffusion},
  journal = {IMA Journal of Applied Mathematics},
  volume  = {42},
  pages   = {155--175},
  year    = {1989},
}

@article{KingLaceyVazquez1995,
  author  = {King, J. R. and Lacey, A. A. and V{\'a}zquez, J. L.},
  title   = {Persistence of corners in free boundaries in {Hele-Shaw} flow},
  journal = {European Journal of Applied Mathematics},
  volume  = {6},
  number  = {5},
  pages   = {455--490},
  year    = {1995},
  doi     = {10.1017/S0956792500001984},
}

\end{document}